\documentclass[11pt,reqno]{amsart}

\usepackage{color}
\usepackage{amsmath, amsthm, amssymb}
\usepackage{amsfonts}
\usepackage{mathtools}
\usepackage[ansinew]{inputenc}
\usepackage[dvips]{epsfig}
\usepackage{graphicx}
\usepackage[english]{babel}
\usepackage{enumerate}
\usepackage{hyperref}
\usepackage{mathrsfs}
\theoremstyle{plain}
\newtheorem{thm}{Theorem}[section]
\newtheorem{cor}[thm]{Corollary}
\newtheorem{lem}[thm]{Lemma}
\newtheorem{prop}[thm]{Proposition}

\newtheorem{conjecture}[thm]{Conjecture}

\theoremstyle{definition}
\newtheorem{defi}[thm]{Definition}

\theoremstyle{remark}
\newtheorem{rem}[thm]{Remark}

\numberwithin{equation}{section}

\newcommand{\fls}{(-\Delta)^s}
\newcommand{\R}{\mathbb{R}}

\newcommand{\eps}{\varepsilon}

\newcommand{\U}{\mathcal{U}}
\newcommand{\be}{{\boldsymbol{e}}}
\newcommand{\Co}{{\mathcal{C}}}
\newcommand{\s}{\zeta}

\newcommand{\average}{{\mathchoice {\kern1ex\vcenter{\hrule height.4pt
width 6pt depth0pt} \kern-9.7pt} {\kern1ex\vcenter{\hrule
height.4pt width 4.3pt depth0pt} \kern-7pt} {} {} }}

\def\R{\mathbb{R}}

\begin{document}

\title[Stable cones in the thin one-phase problem]{Stable cones in the thin one-phase problem}

\author{Xavier Fern\'andez-Real}
\address{EPFL SB, Station 8, CH-1015 Lausanne, Switzerland}
\email{xavier.fernandez-real@epfl.ch}

\author{Xavier Ros-Oton}
\address{ICREA, Pg. Llu\'is Companys 23, 08010 Barcelona, Spain \& Universitat de Barcelona, Departament de Matem\`atiques i Inform\`atica, Gran Via de les Corts Catalanes 585, 08007 Barcelona, Spain \& Centre de Recerca Matem\`atica, Edifici C, Campus Bellaterra, 08193 Bellaterra, Spain}
\email{xros@icrea.cat}

%\keywords{Thin one-phase problem; Fractional Laplacian; Free boundary; fractional one-phase problem; fractional Bernoulli problem.}
%
%\subjclass[2020]{35R35, 35B07, 47G20.}
%
\thanks{This work has received funding from the European Research Council (ERC) under the Grant Agreements No 721675 and No 801867. 
In addition, X. F. was supported by the SNF grant 200021\_182565 and X.R. was supported by the Swiss National Science Foundation, by the grant RED2018-102650-T funded by MCIN/AEI/10.13039/501100011033, and by the Spanish State Research Agency, through the Mar\'ia de Maeztu Program for Centers and Units of Excellence in R{\&}D (CEX2020-001084-M). We would like to thank G. Grubb for her useful comments on the topics of this paper}

\begin{abstract}
The aim of this work is to study homogeneous stable solutions to the thin (or fractional) one-phase free boundary problem. 

The problem of classifying stable (or minimal) homogeneous solutions in dimensions $n\geq3$ is completely open.
In this context, axially symmetric solutions are expected to play the same role as Simons' cone in the classical theory of minimal surfaces, but even in this simpler case the problem is open.

The goal of this paper is twofold.
On the one hand, our first main contribution is to find, for the first time, the stability condition for the thin one-phase problem. Quite surprisingly, this requires the use of ``large solutions'' for the fractional Laplacian, which blow up on the free boundary. 

On the other hand, using our new stability condition, we show that any axially symmetric homogeneous stable solution in dimensions $n\le 5$ is one-dimensional, \emph{independently} of the parameter $s\in (0,1)$. 
\end{abstract}

\maketitle

\section{Introduction}

Consider the energy functional
\begin{equation}
\label{eq.Classical}
\mathcal{J}(u) = [u]^2_{H^1(B_1)} + \big|\{u > 0\}\cap B_1\big| = \int_{B_1} \left( |\nabla u|^2+\chi_{\{u > 0\}}\right)
\end{equation}
where $\chi_A$ is the characteristic function of the set $A$. 

The study of the critical points and minimizers of \eqref{eq.Classical} is known as the (classical) one-phase free boundary problem (or Bernoulli free boundary problem), which is a typical model for flame propagation and jet flows; see \cite{BL82, CV95, We03, PY07, AC81, ACF82, ACF82b, ACF83}. From a mathematical point of view, it was originally studied by Alt and Caffarelli in \cite{AC81}, and since then multiple contributions have been made; see \cite{Caf87, Caf89, Caf88, CJK04,CS05, DJ09,JS15,EE19,ESV20} and references therein.

In this paper, we deal with the fractional analogue of \eqref{eq.Classical}, in which the Dirichlet energy in the functional is replaced by the $H^s$ fractional semi-norm of order $s\in (0, 1)$, 
\begin{equation}
\label{eq.Classicals}
\mathcal{J}_s(u) = [u]^2_{H^s(B_1)} + \big|\{u > 0\}\cap B_1\big|,
\end{equation}
(see \eqref{eq.seminorm} below) which corresponds to the case in which turbulence or long-range interactions are present, and appears in particular in flame propagation; see \cite{CRS10, DS15b} and references therein. 
%It is also expected to appear in the blow-down limit of solutions to semilinear equations $\fls u = f(u)$ when $f$ is similar to a Dirac delta, \cite{FR19}.

This problem was first studied by Caffarelli, Roquejoffre, and Sire in \cite{CRS10}, where they obtained the optimal $C^s$ regularity for minimizers, the free boundary condition on $\partial\{u > 0\}$, and showed that Lipschitz free boundaries are $C^1$ in dimension $n = 2$. More recently, further regularity results for the free boundary have been obtained in \cite{DS12, All12, DR12, DSS14, DS15, DS15b, EKPSS20} among others.  These results imply that free boundaries are regular outside a certain set of singular points $\Sigma$, with ${\rm dim}_{\mathcal{H}}(\Sigma) \le n - k^*_s$ and $k_s^* \ge 3$. The value of $k_s^*$ is the lowest dimension in which there are stable/minimal cones.

Thus, to understand completely the structure and regularity of free boundaries, one must answer the following question:
\[
\textit{What is the first dimension $k_s^*$ in which stable/minimal cones appear?}
\]

This is the question that motivates our present work. 

\subsection{The non-local energy functional}
Let us consider the energy functional, 
\begin{equation}
\label{eq.energy}
\mathcal{J}_\Lambda (v, \R^n) =  [v]^2_{H^s(\R^n)} + \Lambda^2\big|\{v > 0\}\big|,
\end{equation}
depending on the parameter $\Lambda\in \R$, with the fractional semi-norm
\begin{equation}
\label{eq.seminorm}
\hspace{-1mm}[v]^2_{H^s(\R^n)} = \frac{c_{n,s}}{2}\iint_{\R^n\times \R^n} \frac{(v(x)- v(y))^2}{|x-y|^{n+2s}} \, dx\, dy,\quad\text{where}\quad c_{n, s} = \frac{s2^{2s} \Gamma\left(\frac{n+2s}{2}\right)}{\pi^{n/2}\Gamma(1-s)}
\end{equation}
is the constant appearing in the fractional Laplacian,
\[
\fls u (x)= c_{n,s} PV\int_{\R^n} \frac{u(x) - u(y)}{|x-y|^{n+2s}}\, dy.
\] 

Obtaining local minimizers to $\mathcal{J}_\Lambda$ is the \emph{fractional one-phase} free boundary problem. When $s = \frac12$ this is equivalent to the thin one-phase free boundary problem. It is a free boundary problem because, a priori, the zero-level set of the minimizer is unknown, and its boundary is called the ``free boundary''. After understanding the optimal regularity of minimizers, the study of the free boundary constitutes the main topic of research for this type of problem.

Let $u$ be a local minimizer (or critical point) to \eqref{eq.energy} in a ball $B$ (see \eqref{eq.criticalpoint}). Let $\Omega = \{u > 0\}$, and let us suppose $\Omega$ is smooth enough. Let
\[
	d(x) = {\rm dist}(x, \partial \Omega).
\]
Then, by standard variational arguments we have that $\fls u = 0$ in $\Omega\cap B$. Moreover, we have that $u$ solves the following problem involving a condition on the fractional derivative on $\partial \Omega$,
\begin{equation}
\label{eq.OPT}
\left\{
\begin{array}{rcll}
\fls u & = & 0 & \quad\textrm{in } \Omega\cap B\\
u & = & 0 & \quad\textrm{in } \Omega^c \cap B \vspace{1mm} \\ 
\Gamma(1+s) \displaystyle{\frac{u}{d^{s}}} & = & \Lambda& \quad \textrm{on }\partial \Omega\cap B.
\end{array}
\right.
\end{equation}
This is the first variation of the energy functional. It was first proved in \cite{CRS10} but, unfortunately, with a computational mistake in the derivation of the constant. For completeness, in Proposition~\ref{prop.crs} below, we find the precise constant $\Gamma(1+s)$ which, as far as we know, was only explicitly known for the case $s = \frac12$ (see Remarks~\ref{rem.1} and \ref{rem.2} below). We refer to Section~\ref{sec.2} for the definition of critical point.

\subsection{The stability condition}

A main goal of this paper is to obtain the second variation of the energy functional. Namely, we will find the stability condition for~\eqref{eq.energy}.

In order to state the result, we need the following definition:

\begin{defi}
\label{defi.kernel}
Let $\Omega$ be a $C^{1,\alpha}$ domain outside the origin, and let $G_{\Omega, s}(x, y)$ be the Green function of the operator $\fls$ for the domain $\Omega$. Then, we define the kernel $\mathcal{K}_{\Omega, s}:\partial\Omega\times\partial\Omega\to \R$ as
\begin{equation}
\label{eq.kernelK_C}
\mathcal{K}_{\Omega, s}(x, y) = \lim_{\substack{\Omega\ni \bar x\to x\\\Omega\ni \bar y \to y }} \frac{G_{\Omega, s}(\bar x ,\bar y)}{d^s(\bar x)d^s(\bar y)}.
\end{equation}
By well-known boundary regularity estimates for the fractional Laplacian (\cite{RS14, RS17}), \eqref{eq.kernelK_C} is well-defined as soon as the boundary is $C^{1,\alpha}$. 

Furthermore, we also define the following curvature-type term
\[
\mathcal H_{\Omega, s}(x) \coloneqq \int_{\partial \Omega} |\nu(x) - \nu(y)|^2 \mathcal{K}_{\Omega, s}(x, y) d\sigma(y)
\]
for $x\in \partial\Omega$, and where $\nu:\partial\Omega\to \mathbb{S}^{n-1}$ denotes the unit inward normal vector on $\partial\Omega$, and $\sigma$ denotes the area measure on $\partial\Omega$.
\end{defi}

We can now state the second-variation condition for the energy functional \eqref{eq.energy}.

\begin{thm}
\label{thm.main}
Let $s\in(0, 1)$ and let $u\in C^s(\R^n)$ be a global $s$-homogeneous stable solution to  \eqref{eq.OPT}, in the sense \eqref{eq.criticalpoint}-\eqref{eq.stablesolution}. Assume that $\Omega \coloneqq \{u > 0\}$ is a $C^{2,\alpha}$ domain outside the origin.

Let $\mathcal{K}_{\Omega, s}$ and $\mathcal H_{\Omega, s}$ be given by Definition~\ref{defi.kernel}. Then, we have
\begin{equation}
\label{eq.stabilityFin_C}
 \int_{\partial\Omega}\int_{\partial\Omega} \big(f(x) - f(y)\big)^2\mathcal{K}_{\Omega, s}(x, y)\,d\sigma(x) \, d\sigma(y)\ge \int_{\partial\Omega} H_{\Omega, s}  f^2 \, d\sigma(x)
\end{equation}
for all $f\in C^\infty_c(\partial\Omega\setminus\{0\})$. 

Furthermore,  $\mathcal{K}_{\Omega, s}$ is $(-n)$-homogeneous and
\begin{equation}
\label{eq.Kcone}
\mathcal{K}_{\Omega, s}(x, y) \asymp \frac{1}{|x-y|^{n}}\quad\text{for all }~~x, y\in \partial\Omega,
\end{equation}
while $\mathcal H_{\Omega, s}$ is $(-1)$-homogeneous, and
\[
\mathcal{H}_{\Omega, s}(x) \asymp \frac{1}{|x|}\quad\text{for all }~~x\in \partial\Omega,
\]
if $\Omega$ is not a half space.
\end{thm}

Here, we have denoted $g_1(x)\asymp g_2(x)$ if $C^{-1} g_2(x) \le g_1(x) \le C g_2(x)$ for some positive constant $C$ independent of $x$. 

The result stated here is for $s$-homogeneous solutions since we are mainly interested in blow-ups at free boundary points.
We refer the reader to Theorem~\ref{thm.cor.1} below for a more general result dropping the $s$-homogeneity hypothesis.

As we will see later on in the paper, the stability condition \eqref{eq.stabilityFin_C} has an equivalent formulation in terms of large solutions for the fractional Laplacian (which were introduced and studied in \cite{Aba15,Gru15}).
More precisely, \eqref{eq.stabilityFin_C} turns out to be equivalent to 
\begin{equation}\label{intro-large}
\int_{\partial \Omega} f \, T_{\Omega,s}f \, d\sigma \geq \kappa_s\int_{\partial \Omega} U_1 f^2 \, d\sigma
\end{equation}
for all $f\in C^\infty_c(\partial\Omega\setminus\{0\})$,
where 
\[U_1\coloneqq -\frac{1}{\Lambda}\,\partial_\nu\left(\frac{u}{d^s}\right),\qquad T_{\Omega,s}f\coloneqq -\partial_\nu\left(\frac{F}{d^{s-1}}\right),\]
and $F$ is the unique solution of
\[\left\{
\begin{array}{rcll}
\fls F & = & 0 & \quad\textrm{in } \Omega\\
F & = & 0 & \quad\textrm{in } \Omega^c\\
\displaystyle \frac{F}{d^{s-1}} & = & f& \quad \textrm{on }\partial \Omega
\end{array}
\right.\]
satisfying $F \to 0$ for $|x|\to\infty$. (Notice that $F$ blows-up on the free boundary $\partial\Omega$.)

Such equivalence is not trivial, and actually $U_1$ is related, but not equal, to $\mathcal H_{\Omega,s}$.

\begin{rem}
Alternatively, we can rewrite \eqref{eq.stabilityFin_C} in a more symmetric way as 
\[
\int_{\partial\Omega}\int_{\partial\Omega} \left\{\nu(x) \cdot\nu(y) \left(f(x)^2+f(y)^2\right) - 2 f(x) f(y)\right\}\mathcal{K}_{\Omega, s}(x, y) d\sigma(x) d\sigma(y)\ge 0 
\]
for all $f\in C^\infty_c(\partial\Omega\setminus \{0\})$.
\end{rem}

\begin{rem} We emphasize that \eqref{eq.stabilityFin_C} ---or \eqref{intro-large}--- is the non-local counter-part of the result by Caffarelli, Jerison, and Kenig for stable solutions in the classical one-phase obstacle problem \cite{CJK04}. 
In particular, one can show that  both $\mathcal H_{\Omega, s}$ and $U_1$ converge to the mean curvature $H$ of $\partial\Omega$ when $s\uparrow 1$.
In that case, the stability condition can be written as 
\[ \int_{\Omega} |\nabla f|^2 \, dx \geq \int_{\partial \Omega} H f^2 \, d\sigma\]
for all $f\in C^\infty_c(\R^n\setminus\{0\})$.
Unfortunately, we do not have such a simple expression in the nonlocal case $s\in(0,1)$.
\end{rem}

\begin{rem}[Nonlocal minimal surfaces]
Our expression \eqref{eq.stabilityFin_C} also has a similar structure to the one obtained for the second variation of nonlocal perimeters at nonlocal minimal surfaces, see \cite[Eq. (1.5)-(1.6)]{DDW18} and \cite[Theorem 6.1]{FFMMM15}, which is then used in \cite{CCS20} to classify stable $s$-minimal cones in $\R^3$ for $s\sim 1$. Notice, however, that our expression is fundamentally different in nature: while the scalings for the stability condition for nonlocal perimeters are of order $1+s$, our scalings preserve the local structure independently of $s$ (with order 1). Moreover, as we will see, obtaining our stability condition \eqref{eq.stabilityFin_C} or \eqref{intro-large} turns out to be more delicate, and requires fine estimates for $s$-harmonic functions near the boundary. 
\end{rem}

\subsection{Scaling}

Note that since $\partial\Omega$ is $(n-1)$-dimensional, the left-hand side in \eqref{eq.stabilityFin_C} behaves roughly as a fractional Laplacian of order $1$ (i.e., like $\sqrt{-\Delta}$) on $\partial\Omega$, and this is exactly true if $\Omega$ is a half-space.

On the other hand, $\mathcal H_{\Omega, s}$ is $(-1)$-homogeneous (i.e. it equals $H_1 |x|^{-1}$ for a $0$-homogeneous $H_1$). In particular, it is some kind of non-local curvature term that nonetheless preserves the local curvature scaling. Thus, the expression \eqref{eq.stabilityFin_C} can be understood as a Hardy-type inequality on $\partial\Omega$. In particular, as an immediate consequence of this, one can see by an asymptotic analysis that, in dimension $n=~2$, the only stable cones are half spaces. This was known for minimizers, \cite{DS15, EKPSS20}, and here we give a different and short proof of the following result.

\begin{cor}
\label{cor.2D}
Let $s \in (0, 1)$ and let $u\in C^s(\R^2)$ be an $s$-homogeneous global solution to the one phase problem \eqref{eq.OPT}. Assume $u$ is stable in the sense \eqref{eq.criticalpoint}-\eqref{eq.stablesolution}. Then, up to a rotation and a multiplicative constant, $u$ is the half-space solution, $u = (x_1)_+^s$. 
\end{cor}

In higher dimensions, $n\ge 3$, the situation is much more complicated and cannot be understood simply by scaling.
Indeed, we expect the inequality \eqref{eq.stabilityFin_C} to be always true for a sufficiently large multiplicative constant.

\subsection{Axially symmetric cones}

The classification of stable/minimal cones in dimension $n \ge 3$ is an extremely challenging problem. Even in case of the classical Alt-Caffarelli functional \eqref{eq.Classical}, the problem is still not completely understood \cite{JS15}.

In the context of minimal surfaces, the Simons cone is the first counter-example of a non-smooth minimal cone for $n \ge 8$. 
As a consequence, the natural candidates for non-trivial minimal cones in the context of nonlocal minimal surfaces are those with symmetry of ``double revolution'' \cite{DDW18}. 

%the Allen-Cahn equation are those with ``double revolution'' symmetry or ``saddle-shaped'' (see \cite{CT09, LWW16}); also in the fractional Allen-Cahn problem (see \cite{FS20, FS20b, FS19}). 

The role played by the Simons cone for minimal surfaces is played by {axially symmetric} cones in the one phase free boundary problem. 
Indeed, in this context, the natural non-trivial solutions have axial symmetry; see \cite{DJ09}. 
As such, axially symmetric solutions have also been studied in \cite{CJK04, FR19, LWW19}.

Thus, for the thin/fractional one-phase free boundary problem, the first case to be understood is that of \emph{axially symmetric} cones. 
Let us define, for each $\beta\in (0, \infty)$, the axially symmetric cone 
\[
\mathcal{C}(\beta) \coloneqq \big\{(x', x_n)\in \R^{n-1}\times\R : \, |x'|> \beta |x_n| \big\}.
\]
\begin{figure}[h]
\centering
\includegraphics{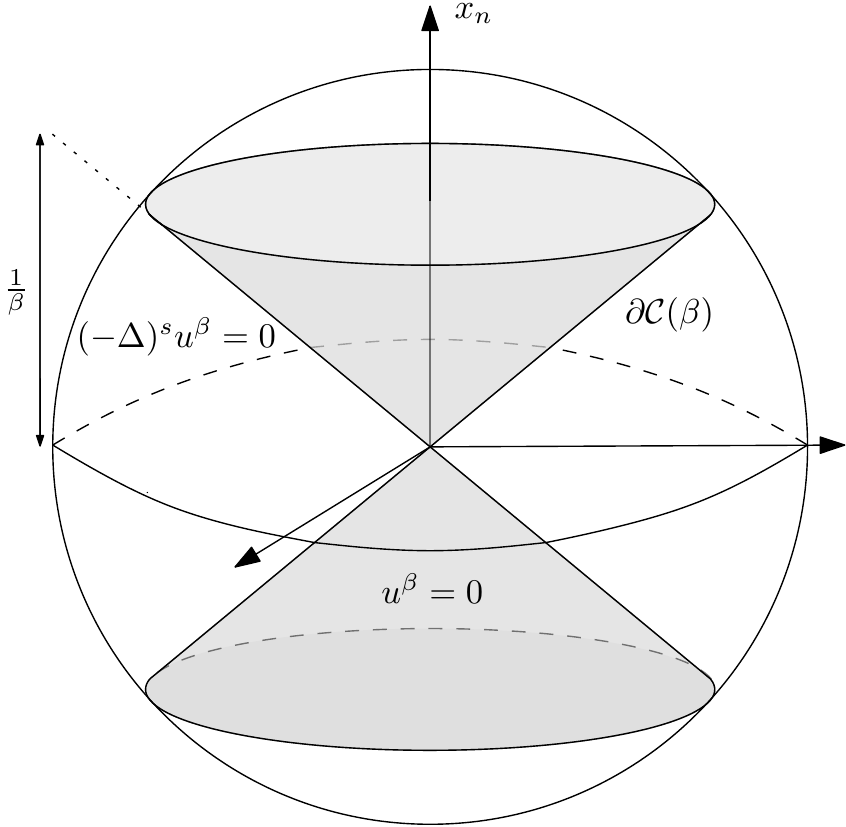}
\caption{The greyed area is the complement of the cone $\mathcal{C}(\beta)$.}
\label{fig.1}
\end{figure}

Let us consider $u^\beta$ the unique positive solution to 
\[
\left\{
\begin{array}{rcll}
\fls u^\beta & = & 0 & \quad\textrm{in } \mathcal{C}(\beta)\\
u^\beta & = & 0 & \quad\textrm{in } \R^n\setminus \mathcal{C}(\beta),
\end{array}
\right.
\]
see Figure~\ref{fig.1}. 
Then $u^\beta$ is $\lambda$-homogeneous, where $\lambda = \lambda(\beta)\in (0, 2s)$. Moreover, $\lambda$ is continuous and strictly monotone in $\beta$ (see \cite{TTV18}), so that there exists a unique $\beta_{n, s}$ such that $u_s \coloneqq u^{\beta_{n, s}}$ is homogeneous of degree $s$. 
In particular, $u/{d^s}$ is $0$-homogeneous, and therefore, by symmetry, it is constant on $\partial \Co(\beta_{n, s})$. Hence, up to a multiplicative constant, $u_s$ is a solution to the fractional one-phase problem with contact set given by $\Co_{n, s} \coloneqq \Co(\beta_{n, s})$.

In the classical one-phase problem ($s=1$), these solutions $u_s$ are known to be unstable in dimensions $n \le  6$ (see \cite{CJK04}).
Still, even in such case, the proof in  \cite{CJK04} is quite delicate and required some fine numerical computations. 

Here, we use our new stability condition \eqref{intro-large} to prove that stable (and in particular, minimal) axially-symmetric cones for the fractional one-phase problem are trivial in dimensions $n\le  5$.

\begin{thm}
\label{thm.low-dim}
Let $s\in (0, 1)$ and let $u\in C^s(\R^n)$ be a global stable solution to \eqref{eq.OPT}, in the sense \eqref{eq.criticalpoint}-\eqref{eq.stablesolution}. 
Assume in addition that $u$ is $s$-homogeneous and axially-symmetric. 
Then, if $n \le 5$, $u$ is one-dimensional.
\end{thm}

To our surprise, our proof of the previous result gives as a condition on $n$ that $n \le 6-\delta$ for any $\delta > 0$, \emph{independently} of $s$. Since we already know that the case $n = 6$ is the best we could hope for if $s = 1$ (\cite{DJ09}), based on Theorem~\ref{thm.low-dim} (and its proof), we conjecture the following:

\begin{conjecture}
Let $s\in (0, 1)$ and let $u\in C^s(\R^n)$ be a global, stable, $s$-homogeneous solution to the fractional one-phase problem. If $n \le 6$, then $u$ is one-dimensional.
\end{conjecture}

As said above, Caffarelli, Jerison, and Kenig proved  in \cite{CJK04} that the analogous of Theorem~\ref{thm.low-dim} for $s=1$ holds up to dimension $n = 6$. 
We show in Section~\ref{sec.7} what would be the analogy for $s\in(0,1)$ to the approach of \cite{CJK04}, using our new stability condition.
We believe this could yield our previous result for $\mathcal{C}_{n,s}$ up to dimension $n = 6$, for all $s\in (0,1)$, but unfortunately this seems to require some delicate numerical computations.

\subsection{Ideas of the proofs}
\label{ss.local}

Let us present here the ideas of the proofs of our two main results, Theorems~\ref{thm.main} and \ref{thm.low-dim}.

\subsubsection{Ideas of the proof of Theorem~\ref{thm.main}.}
The proof of Theorem~\ref{thm.main} is done by constructing explicit competitors and computing the corresponding energy to deduce an expansion up to second order around a critical point. Roughly speaking, the inequality \eqref{eq.stabilityFin_C} corresponds to the excess energy at order $\eps^2$ of $u$ when perturbing the domain $\Omega$ by $\eps f$ in the normal direction. 

Indeed, given a solution $u$ with $\Omega = \{u > 0\}$ smooth, and given an arbitrary function $f\in C^\infty_c(\partial \Omega)$, we consider domain perturbations such that $\partial \Omega$ is stretched by $\eps f(z)$ in the normal direction at $z\in \partial\Omega$. In this way, we obtain a new domain $\Omega_\eps$ which is ``$\eps f$-close'' to $\Omega$. The energy of the new stretched $u$ can be lowered by making it $s$-harmonic on $\Omega^c_\eps$, so that we consider $v_\eps$ our $\eps$-close perturbation of $u$ (in the ``direction'' $f$) to be the solution to
\[
\left\{
\begin{array}{rcll}
\fls v_\eps & = & 0 & \quad\textrm{in } \Omega_\eps \cap B\\
v_\eps & = & 0 & \quad\textrm{in } \Omega_\eps^c \cap B\\
v_\eps & = & u & \quad\textrm{in } \R^n\setminus B. 
\end{array}
\right.
\]
Then we compute the expansion of the energy $\mathcal{J}_{\Lambda, B}(v_\eps)$ (see \eqref{eq.energy2}) in $\eps$:
\[
\mathcal{J}_{\Lambda, B}(v_\eps) = \mathcal{J}_{\Lambda, B}(u) + \eps A_1(u, f) + \eps^2 A_2(u, f) + o(\eps^2). 
\]
The first term, $A_1(u, f)$, corresponds to the first variation of the functional. Imposing that this term vanishes for all $f\in C^\infty_c(\partial\Omega)$ is what yields the constant fractional derivative condition on $\partial\Omega$. 

The second term, $A_2(u, f)$, corresponds to the second variation. The fact that $u$ is stable implies that $A_2(u, f)\ge 0$ for all $f\in C^\infty_c(\partial\Omega)$, and this yields the stability condition from our main result. Let us very briefly explain how to explicitly compute $A_1(u, f)$ and $A_2(u, f)$. 

We assume, for simplicity, that $f\ge 0$, so that we can separate between semi-norms and the measure of $\Omega\setminus \Omega_\eps$ as follows,
\[
\mathcal{J}_{\Lambda, B}(v_\eps) - \mathcal{J}_{\Lambda, B}(u) = [v_\eps]^2_{H^s(\R^n)} - [u]^2_{H^s(\R^n)}  - \Lambda^2|(\Omega\setminus \Omega_\eps)\cap B|.
\]
(In fact, each semi-norm could be infinite, but the difference can be computed, using \eqref{eq.energy2}.) For the second term in the previous expression, a simple geometric argument yields that 
\[
|(\Omega\setminus \Omega_\eps)\cap B| = \eps\int_{\partial\Omega} f + \frac{\eps^2}{2} \int_{\partial\Omega} H f^2 + o(\eps^2),
\]
where $H$ is the mean curvature of $\partial\Omega$ with respect to $\{u = 0\}$. Thus, we just need to expand the difference of semi-norms, which after some manipulations corresponds to 
\begin{equation}
\label{eq.expsemi}
[v_\eps]^2_{H^s(\R^n)} - [u]^2_{H^s(\R^n)} =  -\int_{(\Omega\setminus \Omega_\eps)\cap B} u \fls v_\eps. 
\end{equation}
Notice that the integral is performed in a region $\eps$-close to $\partial\Omega$ (and $\partial\Omega_\eps$). The value of the previous integral will depend on the function $u$, and more specifically, on the behaviour of $u$ near $\partial\Omega$. More precisely, from the boundary regularity for the fractional Laplacian in $C^{1,\alpha}$ domains we know that, if $z\in \partial\Omega_\eps$ is the projection of $x$ onto $\partial\Omega_\eps$, and we denote $d_\eps(x) = {\rm dist}(x, \Omega_\eps^c)$, then
\[
v_\eps(x) = U_0^\eps(z) d_\eps^s(x) + o(d_\eps^s),
\]
for some function $U_0^\eps\in C^0(\partial\Omega_\eps)$. We can now compute the expansion of $\fls v_\eps$ in $\Omega_\eps^c$, which is 
\[
\fls v_\eps(x) = \bar c_s U_0^\eps(z) \bar d_\eps^{-s}(x) + o(\bar d_\eps^{-s}),
\]
and where $\bar d_\eps = {\rm dist}(x, \Omega_\eps)$, and $\bar c_s$ is an explicit constant.  Notice that $v_0 = u$, so that plugging these expansions in \eqref{eq.expsemi} and using that $U_0^\eps = U_0^0 + o(\eps)/\eps$, $d_\eps = d_0 + o(\eps)$, we obtain that
\[
A_1(u, f) = C \int_{\partial\Omega} f(z) \left[c_s (U_0^0(z))^2 - \Lambda^2 \right] \, dz
\]
for some constant $c_s$. Imposing that $A_1(u, f)$ vanishes for all $f\in C^\infty_c(\partial\Omega)$ gives the constant fractional derivative condition on $\partial\Omega$. 

In order to obtain the term in $\eps^2$, $A_2(u, f)$, we need to consider the previous expansions at a higher order. Roughly, in this case we have that
\[
v_\eps(x) \thickapprox U_0^\eps(z) d_\eps^s(x) + U_1^\eps(z) d_\eps^{1+s}(x) + o(d_\eps^{1+s}).
\]
(There is also an extra tangential term, that ends up having no role.) Thus, the first step is to expand $\fls v_\eps(x)$ from here. This is a delicate argument done in Lemma~\ref{lem.form2}, from which, roughly,
\[
\fls v_\eps(x) \thickapprox \bar c_s U_0^\eps(z) \bar d_\eps^{-s}(x) + \bar c_* U_0^\eps(z) H(z) \bar d^{1-s}_\eps(x) + U_1^\eps(z) \bar d_\eps^{1-s}(x) + o(\bar d_\eps^{1-s}). 
\]
We again want to plug this in \eqref{eq.expsemi} to get the terms of order $\eps^2$. In this case, for the terms multiplying $\bar d_\eps^{1-s}$ we use, as before, that $U_0^\eps = U_0^0+o(\eps)/\eps$ and $U_1^\eps = U_1^0 + o(\eps)/\eps$, where now $U_0^0$ is constant and $U_1^0 = \partial_\nu(u /d^s)$ on $\partial\Omega$ (here, $\partial_\nu$ denotes the normal derivative to $\partial\Omega$). For the first term, we need a higher order expansion, in $\eps$, for $U_0^\eps$. That is, 
\[
U_0^\eps = U_0^0 + \eps \tilde A(u, f) + o(\eps). 
\]    
We can now compute $\tilde A(u, f)$ using that $(v_\eps - u)\eps^{-1}\to F_s$, as $\eps \downarrow 0$, and where $F_s$ solves
\begin{equation}
\label{eq.nicolaeq}
\left\{
\begin{array}{rcll}
\fls F_s & = & 0 & \quad\textrm{in } (\Omega \cap B)\\
F_s & = & 0 & \quad\textrm{in } (\Omega\cap B)^c\\
\displaystyle \frac{F_s}{d^{s-1}} & = & s f& \quad \textrm{on }\partial \Omega.
\end{array}
\right.
\end{equation}
(We remark that \eqref{eq.nicolaeq} is a Dirichlet-type problem for the fractional Laplacian involving large solutions, and was first studied in \cite{Aba15, Gru15}. When $s\uparrow1$, such problem converges to the classical Dirichlet problem for the Laplacian.) In particular, $\tilde A(u, f)$ depends on $u$ and $f$ through $U_0^1$ and $\partial_\nu(F/d^{s-1})$.

Putting all together, the stability condition $A_2(u, f)\ge 0$ is
\[
\frac{C_s}{\Lambda}\int_{\partial \Omega} \partial_\nu\left(\frac{u}{d^s}\right) f^2 \, d\sigma \ge \int_{\partial \Omega} f \partial_\nu\left(\frac{F_s}{d^{s-1}}\right)d\sigma.
\]
From here, and after some nontrivial manipulations, we can express $\partial_\nu\left(u/d^s\right)$ in terms of $\Omega$ through the Green function, to get our desired result, Theorem~\ref{thm.main}.

%******
%
%
% and it happens to be $A_1(u, f)  = \int_{\partial \Omega} f \left[\Gamma(1+s)^2 (u/d^s)^2 - \Lambda^2\right]$. Asking $u$ to be a minimizer (or a critical point) implies that $A_1(u, f)$ vanishes for all $f\in C^\infty_c(\partial\Omega)$, from which the fractional derivative boundary condition of the one-phase problem arises. 
%
%In order to do that, we need to establish fine estimates of the expansion of the fractional Laplacian for $s$-harmonic functions in smooth domains, as done in Lemma~\ref{lem.form2}. Doing this, we obtain a localized version of the previous result valid for local minimizers in a ball $B$, Proposition~\ref{prop.main}. In this case, the condition depends on the solution $u$, and in the proof of Theorem~\ref{thm.cor.1} we express it in terms of the domain $\Omega$ only, which then yields as a particular case Theorem~\ref{thm.main}. This condition is expressed as a singular boundary value problem for the fractional Laplacian, which is a possible Dirichlet problem analogy in the non-local case, and we briefly introduce and discuss in ....****

\subsubsection{Ideas of the proof of Theorem~\ref{thm.low-dim}.}

To prove Theorem~\ref{thm.low-dim}, we need a local formulation of the stability condition: an alternative form of the stability condition, as seen in the extension variable (introduced in Subsection~\ref{ss.extension}).
%It is a generalization of Proposition~\ref{prop.main} (and Theorem~\ref{thm.main}). 

More precisely, if we extend to $\R^{n+1}_+$ as $(x, y)\in \R^n\times\R_+$, and we denote by $u = u(x)$ our solution (so that, as an abuse of notation $u:\R^n\times\{0\}\to \R$ and $\Omega\subset\R^n\times\{0\}$), then in Proposition~\ref{prop.main_E} we prove that our stability condition in the extended variable reads as
\begin{equation}
\label{eq.condFinB_E_I}
\begin{split}
\int_{\partial \Omega} \frac{F}{d^{s-1}} \bigg\{\partial_\nu \left(\frac{F}{d^{s-1}}\right) & -\frac{\Gamma(2+s)}{s\Lambda} \partial_\nu\left(\frac{u}{d^s}\right) \frac{F}{d^{s-1}}  -2^{s-1}\frac{1-s}{s}y^{1+s}L_a F \bigg\}\leq \\
& ~\le -\frac{d_s}{\Gamma(1+s)\Gamma(s)} \int_{\Omega} F \partial_y^a F - \frac{d_s}{\Gamma(1+s)\Gamma(s)} \int_{\{y > 0\}} F L_a F,
\end{split}
\end{equation}
for all $F$ with $F \equiv 0$ in $(\R^n\setminus \Omega)\times\{0\}$ and such that each of the previous terms is well-defined, and where $d_s$ is given by \eqref{eq.ds}. (We recall that $\nu$ denotes the unit inward normal to $\partial \Omega$.) 

Notice that we are interested in test functions $F$  that blow up like $d^{s-1}$ when approaching $\partial\Omega\times\{0\}$ (namely, behaving as the \emph{large} solutions introduced above). We also denote $u = u(x) = u(\s, \tau)$ where $\s^2 = x_1^2+ \dots + x_{n-1}^2$ and $\tau = x_n$, that is, $u$ is axially symmetric in the $x_n$-direction. 

Once condition \eqref{eq.condFinB_E_I} is established, we take 
\[F = \eta\, \partial_\s u \]
as a test function.
Here, we take $u$ to be its own $a$-harmonic extension towards $\{y > 0\}$, and denote $\partial_\s$ the  derivative in the $\s$ direction.
It is important to notice that $\partial_\s u$ is a \emph{large} solution of the type \eqref{eq.nicolaeq}.  

We show in Proposition~\ref{prop.stab_E} that, somewhat surprisingly, this yields a new and much simpler stability condition in the extended variable,
\begin{equation}
\label{eq.stab_E_2}
\int_{\{y > 0\}}  (\partial_\s u)^2 \,|\nabla \eta|^2\,y^{1-2s} \, dx\,dy
\ge (n-2) \int_{\{y > 0\}}  (\partial_\s u)^2 \,\eta^2 \,\s^{-2}\, y^{1-2s}\,dx\,dy 
\end{equation}
for all test functions $\eta$. 
By taking now 
\[\eta = \s^{-\alpha}\]
and optimizing in $\alpha$, we reach a contradiction with the stability condition for non-trivial solutions if $n < 6$.

\vspace{3mm}

The idea of taking $x\cdot\nabla u$ or $\partial_\s u$ in the stability condition goes back to \cite{CC04}, where Cabr\'e and Capella studied radial stable solutions of $-\Delta u=f(u)$.
More recently, this type of test function has been also used in \cite{CR13,CFRS20} in case of semilinear equations, and even in the classical one-phase problem in \cite{FR19} when seen as a limit of semilinear equations.
Finally, in case of nonlocal equations of the type $(-\Delta)^s u=f(u)$, this idea has been used in \cite{CDDS11,San18}.

Our proof here turns out to be much more delicate, since our stability condition \eqref{eq.condFinB_E_I}-\eqref{eq.nicolaeq} is quite different (and much more singular) than those for semilinear equations.
Still, we end up obtaining a simple stability condition \eqref{eq.stab_E_2} with no free boundary terms, in which we can then take appropriate test functions $\eta$.

\subsection{Structure of the paper}

The paper is organized as follows.

In Section~\ref{sec.2} we introduce some preliminary results that will be useful throughout the work, and we obtain the critically condition or first variation condition for the functional \eqref{eq.energy} with the explicit constant $\Gamma(1+s)$. In Section~\ref{sec.3} we then focus our attention on second order variations and obtain the stability condition Theorem~\ref{thm.main} (see also Proposition~\ref{prop.main}). In order to do that we use fine estimates for the expansion of the fractional Laplacian of an $s$-harmonic function outside the domain, in Lemma~\ref{lem.form2}. In Section~\ref{sec.4} we then use our main result, Theorem~\ref{thm.main}, to prove Corollary~\ref{cor.2D} on the instability of non-trivial cone-like solutions in $\R^2$.

In Section~\ref{sec.5} we express the previously obtained stability condition in $\R^n$ in terms of the extension variable towards $\R^{n+1}$, in Proposition~\ref{prop.main_E}. We then use it in Section~\ref{sec.6} to prove Theorem~\ref{thm.low-dim}, stating that axially-symmetric solutions are either one-dimensional or unstable, for dimensions up to $n = 5$. We finish, in Section~\ref{sec.7}, with what would be the analogous numerical stability condition approach developed by Caffarelli, Jerison, and Kenig in \cite{CJK04}, in the context of the fractional one-phase problem, and that could yield the optimal dimension $n = 6$ for the previous statement.

%\subsection{Acknowledgements}
%
%This work has received funding from the European Research Council (ERC) under the Grant Agreements No 721675 and No 801867. 
%
%In addition, X. F. was supported by the SNF grant 200021\_182565 and X.R. was supported by the Swiss National Science Foundation, by the grant RED2018-102650-T funded by MCIN/AEI/10.13039/501100011033, and by the Spanish State Research Agency, through the Mar\'ia de Maeztu Program for Centers and Units of Excellence in R{\&}D (CEX2020-001084-M).
%

\section{Preliminaries and the first variation}
\label{sec.2}

In this section we introduce some preliminaries regarding the definitions of local minimizer, critical points, and stable solutions for \eqref{eq.energy}, as well as the Caffarelli-Silvestre extension.
Then, we find the first variation condition, Proposition~\ref{prop.main}, computing the explicit constant $\Gamma(1+s)$ in \eqref{eq.OPT}.

We start with the basic definitions for the energy functional \eqref{eq.energy}.

\subsection{Local minimizer, critical point, and stable solution}

\label{ss.crit_stable}

Let us define what we mean by local minimizer for the energy functional $\mathcal{J}_\Lambda$ (cf. \cite{EKPSS20}). 
Let $B\subset \R^n$ be a fixed ball. 
We want a function $u$ such that, under perturbations in $B$, $\mathcal{J}_\Lambda(u)$ cannot decrease its energy. 
In general, though, such energy could (and, in many cases, will) be infinite. 
To avoid this, we instead consider the associated functional $\mathcal{J}_{\Lambda, B}$  involving only those terms of $\mathcal{J}_\Lambda$ that could change under perturbations in $B$:
\begin{equation}
\label{eq.energy2}
\mathcal{J}_{\Lambda, B}(v) =  \frac{c_{n,s}}{2} \iint_{\R^{2n}\setminus (B^c)^2} \frac{(v(x) - v(y))^2}{|x-y|^{n+2s}}\, dx\, dy + \Lambda^2 |\{v > 0\}\cap B|. 
\end{equation}
We then say that $u\in L^1_{\rm loc}$ is a local minimizer of $\mathcal{J}_\Lambda$ in $B$ if 
\begin{equation}
\label{eq.loc_min}
\text{
$\mathcal{J}_{\Lambda, B}(u) \le \mathcal{J}_{\Lambda, B}(v)$ for all $v$ s.t. $v - u\in H^s(\R^n)$ and $u \equiv v$ in $\R^n \setminus B$.
}
\end{equation}
We say that $u$ is a global minimizer of $\mathcal{J}_\Lambda$, if it is a local minimizer for all $B\subset \R^n$.

Since the functional $\mathcal{J}_\Lambda$ is non-smooth, the notion of critical (and stable) points is delicate. We will always be dealing with weak/viscosity solutions to the problem, and in our assumptions we will include that the domain $\Omega \coloneqq \{u > 0\}$ is smooth around the points we want to deal with. Under these assumptions, in order to get the first (and second) variation of our functional it is enough to consider smooth domain variations. The definition of critical point and stable solution  presented here are made under the assumption that the previous hypotheses hold.

Given a domain variation $\Psi\in C^\infty_c(B; \R^n)$ we define
\[
u_\eps(x) \coloneqq u(x+\eps \Psi(x)).
\] 
We then say that $u\in L^1_{\rm loc}$ with $\Omega = \{u = 0\}$ smooth is a critical point (with respect to domain variations) of $\mathcal{J}_\Lambda$ if
\begin{equation}
\label{eq.criticalpoint}
\frac{d}{d\eps}\bigg|_{\eps = 0} \mathcal{J}_{\Lambda, B} (u_\eps)  = 0\quad\text{for all $B\subset \R^n$ and $\Psi\in C^\infty_c(B; \R^n)$}.
\end{equation}

Similarly, we say that $u\in L^1_{\rm loc}$ is a stable solution (with respect to domain variations) of $\mathcal{J}_\Lambda$ if it is a critical point, \eqref{eq.criticalpoint} and
\begin{equation}
\label{eq.stablesolution}
\frac{d^2}{d\eps^2}\bigg|_{\eps = 0} \mathcal{J}_{\Lambda, B} (u_\eps)  \ge 0\quad\text{for all $B\subset \R^n$ and $\Psi\in C^\infty_c(B; \R^n)$}.
\end{equation}

We now show how to ``localize'' the problem, by means of the Caffarelli-Silvestre extension for the fractional Laplacian.

\subsection{The extension variable}

\label{ss.extension}

While we will often work with the nonlocal formulation of the variational problem, \eqref{eq.energy}-\eqref{eq.energy2}-\eqref{eq.loc_min}, the fractional one-phase obstacle problem is sometimes referred (and studied) as the thin one-phase problem (see \cite{CRS10, DS12, DS15, DS15b, EKPSS20} among others). This is due to the Caffarelli-Silvestre  extension for the fractional Laplacian and fractional Sobolev norms (\cite{CS07}), that allows an equivalent formulation of the previous non-local variational problem as a local problem defined in one extra dimension. Namely, if we want to compute $\fls v$ for some $v:\R^{n}\to \R$, and we denote the points in $\R^{n+1}$ as $(x, y)\in \R^n\times\R$, we can consider the $a$-harmonic extension of $v$ towards $\R^{n+1}_+$. That is, a function $\bar v: \R^{n+1}\to \R$ such that
\[
L_a \bar v = 0\quad\text{for}\quad y > 0,\qquad \bar v(x, 0) = v(x),
\]
where 
\begin{equation}
\label{eq.La}
L_a \bar v \coloneqq {\rm div}(|y|^a\nabla \bar v),\qquad a  =1-2s\in (-1, 1). 
\end{equation}
Then,
\begin{equation}
\label{eq.FracDer}
\fls v(x) = -d_s \partial_y^a \bar v (x,0),\qquad\text{where} \quad \partial_y^a v(x, 0) =\lim_{y\downarrow 0} y^a\partial_y \bar v(x, y) 
\end{equation}
and we have denoted 
\begin{equation}
\label{eq.ds}
d_s = 2^{2s-1}\frac{\Gamma(s)}{\Gamma(1-s)}. 
\end{equation}
We also have the equivalence, in this case, 
\begin{equation}
\label{eq.fromdowntoup}
[v]_{H^s(\R^n)} = d_s [\bar v]_{H^1(\R^{n+1}_+, |y|^a)}
\end{equation}
where we have introduced the weighted Sobolev space $H^1(\Omega,|y|^a)$ with semi-norm,
\[
[w]^2_{H^1(\Omega, |y|^a)} = \int_{\Omega} |\nabla w|^2|y|^a\, dx\, dy. 
\] 

Thus, we define the following local energy functional in $\R^{n+1}$
\begin{equation}
\label{eq.energy_loc}
\mathcal{I}_{\Lambda^*}(w, \R^{n+1}) = [w]^2_{H^1(\R^{n+1}, |y|^a)} + \Lambda_*^2\, \mathcal{H}^n\big(\{(x, 0)\in \R^{n+1} : w(x, 0) > 0 \}\big),
\end{equation}
where $\mathcal{H}^n$ is the $n$-dimensional Hausdorff measure. We can analogously define a localized energy functional in $B\subset \R^{n+1}$ as 
\begin{equation}
\label{eq.energy_loc2}
\mathcal{I}_{\Lambda^*}(w, B) = [w]^2_{H^1(B, |y|^a)} + \Lambda_*^2\, \mathcal{H}^n\big(\{(x, 0)\in B : w(x, 0) > 0 \}\big),
\end{equation}
where now, due to the local nature of the problem, one can really focus only on the set $B$ without intervention from $B^c$. 

We say that $u_*$ is a local minimizer for $\mathcal{I}_{\Lambda_*}$ in $B\subset \R^{n+1}$ if $\mathcal{I}_{\Lambda_*}(u_*, B) \le \mathcal{I}_{\Lambda_*}(w, B)$ for all $w$ such that $u_* = w$ in $B^c$. Similarly, we say that $u_*$ is a global minimizer for $\mathcal{I}_{\Lambda*}$ if it is a local minimizer for all $B\subset \R^{n+1}$. 

We note that, due to the equivalence \eqref{eq.fromdowntoup} and the fact that $a$-harmonic functions are local minimizers of the weighted Dirichlet energy, the $a$-harmonic extension of a global minimizer of $\mathcal{J}_{\Lambda_*}$ is a global minimizer for $\mathcal{I}_{\Lambda}$, when
\begin{equation}
\label{eq.lambdastar}
2 \Lambda^2 = d_s \Lambda_*^2.
\end{equation}
The extra factor 2 appears because in the equivalence \eqref{eq.fromdowntoup} we consider only a half-space $\R^{n+1}_+$. In particular, for $s = \frac12$ (when $a = 0$ and $L_a = \Delta$), $\sqrt{2}\Lambda = \Lambda_*$.

\subsection{The first variation}

The first variation (criticality condition) for the fractional one-phase problem is the following:

\begin{prop}
\label{prop.crs}
Let $u$ be a local minimizer to \eqref{eq.energy}, in the sense \eqref{eq.energy2}-\eqref{eq.loc_min}; or a critical point to \eqref{eq.energy} in $B\subset\R^n$, in the sense \eqref{eq.criticalpoint}. Let $\Omega \coloneqq \{u > 0\}$, and let us suppose that $\partial \Omega \cap B$ is at least $C^{1,\alpha}$. 
Then, $u$ satisfies 
\begin{equation}
\label{eq.OPT_2}
\left\{
\begin{array}{rcll}
\fls u & = & 0 & \quad\textrm{in } \Omega\cap B\\
u & = & 0 & \quad\textrm{in } \Omega^c \cap B \vspace{1mm}\\
\Gamma(1+s) \displaystyle{\frac{u}{d^{s}}} & = & \Lambda& \quad \textrm{on }\partial \Omega\cap B.
\end{array}
\right.
\end{equation}
\end{prop}

\begin{rem}
\label{rem.1}
Even if it is already known, we will prove the previous result for two reasons: on the one hand, we think it is an opportunity to introduce some of the expressions that will be used later on; and on the other hand, we compute the explicit constant for the normal derivative (depending on $\Lambda$). 
The constant appearing in \cite{CRS10} is incorrect, unfortunately, because of a computational mistake in the derivation. (Notice that the constant obtained here is actually simpler than the one in \cite{CRS10}.)
\end{rem}
\begin{rem} 
\label{rem.2}
Our constant for the fractional derivative coincides with the one in \cite[Proposition 3.13]{DS15} in the case $s = \frac12$, where they obtain that the (fractional) normal derivative is $\sqrt{2\pi^{-1}}$ for minimizers of \eqref{eq.energy_loc2} with $\Lambda_* = 1$. Using that from \eqref{eq.lambdastar} $\Lambda = 2^{-\frac12}$ in this case, we see that both results are the same. 
\end{rem}
\begin{rem}
While the previous result is originally proved for minimizers in \cite{CRS10}, we do it also for critical points. Both for us, and for \cite{CRS10}, the proof is the same for minimizers and critical points. We think, however, that it is an opportunity to introduce the competitors that will be used later on for the stable solutions. Also, our proof of the first variation condition does not use the Caffarelli-Silvestre extension of the fractional Laplacian.
\end{rem}

Let us start by a couple of lemmas that will be useful throughout the proof of Proposition~\ref{prop.crs}. 

The first lemma is a simple (but useful) identity involving fractional semi-norms and the fractional Laplacian.

\begin{lem}
\label{lem.form1}
Let $u$ and $v$ be such that $u \equiv v$ in $\R^n\setminus B$. Then
\begin{align*}
\frac{c_{n,s}}{2}\iint_{\R^{2n}\setminus (B^c)^2} \frac{(v(x)-v(y))^2}{|x-y|^{n+2s}}\, dx\, dy -\frac{c_{n,s}}{2}&\iint_{\R^{2n}\setminus (B^c)^2}  \frac{(u(x)-u(y))^2}{|x-y|^{n+2s}}\, dx\, dy  = \\
& = \int_B (v-u) \fls (v+u).
\end{align*}
\end{lem}
\begin{proof}
We are going to use the following identity from \cite{DRV17}:
\[
\langle u, v\rangle_{H^s(B)} = \int_B v\fls u + \int_{B^c} v \mathcal{N}_s u,
\]
where we have denoted by 
\begin{equation}
\label{eq.notation}
\langle u, v\rangle_{H^s(B)} \coloneqq \frac{c_{n,s}}{2} \iint_{\R^{2n}\setminus (B^c)^2} \frac{(u(x)- u(y))(v(x)-v(y))}{|x-y|^{n+2s}}\, dx\, dy 
\end{equation}
and 
\[
\mathcal{N}_s u = c_{n,s}\int_B \frac{u(x) - u(y)}{|x-y|^{n+2s}}\, dy\quad\text{for}\quad x\in B^c.
\]
Thus, we have  (using that $u \equiv v$ in $B^c$)
\begin{align*}
\langle v, v\rangle_{H^s(B)} & = \int_B v\fls v + \int_{B^c} v \mathcal{N}_s v = \int_B v\fls v + \int_{B^c} u \mathcal{N}_s v \\
& = \int_B (v-u)\fls v + \langle v, u\rangle_{H^s(B)}\\
& = \int_B (v-u)\fls v + \int_\Omega v \fls u +\int_{B^c} v\mathcal{N}_s u \\
& = \langle u, u\rangle_{H^s(B)} + \int_B (v-u) \fls (v+u),
\end{align*}
as wanted. 
\end{proof}

The second lemma is the first order expansion of the fractional Laplacian of an $s$-harmonic function outside of the domain. Here, we denote $d(x) = {\rm dist}(x, \partial \Omega)$.

\begin{lem}
\label{lem.constants}
Let $u_+ = \max\{u, 0\}$ and $u_- = -\min\{u, 0\}$. Then
\begin{equation}
\label{eq.constant0}
\fls (x_1)_+^{s} = \bar c_s (x_1)_-^{-s}\quad\text{with}\quad \bar c_s = -\frac{\Gamma(1+s)}{\Gamma(1-s)}.
\end{equation} 

Moreover, if $\Omega\subset \R^n$ is $C^{1,\alpha}$, with $0\in \partial \Omega$, and $u$ solves 
\[
\left\{
\begin{array}{rcll}
\fls u & = & 0 & \quad\textrm{in } B_1\cap \Omega\\
u & = & 0 & \quad\textrm{in } B_1\setminus \Omega,
\end{array}
\right.
\]
then there exists $U_0\in\R$ such that
\[
u(x) = U_0 d^s(x) + O(|x|^{s+\alpha})\quad\text{in}\quad \Omega,
\]
and 
\[
\fls u(x)  = \bar c_s U_0 d^{-s}(x) + O(|x|^{\alpha})\,d^{-s}(x) \quad\text{in}\quad B_1\setminus \Omega,
\]
where $\bar c_s$ is given by \eqref{eq.constant0}. 
\end{lem}
\begin{proof}
Notice that 
\begin{align*}
\fls v(x) & = -\frac{c_{n,s}}{2} \int_{\R^n} \frac{v(x+y)+v(x-y)-2v(x)}{|y|^{n+2s}} dy\\
& = -\frac{c_{n,s}}{2}\int_{\mathbb{S}^{n-1}} \frac12\int_{-\infty}^\infty \frac{v(x+r\theta)+v(x-r\theta)-2v(x)}{r^{1+2s}} dr\, d\theta.
\end{align*}

Suppose now that $v(x) = v(x_n)$ so 
\begin{align*}
\fls v (x) & = -\frac{c_{n,s}}{2}\int_{\mathbb{S}^{n-1}} \frac12\int_{-\infty}^\infty \frac{v(x_n+r\theta_n)+v(x_n-r\theta_n)-2v(x_n)}{r^{1+2s}} dr\, d\theta\\
& = \frac{c_{n,s}}{2c_{1,s}}\int_{\mathbb{S}^{n-1}}\fls_\R v(x_n + \theta_n\cdot)\, d\theta = \frac{c_{n,s}}{2c_{1,s}} \int_{\mathbb{S}^{n-1}}|\theta_n|^{2s}\, d\theta \fls_\R v(x_n),
\end{align*}
where $\fls_\R$ denotes the one-dimensional fractional Laplacian. 
Now notice that 
\begin{align*}
 \int_{\mathbb{S}^{n-1}}|\theta_n|^{2s}\, d\theta& = 2|\mathbb{S}_{n-2}|\int_0^{\pi/2} (\sin(\theta))^{2s}(\cos(\theta))^{n-2}\, d\theta \\
 & = |\mathbb{S}_{n-2}| \int_0^1 t^{s-\frac12} (1-t)^{\frac{n-3}{2}}\, dt 
 = |\mathbb{S}_{n-2}|\frac{\Gamma\left(\frac{n-1}{2}\right) \Gamma\left(s+\frac12\right)}{\Gamma\left(\frac{n}{2} + s\right)}.
 \end{align*}

Notice also that 
\[
\frac{c_{n,s}}{c_{1,s}} = \frac{\Gamma\left(\frac{n}{2}+s\right)}{\pi^{\frac{n-1}{2}}\Gamma\left(s+\frac12\right)},\qquad |\mathbb{S}^{n-2}| = \frac{2\pi^{\frac{n-1}{2}}}{\Gamma\left(\frac{n-1}{2}\right)}.
\]
Putting it all together, 
\[
\fls v (x) = \fls_\R v(x_n).
\]
Now, if we denote $v(t) = t^s_+$ we can compute
\[
\fls (x_1)^s_+ = (\fls_\R v)(x_1) = \bar c_s (x_1)_-^{-s}.
\]
and $\bar c_s$ is simply given by $(\fls_\R v)(-1)$,
\[
\bar c_s = (\fls_\R v)(-1) = c_{1,s}\int_{0}^\infty\frac{-t^s}{(1+t)^{1+2s}}dt.
\]
Let us compute with the change of variable $t\mapsto z = \frac{1}{1+t}$
\[
\int_0^\infty\frac{t^s}{(1+t)^{1+2s}}dt= \int_0^1 (1-z)^s z^{s-1}\, dz= \frac{\Gamma(s+1)\Gamma(s) }{\Gamma(2s+1)}. 
\]
Thus, plugging in the value of $c_{1,s}$, 
\[
\bar c_s  = -c_{1,s} \frac{\Gamma(s+1)\Gamma(s) }{\Gamma(2s+1)} = -\frac{\Gamma(1+s)}{\Gamma(1-s)},
\]
where we are also using the duplication formula for the gamma function,
\[
\Gamma(2z) = \Gamma(z) \Gamma\big(z+ {\textstyle\frac12}\big) 2^{2z-1}\pi^{-\frac12}. 
\]

Let now $\Omega\subset \R^n$ be any $C^{1,\alpha}$ domain.
Then, by \cite{RS17}, we know that $u/d^s\in C^\alpha(\overline\Omega\cap B_{1/2})$.
Thus, if we define $U_0\coloneqq(u/d^s)(0)$ we then have
\[u(x)=(u/d^s)(x)\,d^s(x) = \big(U_0+O(|x|^{\alpha})\big)\,d^s(x) = U_0d^s(x)+O(|x|^{\alpha})\,d^s(x).\]
Since $d^s(x)=O(|x|^s)$, this proves the expansion for $u$ near $0$.
Notice that such expansion also implies that
\[u(x)= U_0(x\cdot \nu)_+^s+O(|x|^{s+\alpha}),\]
where $\nu$ is the normal vector to $\partial\Omega$ at the origin.

Now, thanks to the previous expansion, we find that for $x=-t\nu\in \Omega^c$, with $t>0$,
\[(-\Delta)^s u(x) = U_0(-\Delta)^s (x\cdot \nu)_+^s+O(t^{-s+\alpha}) = \bar c_s U_0 t^{-s}+O(t^{-s+\alpha}),\]
where we used the first part of the Lemma (the case $n=1$).
Since this can be done not only at the origin but at every boundary point $z\in \partial\Omega$, we deduce that for every $x=z-t\nu_z\in \Omega^c$
\[(-\Delta)^s u(x) = \bar c_s U_z t^{-s}+O(t^{-s+\alpha}).\]
For each $x\in\Omega^c$ we can choose $z\in\partial\Omega$ such that $|x-z|=d(x)$, and then we deduce
\[(-\Delta)^s u(x) = \bar c_s U_z d^{-s}(x)+O(d^{-s+\alpha}(x)).\]
Since $U_z=U_0+O(|z|^\alpha) = U_0+O(|x|^\alpha)$, we finally get
\[(-\Delta)^s u(x) = \bar c_s U_0 d^{-s}(x)+O(|x|^\alpha)\,d^{-s}(x),\]
as wanted.
\end{proof}

We now have all the ingredients to give the:

\begin{proof}[Proof of Proposition~\ref{prop.crs}]
We divide the proof into two steps. In the first step we build the right competitors, and in the second step we use them to deduce the desired properties. 
\\[0.1cm]
{\bf Step 1.} According to the definition of critical point, \eqref{eq.criticalpoint}, we need to consider competitors of the form $u_\eps(x) = u(x+\eps\Psi(x))$ for smooth domain variations variations $\Psi$ supported in $B\subset \R^n$, in order to compute the expansion of the energy $\mathcal{J}_{\Lambda, B}(u_\eps)$ in $\eps$, for all $B\subset \R^n$. 
We notice two important properties: on the one hand, for $\eps$ small, $x+\eps\Psi$ is a diffeomorphism; on the other hand, we can always lower the energy by making $u_\eps$ $s$-harmonic in its positivity region. With these two properties, we have that it is enough to perform smooth $\eps$-deformations of the contact set $\{u = 0\}$ while keeping the positivity region $s$-harmonic. 

Let us now show the theorem by building a competitor $v_\eps$ in $B$. Let us consider a fixed function $f\in C^\infty_c(\partial \Omega)$, $f\ge 0$, and without loss of generality let us assume that ${\rm supp}(f) \subset B\cap \partial \Omega$ (the case for general $f$ without sign restriction is discussed at the end). Let us denote, for $x\in \R^n$, $z = \pi_\Omega(x)$ for any $z\in \partial \Omega$ such that $d(x) = {\rm dist}(x, z)$. Notice that if $d(x)$ is small enough, since the domain $\Omega$ is smooth, this is uniquely determined. (In the formulation above, we are considering the surface $({\rm Id} + \eps \Psi)(\partial\Omega)$ as a function over $\partial\Omega$.) 

Let us define the domain $\Omega_\eps$ for $\eps > 0$ as
\[
\Omega_\eps = \big\{x\in \Omega : d(x) \ge \eps f(\pi_\Omega(x))\big\}.
\]
Namely, we extend the complementary of $\Omega^c$ by $\eps f(z)$ for each $z\in \partial \Omega$ (see Figure~\ref{fig.0}). Let us denote by $v_\eps$ our competitor, which is the solution to  
\[
\left\{
\begin{array}{rcll}
\fls v_\eps & = & 0 & \quad\textrm{in } (\Omega_\eps\cap B)\\
v_\eps & = & u & \quad\textrm{in } B^c\\
v_\eps & = & 0& \quad \textrm{in }B \setminus \Omega_\eps.
\end{array}
\right.
\]

\begin{figure}
\centering
\includegraphics{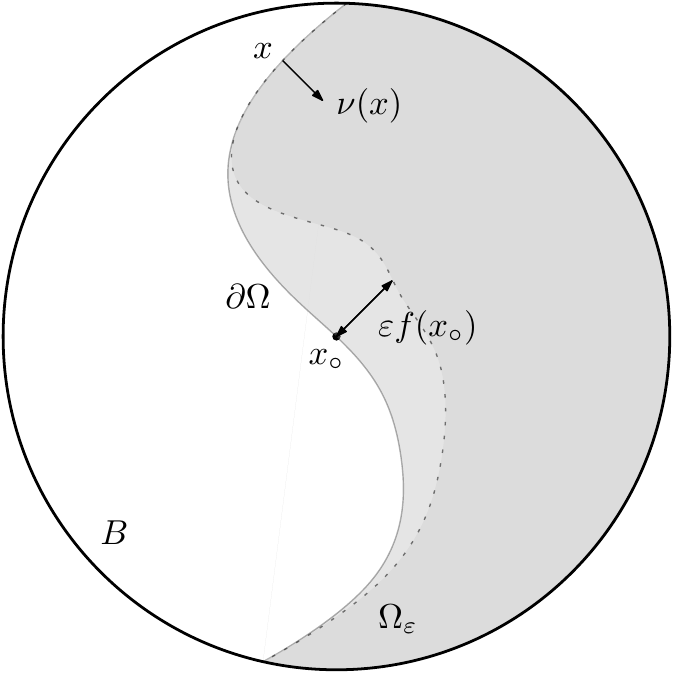}
\caption{The domains $B$, $\Omega$, and $\Omega_\eps$.}
\label{fig.0}
\end{figure}

Now, we have that $\mathcal{J}_{\Lambda, B}(v_\eps) \le \mathcal{J}_{\Lambda, B}(u_\eps) $ and they coincide for $\eps = 0$; therefore, the expansion in $\eps$ coincides at order 0 and 1 (and second derivatives are ordered). 
\\[0.1cm]
{\bf Step 2.}
Let us now compute the first order term in the expansion, $\mathcal{J}_{\Lambda, B}(v_\eps) - \mathcal{J}_{\Lambda, B}(u)$. We have, using the notation from \eqref{eq.notation}, and Lemma~\ref{lem.form1}
\begin{align*}
\mathcal{J}_{\Lambda, B}(v_\eps) - \mathcal{J}_{\Lambda, B}(u) & = \langle v_\eps, v_\eps\rangle_{H^s(B)} - \Lambda \langle u, u\rangle_{H^s(B)} -\Lambda^2 |(\Omega\setminus \Omega_\eps)\cap B|\\
& = \int_B (v_\eps - u) \fls (v_\eps+u) -\Lambda^2 |\Theta_\eps|,
\end{align*}
where we have denoted $\Theta_\eps \coloneqq (\Omega\setminus \Omega_\eps)\cap B$.
Notice now that, using $u\fls u = v_\eps \fls v_\eps = v_\eps\fls u = 0$ in $B$, and $u\fls v_\eps = 0$ in $\Omega^c\cup \Omega_\eps$, 
\begin{equation}
\label{eq.20}
\mathcal{J}_{\Lambda, B}(v_\eps) - \mathcal{J}_{\Lambda, B}(u) =-\int_{\Theta_\eps} u \fls v_\eps - \Lambda^2|\Theta_\eps|.
\end{equation}
Now, on the one hand
\begin{equation}
\label{eq.21}
|\Theta_\eps| = \eps \int_{\partial \Omega} f + o(\eps). 
\end{equation}

On the other hand, let us parametrize the points in $\Theta_\eps$ as $z+t\nu(z)$, where $z\in \partial \Omega$, $t > 0$, and $\nu(z)$ is the unit inward normal to $\partial \Omega$ at $z$. Notice also that the volume element here is $dV = (1+O(t)) \, d\sigma(z) \, dt$ where $d\sigma$ is the area element on $\partial \Omega$. Using this parametrization we can expand $u$ at $z\in \partial \Omega$ as 
\begin{equation}
\label{eq.11}
u(z+t\nu(z) ) = \frac{u}{d^s}(z)\, t^s + o(t^{s}),
\end{equation}
where 
\[
\frac{u}{d^s} (z) = \lim_{\tau\downarrow 0} \frac{u(z+\tau \nu(z))}{\tau^s}. 
\]
is the fractional normal derivative. Similarly, we can expand $v_\eps$ around $x_\circ = z+\eps f(z)\nu(z) \in \partial\Omega_\eps$ as 
\[
v_\eps(x) = \frac{v_\eps}{d_\eps^s}(x_\circ)\,d_\eps^s(x) + o(|x-x_\circ|^s)\quad\text{in}\quad  \Omega_\eps,
\]
where 
\[d_\eps(x)\coloneqq{\rm dist}(x, \partial \Omega_\eps).\]
By Lemma~\ref{lem.constants}, 
\[
\fls v_\eps = \bar c_s \,\frac{v_\eps}{d_\eps^s}(x_\circ)\,d_\eps^{-s} + o(d_\eps^{-s})\quad\text{in}\quad \Omega_\eps^c,
\]
where $\bar c_s$ is given by \eqref{eq.constants}. 
Notice now that, if $x = z+t\nu(z) \in \Theta_\eps$ ($0 < t < \eps f(z)$), 
\[
d_\eps^2(x)\big(1+  [\eps |\nabla f|+o(\eps)]^2\big)= (\eps f(z) - t)^2 
\]
so 
\begin{equation}
\label{eq.dist_asimpt}
d_\eps(x) = (\eps f(z) - t)(1 + o(\eps)).
\end{equation}
That is, 
\[
\fls v_\eps (z+t\nu(z) ) = \bar c_s \,\frac{v_\eps}{d_\eps^s}(x_\circ)\, \big((\eps f(z) - t)(1 + o(\eps))\big)^{-s} + o(t^{-s}). 
\]
Similarly, since $v_\eps \to u$ and $(v_\eps/d_\eps^s)(x_\circ) = (u/d^s)(z) + o(1)$ as $\eps\downarrow 0$, we have
\begin{equation}
\label{eq.12}
\fls v_\eps (z+t\nu(z) ) = \bar c_s \,\frac{u}{d^s}(z) (\eps f(z) - t)^{-s}(1+o(1)) + o(t^{-s}). 
\end{equation}
Putting \eqref{eq.11} and \eqref{eq.12}  together, 
\begin{align*}
\int_{\Theta_\eps} & u \fls v_\eps = \\
& = \int_{\partial \Omega}\int_0^{\eps f(z)}  t^s \bar c_s \left( \left(\frac{u}{d^s}(z)\right)^2 (\eps f(z) - t)^{-s}[1+o(1)] +  o(t^{-s})\right) \, dt \, d\sigma(z)\\
& =  \bar c_s \int_{\partial \Omega} \left(\frac{u}{d^s}(z)\right)^2  \int_0^{\eps f(z)}  t^s (\eps f(z) - t)^{-s}[1+o(1)] \, dt \, d\sigma(z) + o(\eps)\\
&  = \eps \bar c_s \int_{\partial \Omega}f(z) \left(\frac{u}{d^s}(z)\right)^2  \int_0^{1}  t^s (1 - t)^{-s} \, dt \, d\sigma(z) + o(\eps).
\end{align*}
Now notice that 
\[
\int_0^{1}  t^s (1 - t)^{-s} \, dt = \Gamma\left(1+s\right)\Gamma\left(1-s\right),
\]
so that
\begin{equation}
\label{eq.22}
\int_{\Theta_\eps} u \fls v_\eps = -\eps  \Gamma\left(1+s\right)^2 \int_{\partial \Omega}f(z) \left(\frac{u}{d^s}(z)\right)^2  \, d\sigma(z) + o(\eps).
\end{equation}
Combining \eqref{eq.20}, \eqref{eq.21} and \eqref{eq.22}, 
\[
\mathcal{J}_{\Lambda, B}(v_\eps) - \mathcal{J}_{\Lambda, B}(u) =\eps  \int_{\partial \Omega}f(z) \left[  \Gamma\left(1+s\right)^2\left(\frac{u}{d^s}(z)\right)^2-\Lambda^2\right]  \, d\sigma(z) + o(\eps).
\]

By \eqref{eq.criticalpoint}), we deduce
\[
 \int_{\partial \Omega}f(z) \left[ \Gamma\left(1+s\right)^2\left(\frac{u}{d^s}(z)\right)^2-\Lambda^2 \right]  \, d\sigma(z) = 0
\]
for all $0\leq f\in C^\infty_c(B\cap \partial \Omega)$, and hence
\[
\Gamma(1+s)\,\frac{u}{d^s}(z) = \Lambda
\]
for all $z\in \partial\Omega$, as we wanted to see. 
\end{proof}

\section{The stability condition}
\label{sec.3}

In this section we will prove the following result, which is our first characterisation of the stability condition.

\begin{prop}
\label{prop.main}
Let $u$ be a local minimizer to \eqref{eq.energy}, in the sense \eqref{eq.energy2}-\eqref{eq.loc_min}; or a stable critical point to \eqref{eq.energy} in $B\subset\R^n$, in the sense \eqref{eq.criticalpoint}-\eqref{eq.stablesolution}. Let $\Omega \coloneqq \{u > 0\}$, and let us suppose that $\partial \Omega \cap B$ is  at least $C^{2,\alpha}$. 
Then, $u$ satisfies
\begin{equation}
\label{eq.condFinB}
-\frac{1+s}{s}\cdot \frac{\Gamma(1+s)}{\Lambda} \int_{\partial \Omega} \partial_\nu\left(\frac{u}{d^s}\right) f^2 \, d\sigma \le -\int_{\partial \Omega} f \partial_\nu\left(\frac{F}{d^{s-1}}\right)d\sigma,
\end{equation}
for all $f\in C^\infty_c(\partial \Omega\cap B)$; where $\nu$ is the unit inward normal vector on $\partial\Omega$, and $F$ is the solution to 
\begin{equation}
\label{eq.FinB}
\left\{
\begin{array}{rcll}
\fls F & = & 0 & \quad\textrm{in } (\Omega \cap B)\\
F & = & 0 & \quad\textrm{in } (\Omega\cap B)^c\\
\frac{F}{d^{s-1}} & = &  f& \quad \textrm{on }\partial \Omega,
\end{array}
\right.
\end{equation}
which is a possible analogy of the Dirichlet problem for the fractional Laplacian.
\end{prop}

Problem~\eqref{eq.FinB} is a singular boundary value problem for the fractional Laplacian, and it is well posed for any boundary value $f\in C^0(\partial\Omega)$. 
It was first studied in \cite{Aba15, Gru15}.

To prove the result we will need the following lemma, which is a higher order version of Lemma~\ref{lem.constants} above. 

\begin{lem}
\label{lem.form2}
Let $\Omega\subset \R^n$ be any $C^{2,\alpha}$ domain, with $0\in \partial \Omega$, and let $u$ solve 
\[
\left\{
\begin{array}{rcll}
\fls u & = & 0 & \quad\textrm{in } B_1\cap \Omega\\
u & = & 0 & \quad\textrm{in } B_1\setminus \Omega.
\end{array}
\right.
\]
Let us define, for $x\in \Omega$,
\[
\eta(x) = \frac{u(x)}{d^s(x)},
\]
where  $d(x) = {\rm dist}(x, \partial \Omega)$. 
Then, $\eta\in C^{1,\alpha}(\overline\Omega \cap B_{1/2})$ and we can express the expansion of $u$ around $0$ as 
\[
u(x) = \eta(0) d^s(x) +(\nabla \eta(0) \cdot x )d^{s}(x) + o(|x|)\,d^{s}(x)\quad\textrm{ in } \Omega.
\]
Moreover, we can expand $\fls u$ at 0 as
\begin{align*}
\fls u(x) & = \bar c_s \eta(0) d^{-s}(x) + \bar c_* \eta(0) H d^{1-s}(x) + \bar c_{1+s}\partial_\nu \eta(0) d^{1-s}(x)\, + \\
& \qquad\qquad\qquad\qquad\quad +  \bar c_s (\nabla_{\tau} \eta(0)\cdot x) d^{-s}(x)  + o(|x|)\,d^{-s}(x) \quad\textrm{ in } B\setminus \Omega
\end{align*}
around 0, where $H$ denotes the mean curvature of $\partial\Omega$ at $0$ (with respect to $\{u = 0\}$), $\nu$ denotes the unit inward normal derivative to $\partial\Omega$ at 0, and $\nabla_\tau$ denotes the tangential component of the gradient at 0. The constants are given by 
\begin{equation}
\label{eq.constants}
\bar c_s = -\frac{\Gamma(1+s)}{\Gamma(1-s)}, \qquad \bar c_* = \frac{s\Gamma(1+s)}{\Gamma(2-s)},\qquad  \bar c_{1+s} = \frac{\Gamma(2+s)}{\Gamma(2-s)}.
\end{equation}
\end{lem}

\begin{proof}
Notice that, by \cite[Theorem 1.4]{AR20}, $\eta\in C^{1,\alpha}(\overline{\Omega}\cap B_{1/2})$, so 
\[
u(x) = \eta(0) d^s(x) + (\nabla \eta(0) \cdot x) d^s(x) + o(|x|)\,d^s.
\]
Let us now divide the proof into three steps. In the first step we perform some computations that will be useful in the following ones.
\\[0.1cm]
{\bf Step 1.} Let us denote by $\delta$ the signed distance to $\Omega$, so that $\delta > 0$ in $\Omega$ and $\delta \le  0$ in $\Omega^c$. We also consider the extension problem: by taking coordinates in $\R^{n+1}$, $(x, y)$ with $x\in \R^n$ and $y\in \R_+$ (see \cite{CS07} and subsection~\ref{ss.local}).

Let us also define by $r = r(X)$ the distance to $\partial\Omega$ in $\R^n\times \R_+$, namely
\[
r = (\delta^2+y^2)^{\frac12}.
\] 
Notice that we also have, on $\partial \Omega$, $\Delta \delta = \Delta_x \delta = H$, where $H = H(x)$ is the mean curvature of the level set of $\{\delta = t\}$ with respect to $\{\delta \le t\}$ (in particular, when $\delta$ goes to zero, $H$ is the mean curvature of $\partial \Omega$ with respect to $\{u = 0\}$). 

Let us consider the operator 
\[
L_a v \coloneqq {\rm div}(y^a\nabla v),\qquad \text{where }\quad a = 1-2s. 
\]
We define also 
\[
\mathcal{U}_s \coloneqq (r+\delta)^s.
\]
For convenience to the reader, we collect some identities that are useful in the computations below.
\[
\begin{array}{ll}
L_a \delta =  H y^a & \quad L_a(F^\alpha) = \alpha F^{\alpha-1} L_aF + \alpha(\alpha-1)F^{\alpha-2}|\nabla F|^2 y^a\\
\nabla r = \frac{1}{r}(\delta\nabla \delta, y) & \quad L_a \mathcal{U}_s  = s H\frac{\mathcal{U}_s}{r} y^a\\
|\nabla \delta|^2 = |\nabla r|^2 = 1 & \quad \nabla \mathcal{U}_s = s (r+\delta)^{s-1}(\nabla \delta + \nabla r)  \\
r L_a r = (H\delta +1 +a)y^a&  \quad \nabla \delta \cdot \nabla r = \frac{\delta}{r}\\
|\nabla (r+\delta)|^2 = 2\left(1+\frac{\delta}{r}\right)&  \quad \nabla \mathcal{U}_s \cdot \nabla \delta = \nabla \mathcal{U}_s \cdot \nabla r = s \frac{\mathcal{U}_s}{r}.\\
 & 
\end{array}
\]
From here, we can also compute 
\begin{align*}
L_a(\mathcal{U}_s \delta ) & = sH \mathcal{U}_s \frac{\delta}{r} y^a + H\mathcal{U}_s y^a +2s\frac{\mathcal{U}_s}{r} y^a\\
L_a(\mathcal{U}_s r) & = H\mathcal{U}_s \frac{\delta}{r} y^a + sH \mathcal{U}_s y^a + 2\frac{\mathcal{U}_s}{r} y^a. 
\end{align*}
In particular, 
\[
L_a (\mathcal{U}_s(sr - \delta)) = -(1-s^2)H y^a \mathcal{U}_s,
\]
and, if we denote 
\[
\mathcal{V}_s \coloneqq \mathcal{U}_s -\frac{H}{2} \mathcal{U}_s \delta + \frac{1}{1-s}\frac{H}{2} \mathcal{U}_s (\delta-sr) = \mathcal{U}_s + \frac{H}{2} \mathcal{U}_s \frac{s}{1-s} (\delta-r),
\]
we have
\[
L_a \mathcal{V}_s = s \frac{H^2}{2} \frac{\mathcal{U}_s }{r} (r-\delta) y^a.
\]
Moreover, by considering the region where $\delta = -t$ for some $t > 0$ (i.e., in $\Omega^c$)
\begin{equation}
\label{eq.comp1}
y^a \partial_y \mathcal{U}_s = \frac{s}{r (r+t)^{s-1}} \to s2^{1-s}t^{-s}\quad\text{as}\quad y \downarrow 0,
\end{equation}
so that 
\begin{equation}
\label{eq.comp2}
y^a\partial_y \mathcal{V}_s \to s2^{1-s} t^{-s} - H\frac{s^2}{1-s} 2^{1-s} t^{1-s}\quad\text{as}\quad y \downarrow 0. 
\end{equation}
We also recall that, if $U(x, y)$ is $a$-harmonic ($L_a U = 0$ in $\{y > 0\}$), then 
\[
\fls U(\cdot, 0) = -d_s \lim_{y \downarrow 0} y^a \partial_y U(x, y),\qquad \text{with}\quad  d_s = 2^{2s-1}\frac{\Gamma(s)}{\Gamma(1-s)}.
\]
(See \cite{ST10}.)

Moreover, if $\nabla \delta(0) = \boldsymbol{e}_n$ and $x = (x', x_n)\in \R^{n-1}\times \R$, given $v\in \R^{n-1}$,
\[
L_a \left[\mathcal{U}_s v \cdot x' \right]  =  s H \frac{U_a}{r} (v \cdot x') y^a  +2\nabla_{x'} \mathcal{U}_s \cdot v y^a,
 \]
 where $\nabla_{x'}$ denotes the gradient in the first $n-1$ coordinates.
 We can compute $\nabla_{x'} \mathcal{U}_s$ as
  \[
 \nabla_{x'} \mathcal{U}_s = s\, \mathcal{U}_s \frac{1}{r}  \nabla_{x'} \delta .
 \]
 Notice that, since $\nabla_{x'} \delta(0) = 0$ and $\Omega$ is smooth, so
\[
 L_a \left[\mathcal{U}_s v\cdot x'\right] =  \frac{\mathcal{U}_s}{r} y^a O(|x|),
\]
 where we have used that $\nabla_{x'} \delta(x) = O(|x|)$.
\\[0.1cm]
{\bf Step 2.} By \cite[Proposition 4.1]{JN17}-\cite[Theorem 3.1]{DS15b}, we can express  the $a$-harmonic extension of $u$ towards $\{y > 0\}$ as
\[
u(x, y) = 2^{-s}\U_s \left(P(x, r) + o(|(x, r)|^{k+\frac12})\right) 
\]
for a polynomial $P$ in $x$ and $r$ of degree $k$. Expanding it in terms of $\delta$, and denoting by $\nu$ the unit inward normal vector to $\partial\Omega$ at 0,  we have that 
\[
u(x, y) = 2^{-s}\U_s \left(a_0 +a_1 \delta + a_2 r + A_3\cdot (x-(x\cdot\nu)\nu) + o(|(x, r)|^{\frac32})\right),
\] 
where we are using that, at first order, $(x\cdot\nu)\nu$ is like $\delta$ and $(x-(x\cdot\nu)\nu)$ corresponds to the tangential space to $\partial\Omega$ at 0. 

Without loss of generality, let us assume $\nu = \boldsymbol{e}_n$ and let us denote points in $\R^n$ as $x = (x',x_n)\in \R^{n-1}\times \R$ (notice that, then, $\delta(x) = x_n + o(|x|)$). Then
\[
u(x, y) = 2^{-s}\U_s \left(a_0 +a_1 \delta + a_2 r + A_3\cdot x' + o(|(x, r)|^{\frac32})\right),
\] 
For future convenience, we re-express it as
\[
u(x, y) = 2^{-s}\U_s \left(a_0 +\tilde a_1 (\delta-sr) + \tilde a_2 (\delta-r) + A_3\cdot x' + \dots\right),
\] 
so that, from the definition of $\eta(x)$,
\begin{equation}
\label{eq.from1}
\eta(x) = \frac{u(x ,0)}{\delta^s(x)} = a_0+\tilde a_1 (1-s)  \delta +  A_3\cdot x'+o(|x|). 
\end{equation}

By expanding $\eta(x)$ at $0\in \partial\Omega$ we also have 
\[
\eta(x) = \eta(0) + \nabla \eta(0) \cdot x + o(|x|) = \eta(0) + \nabla \eta(0) \cdot x'  + \partial_n \eta(0) x_n + o(|x|).
\]
Comparing it with \eqref{eq.from1}, and using that $\delta = x_n + o(|x|)$, we reach  
\[
a_0 = \eta(0),\quad \tilde a_1(1-s) = \partial_n \eta(0),\quad A_3 = \nabla_{x'} \eta(0).
\]

Thus
\[
\frac{u(x, y)}{2^{-s}\mathcal{U}_s} = \eta(0) + \frac{\partial_n\eta(0)}{1-s} (\delta-sr) +\tilde a_2 (\delta-r) +\nabla_{x'} \eta(0)\cdot x' + o(|(x, r)|). 
\]  
We now use the fact that $u(x, y)$ is the $a$-harmonic expansion of $u$, and thus $L_a u(x, y) = 0$ for $y > 0$. That is, 
\begin{align*}
y^{-a}L_a \bigg(\mathcal{U}_s\bigg[\eta(0) + \frac{\partial_n\eta(0)}{1-s} (\delta-sr) & +\nabla_{x'} \eta(0)\cdot x'\bigg]\bigg)\,+ \\
& + \tilde a_2 y^{-a}L_a \bigg(\mathcal{U}_s(\delta-r) \bigg) = o(|(x, r)|^{s-1}).
\end{align*}
This implies some cancellations that give the only value of $\tilde a_2$ possible (notice that we cannot proceed as before, since the term multiplying $\tilde a_2$ cannot be seen from $\{y = 0\}$).

From the calculations in the first step, 
\[
L_a(\mathcal{U}_s(\delta-sr) ) = C\mathcal{U}_s y^a  = Cy^aO(|(x, r)|^s) = y^a o(|(x, r)|^{s-1})
\]
and 
\[
L_a (\mathcal{U}_s \nabla_{x'}\eta(0)\cdot x') = \frac{\mathcal{U}_s}{r} y^a O(|x|) = y^a O(|(x, r)|^s) = y^a o(|(x, r)|^{s-1})
\]
are already ``almost'' $a$-harmonic, so we have to impose 
\begin{equation}
\label{eq.impose}
y^{-a}L_a \bigg(\mathcal{U}_s\bigg[\eta(0)  +\nabla_{x'} 	\eta(0)\cdot x'+\tilde a_2 (\delta-r)\bigg]\bigg)  = o(|(x, r)|^{s-1}).
\end{equation}

 Therefore, from \eqref{eq.impose} 
 \[
y^{-a}L_a \bigg(\mathcal{U}_s\bigg[\eta(0)  +\tilde a_2 (\delta-r)\bigg]\bigg)  = o(|(x, r)|^{s-1}).
 \]
 Notice that, in Step 1, we had $y^{-a}L_a\mathcal{V}_s = o(|(x, r)|^{s-1})$, so we must have 
 \[
\frac{\tilde a_2}{\eta(0)} = \frac{H}{2}\frac{s}{1-s}.
 \]
 In all, the expansion of $u$ at 0 must be
\[
\frac{u(x, y)}{2^{-s}\mathcal{U}_s} = \eta(0) + \frac{\partial_n\eta(0)}{1-s} (\delta-sr) + \eta(0)\frac{H}{2}\frac{s}{1-s} (\delta-r) + \nabla \eta(0) \cdot x' + o(|(x, r)|^{\frac32}).
\] 
Alternatively,
\[
2^s u(x, y) =  \eta(0) \mathcal{V}_s   + \frac{\partial_n\eta(0)}{1-s} \mathcal{U}_s (\delta-sr) + \mathcal{U}_s \nabla \eta(0) \cdot x' + \U_s o(|(x, r)|^{\frac32}),
\]
and this expression will allow us to compute the fractional Laplacian. 
\\[0.1cm]
{\bf Step 3.} Now, the fractional Laplacian $\fls u(x)$ can be computed as 
\[
\fls u (x) = - d_s \lim_{y \downarrow 0} y^a\partial_y u(x, y). 
\]
As in Step 1, we use that $\delta = -t$ for some $t > 0$ (we are the area where $\{u = 0\}$, since otherwise we already know by assumption that $u$ has vanishing fractional Laplacian), and we also recall the computations \eqref{eq.comp1} and \eqref{eq.comp2}. We can then compute
\[
y^a\partial_y \left\{ \mathcal{U}_s (-t-sr) \right\} = -s\frac{t+sr}{t(r+t)^{s-1}}-s \mathcal{U}_s \frac{y^{1+a}}{r}\to - s(1+s)2^{1-s}t^{1-s}\text{ as $y\downarrow 0$}. 
\]
Putting it all together, and from the expansion of $u$ we have 
\begin{align*}
-2^{-a}d_s^{-1}\fls u(x) & = \eta(0) \left[s t^{-s} - H\frac{s^2}{1-s}  t^{1-s}\right]-\frac{\partial_n \eta(0)}{1-s}s(1+s)t^{1-s} \,+\\
 & \hspace{50mm} + s t^{-s}\nabla \eta(0) \cdot x' + o(|x|t^{-s}).
\end{align*}
Substituting the corresponding values we obtain the desired result. 
\end{proof}

Let us now give the proof of the stability condition:

\begin{proof}[Proof of Proposition~\ref{prop.main}]
We divide the proof into four steps. Without loss of generality we will assume that $\Lambda = \Gamma(1+s)$ so that, by Proposition~\ref{prop.crs}, $u/{d^s} \equiv 1$ on $\partial \Omega$. 
\\[0.1cm]
{\bf Step 1.}
Let us assume $f \ge 0$, and let us build a competitor, as in the proof of Proposition~\ref{prop.crs}. We will take general $f$ in the last step. 

That is, let us consider a fixed function $f\in C^\infty_c(\partial \Omega)$, $f\ge 0$, and we assume that ${\rm supp}(f) \subset B\cap \partial \Omega$. We recall the definition of $\Omega_\eps$ and $v_\eps$ as in Proposition~\ref{prop.crs}: for $x\in \R^n$, we denote $z = \pi_\Omega(x)$ for any $z\in \partial \Omega$ such that $d(x) = {\rm dist}(x, z)$. We define the domain $\Omega_\eps$ as
\[
\Omega_\eps = \big\{x\in \Omega : d(x) \ge \eps f(\pi_\Omega(x))\big\}.
\]
Denote by $v_\eps$ our competitor, which is the solution of  
\[
\left\{
\begin{array}{rcll}
\fls v_\eps & = & 0 & \quad\textrm{in } (\Omega_\eps\cap B)\\
v_\eps & = & u & \quad\textrm{in } B^c\\
v_\eps & = & 0& \quad \textrm{in }B \setminus \Omega_\eps.
\end{array}
\right.
\]
(Recall Figure~\ref{fig.0}.)

We now define the function 
\[
F_\eps = \frac{u-v_\eps}{\eps},\qquad F_s = \lim_{\eps\downarrow 0} F_\eps. 
\]

Notice that $F_s$ satisfies $F_s = 0$ in $(\Omega \cap B)^c$, and $\fls F_s = 0$ in $\Omega\cap B$. Let us now see that $F_s$ is well defined in the interior of $\Omega$. In particular, we will show that $F_\eps$ is smooth at any $x_\circ\in \Omega$, with estimates independent of $\eps$. 

Let $x_\circ\in \Omega$, $t_\circ \coloneqq {\rm dist}(x_\circ, \partial\Omega) > 0$. Let us denote $w_\eps = u-v_\eps$. Then, $\fls w_\eps = 0$ in $\Omega_\eps \cap B$ and ${\rm dist}(x_\circ, \Omega_\eps) > \frac12 t_\circ$ for $\eps$ small enough. Let us denote by $P_\eps(x, y): (\Omega^\eps\cap B) \times (\Omega_\eps\cap B)^c\to \R$ the Poisson kernel associated to the (smooth) domain $\Omega_\eps$. That is, if $w$ is such that
\[
\left\{
\begin{array}{rcll}
\fls w & = & 0 & \quad\textrm{in } (\Omega_\eps\cap B)\\
w & = & h & \quad\textrm{in } (\Omega_\eps\cap B)^c,
\end{array}
\right.
\]
then 
\[
w(x) = \int_{(\Omega_\eps\cap B)^c} P_\eps(x, y) h(y)\, dy. 
\]
For our function $w_\eps$, we have that $w_\eps  =0$ in $\Omega^c$, and thus
\begin{equation}
\label{eq.hPoisson}
w_\eps(x_\circ) = \int_{\Theta_\eps} P_\eps(x_\circ, y) u(y)\, dy. 
\end{equation}
where $\Theta_\eps = (\Omega\setminus \Omega_\eps)\cap B$. From the growth condition of $u$ at the boundary we know that $u \sim d^s$ on $\Theta_\eps$. On the other hand, we have estimates for the Poisson kernel (see \cite[Theorem 1.5]{CS98})
\begin{equation}
\label{eq.Poisson}
P_\eps(x, y) \le C \frac{d_\eps^s(x)}{d_\eps^s(y) (1+d_\eps(y))^s} \frac{1}{|x-y|^n} 
\end{equation}
where $d_\eps(z) = {\rm dist}(z, \partial\Omega_\eps)$, and the constant $C$ depends only on $s$ and the regularity of $\Omega_\eps$. From the smoothness of $\Omega$ and $f$, we can take a uniform $C$ in $\eps$.

Finally, combining the fact that $u\sim d^s$, \eqref{eq.Poisson}, \eqref{eq.hPoisson}, and $f$ is bounded, we reach that 
\[
w_\eps(x_\circ) \le C \int_{\partial\Omega} \int_0^{\eps f(z)} \left(\frac{\eps f(z)- t}{t}\right)^s\, dt d\sigma(z)= C \eps,
\]
for some $C$ depending only on $n$, $s$, $t_\circ$, and the exterior and interior ball condition of $\Omega$ and $\|f\|_{C^2(\partial\Omega)}$. In particular, $F_\eps$ is bounded in $B_{t_\circ/4}(x_\circ)$ independently of $\eps$. We can repeat the same argument for the derivatives of $w_\eps$ to deduce that $F_\eps$ is smooth (independently of $\eps$) in $B_{t_\circ/4}(x_\circ)$, so that by Arzel\`a-Ascoli we can take subsequences and conclude that $F_\eps \to F_s$ uniformly and $F_s$ is smooth in $B_{t_\circ/4}(x_\circ)$. 

In all, we have proved that $F_s$ is well defined (and smooth) in the interior of $\Omega$. Nonetheless, the function $F_s$ might explode when approaching the boundary $\partial\Omega$. Notice, however, that from the argumentation above, if $z_\circ\in \partial (B\cap \Omega)\setminus \partial\Omega$, then $F_s(z) \to 0$ as $z\to z_\circ$ and $F_s$ is continuous across $\partial B$. 
\\[0.1cm]
{\bf Step 2.}
Let us now compute an expansion around boundary points on $\partial\Omega$ for $F_s$. Recall that $F_s = 0$ in $(\Omega \cap B)^c$, and $\fls F_s = 0$ in $\Omega\cap B$. On the other hand, if we denote by $\nu : \partial \Omega \to \mathbb{S}^{n-1}$ the unit inward normal vector to $\partial \Omega$, then by Lemma~\ref{lem.form2} we can expand $u$ at $z$ as 
\[
u(z+t\nu(z) ) = t^s + U_1(z) t^{1+s} + o(t^{1+s}),
\]
for $t \ge 0$, and where $U_1(z) = \partial_\nu \left(\frac{u}{d^s}\right)(z)$,  since $u$ is a solution to the one-phase problem, \eqref{eq.OPT}, and by Proposition~\ref{prop.crs} the normal derivative is constant (recall $\Lambda = \frac{1}{\Gamma(1+s)}$); in particular, the tangential term in the expansion at order $1+s$ vanishes.

Let us also expand $v_\eps$ in the $\nu(z)$ direction at the point $P_\eps(z) \coloneqq z+\eps f(z) \nu(z)$. To do so, let us denote by $d_\eps(x) = {\rm dist}(x, \partial\Omega_\eps)$, and $\eta_\eps(x) = {v_\eps(x)}/{d_\eps^s(x)}$ for $x\in \overline{\Omega_\eps}$. Then, we have that, expanding $v_\eps$ at $P_\eps(z)$,
\[
v_\eps(x) = \eta_\eps(P_\eps(z)) d^s_\eps(x) + (\nabla \eta(P_\eps(z))\cdot x) d_\eps^s(x) + o(|x-P_\eps(z)|d_\eps^s(x)). 
\]
We can now separate the term $\nabla \eta(P_\eps(z))\cdot x$ between its normal and its tangential directions to $\partial\Omega_\eps$ at $P_\eps(z)$. Thus, we have
\begin{equation}
\label{eq.expandas}
v_\eps(x) = V_{0,\eps} d^s_\eps(x) + V_{1,\eps} d^{1+s}_\eps(x)+(\tilde V_{\tau,\eps} \cdot x) d_\eps^s(x) + o(|x-P_\eps(z)|d_\eps^s(x)),
\end{equation}
where $\tilde V_\tau = \nabla_\tau \eta(P_\eps(z))$ is the tangential (to $\partial\Omega_\eps$) gradient of $\eta$ at $P_\eps(z)$, and $x\in \Omega_\eps$, and we have denoted for convenience $V_{0,\eps} = \eta_\eps(P_\eps(z))$ and $V_{1,\eps} = \partial_\nu \eta_\eps(P_\eps(z))$. Notice that, from the convergence of $v_\eps\to u$ as $\eps\downarrow 0$, we have 
\[
V_{0,\eps} = 1+o(1),\qquad  |\tilde V_{\tau,\eps}| = o(1),
\]
i.e. $\tilde V_{\tau,\eps}\to 0$ and $V_{0,\eps}\to 1$ as $\eps\downarrow 0$.

As in \eqref{eq.dist_asimpt} we have that, if $x = z+t f(z)\nu(z)$ for $t > 0$,
\[
d_\eps(x) = (t-\eps f(z))(1 + o(\eps)). 
\]
That is, 
\begin{equation}
\label{eq.ded115}
\begin{split}
v_\eps(x) & = V_{0,\eps} (t-\eps f(z))^s+ V_{1,\eps} (t-\eps f(z))^{1+s}\,+
\\& \quad +(\tilde V_{\tau,\eps} \cdot x) (t-\eps f(z))^s+ o((t-\eps f(z))^{1+s}) + o(\eps),
\end{split}
\end{equation}
We can rewrite it as  
\begin{align*}
v_\eps(x)  & = V_{0,\eps} \left(t^s-st^{s-1}\eps f(z) + o(\eps)\right) + V_{1,\eps} \left(t^{1+s} -(1+s)t^s\eps f(z) + o(\eps)\right)\,+ \\
& \quad +(\tilde V_{\tau,\eps} \cdot x) \left(t^s-st^{s-1}\eps f(z) + o(\eps)\right)+o((t-\eps f(x))^{1+s})+o(\eps).
\end{align*}
We want now to consider $\frac{u-v_\eps}{\eps}$ and let $\eps \downarrow 0$. Notice that $(\tilde V_{\tau,\eps}\cdot x) = t\, o(1)$, so
\begin{align*}
(u-v_\eps)(x)& = s f(z) \eps t^{s-1} + \left(1-V_{0,\eps} +(1+s)\eps V_{1,\eps} f(z)\right)t^s\,+\\
& \quad + (U_1(z) - V_{1,\eps} )t^{1+s} + o(\eps) + \eps o(t^{1+s}).
\end{align*}
That is, recalling $x = z+tf(z) \nu(z)$, 
\begin{align*}
F_s(x) = \lim_{\eps\downarrow 0} \frac{(u-v_\eps)(x)}{\eps} = & \,s f(z) t^{s-1} + \left( \lim_{\eps\downarrow 0}\frac{1-V_{0,\eps}}{\eps} + (1+s) V_{1,\eps} f(z) \right) t^s\,+\\
& \quad + t^{1+s}\lim_{\eps\downarrow 0}\frac{U_1(z) - V_{1,\eps}}{\eps} + o(t^{1+s}). 
\end{align*}
In particular, $F_s$ is a solution to 
\[
\left\{
\begin{array}{rcll}
\fls F_s & = & 0 & \quad\textrm{in } (\Omega \cap B)\\
F_s & = & 0 & \quad\textrm{in } (\Omega\cap B)^c \vspace{1mm}\\
\displaystyle \frac{F_s}{d^{s-1}} & = & s f& \quad \textrm{on }\partial \Omega,
\end{array}
\right.
\]
which is a well-posed Dirichlet-type problem for the fractional Laplacian (see \cite{Aba15}).
Moreover, from the expansion of solutions at points on the boundary, we know that the term in $t^{1+s}$ in the expansion of $F_s$ is bounded, so 
\begin{equation}
\label{eq.v0v10}
\lim_{\eps \downarrow 0 } V_{1,\eps} = U_1(z)\quad\Longleftrightarrow \quad V_{1,\eps} = U_1(z) + o(1). 
\end{equation}
Finally, from the coefficient of $t^s$, which corresponds to $\partial_\nu\left(\frac{F_s}{d^{s-1}}\right) $, and using the previous expression \eqref{eq.v0v10}, we have
\[
\partial_\nu\left(\frac{F_s}{d^{s-1}}\right) =  \lim_{\eps\downarrow 0}\frac{1-V_{0,\eps}}{\eps} + (1+s) U_1(z) f(z).
\]
Alternatively,
\begin{equation}
\label{eq.v0v1}
V_{0,\eps} = V_{0,\eps}(z) = 1-\eps \partial_\nu\left(\frac{F_s}{d^{s-1}}\right)+(1+s) \eps U_1(z) f(z) +o(\eps),
\end{equation}
we have an expansion up to order $\eps$ for $V_{0,\eps}$. 
\\[0.1cm]
{\bf Step 3.}
We now want to perform an expansion of $\mathcal{J}_{\Lambda, B}(v_\eps) - \mathcal{J}_{\Lambda, B}(u)$. This is formed by two terms (see \eqref{eq.20}). 

On the one hand, let us, as in the proof of Proposition~\ref{prop.crs}, consider a parametri\-zation of the points in $\Theta_\eps = (\Omega\setminus \Omega_\eps)\cap B$ as $z+t\nu(z)$ with $z\in \partial\Omega$, $0 < t < \eps f(z)$. We need another term in the expansion of the volume element, which is $dV = (1+Ht + o(t) ) d\sigma(z) dt$. 

Thus, we expand equation~\eqref{eq.21} to one more term as
\[
|\{u > 0\}| - |\{v_\eps > 0\}| = |\Theta_\eps| = \int_{\partial \Omega} \int_0^{\eps f(z) } \big(1+tH(z) + o(t) \big)\, dt\, d\sigma(z),
\]
where $H(z)$ is the mean curvature of $\partial \Omega$ at $z$ (with respect to $\{u = 0\}$). Namely,  
\begin{equation}
\label{eq.diffmean}
|\{u > 0\}| - |\{v_\eps > 0\}| = \eps \int_{\partial \Omega} f + \frac{\eps^2}{2} \int_{\partial \Omega} H f^2 + o(\eps^2).
\end{equation}

On the other hand, let us now compute the remaining term in $\mathcal{J}_{\Lambda, B}(v_\eps)-\mathcal{J}_{\Lambda, B}(u)$. Namely, 
\[
\int_{\Theta_\eps} u \fls v_\eps. 
\]
We expand both $u$ and $\fls v_\eps$ as in the proof of Proposition~\ref{prop.crs}, but we need one more term in the expansion now. First, we already know
\begin{equation}
\label{eq.epsexpu}
u(z+t\nu(z) ) = t^s + U_1(z) t^{1+s} + o(t^{1+s}).
\end{equation}

We can also expand $v_\eps$ at $x_\circ = z+\eps f(z) \nu(z)\in \partial \Omega_\eps$ as \eqref{eq.expandas}. 
By Lemma~\ref{lem.form2} we then have that 
\begin{align*}
\fls v_\eps(x) & = \bar c_s V_{0,\eps} d^{-s}_\eps(x) + \bar c_{1+s} V_{1,\eps} d^{1-s}_\eps(x) + \bar c_* V_{0,\eps} H_\eps  d^{1-s}_\eps(x)\,+\\
& \qquad\qquad\qquad\quad+\bar c_s (\tilde V_{\tau,\eps} \cdot x)  d^{-s}_\eps(x)+ o(|x-P_\eps(z)|d_\eps^s(x)) \quad \text{ in }\Omega_\eps^c,
\end{align*}
where $H_\eps(z) = H(z+\eps f(z)\nu(z))$ is the mean curvature of $\partial\Omega_\eps$ (with respect to $\{v_\eps =0 \}$). That is, from \eqref{eq.dist_asimpt} and as in the deduction of \eqref{eq.ded115},
 \begin{equation}
 \label{eq.flsvepsexp}
 \begin{split}
  \fls v_\eps(x) & = \bigg[\bar c_s V_{0,\eps}(z) (\eps f(z) - t)^{-s} +  \bar c_* V_{0, \eps}(z) H_\eps (\eps f(z) - t)^{1-s} \,+\\
 & \qquad + \bar c_{1+s} V_{1,\eps}(z) (\eps f(z) - t)^{1-s}+\bar c_s(\tilde V_{\tau,\eps}\cdot x)(\eps f(z) -t)^{-s}\bigg](1+o(\eps)). 
 \end{split}
 \end{equation}
Now, we can compute
\[
\int_{\Theta_\eps} u \fls v_\eps = \int_{\partial \Omega}\int_0^{\eps f(z)} u(z+t\nu(z)) \fls v_\eps (z+t\nu(z)) \big(1+tH + o(t)\big) \, dt\, d\sigma(z). 
\]
Using \eqref{eq.epsexpu} and \eqref{eq.flsvepsexp}, and noticing that the term involving $\tilde V_{\tau,\eps}$ is $o(\eps^2)$ in the integral, since $\tilde V_{\tau,\eps} = o(1)$, we have that 
\begin{align*}
&\int_{\Theta_\eps} u \fls v_\eps  =\\
& ~~ = \int_{\partial \Omega}\int_0^{\eps f(z)} t^s(1+U_1(z) t)(\eps f(z) - t)^{-s}\bigg\{ \bar c_* V_{0,\eps}(z) H_\eps(z) (\eps f(z) - t)\,+ \\
& \qquad  +\bar c_s V_{0,\eps}(z)+ \bar c_{1+s} V_{1,\eps}(z) (\eps f(z) - t)\bigg\}(1+tH(z))dtd\sigma + o(\eps^2)\\
& = \eps \int_{\partial \Omega}f(z)\int_0^{1} \left(\frac{t}{1-t}\right)^s(1+U_1(z)\eps f(z) t)\bigg\{ \bar c_* V_{0,\eps}(z) H_\eps(z) \eps f(z) (1 - t)\,+\\
& \qquad + \bar c_s V_{0,\eps}(z) + \bar c_{1+s} V_{1,\eps}(z) \eps f(z)(1 - t)\bigg\}(1+\eps f(z) tH(z))dtd\sigma + o(\eps^2).
\end{align*}
By making use now of the expansions of $V_{0,\eps}(z) $ and $V_{1, \eps}(z)$, \eqref{eq.v0v1} and \eqref{eq.v0v10}, 
 \begin{align*}
\int_{\Theta_\eps} u \fls v_\eps   = &\,\eps \int_{\partial \Omega}f(z)\int_0^{1} \left(\frac{t}{1-t}\right)^s(1+U_1(z)\eps f(z) t)\bigg\{ \bar c_*  H_\eps(z) \eps f(z) (1 - t)\,+ \\
&  + \bar c_s  \big(1-\eps\partial_\nu (F_s d^{1-s}) + (1+s)\eps f(z) U_1(z) \big) \\
& + \bar c_{1+s} U_1(z) \eps f(z)(1 - t)\bigg\}\big(1+\eps f(z) t H(z)\big)dt\,d\sigma + o(\eps^2).
\end{align*}
 Since the terms of order $\eps$ will vanish with those in \eqref{eq.diffmean} (by Proposition~\ref{prop.crs}), we are interested in the terms of order $\eps^2$. That is, if 
 \begin{equation}
 \label{eq.31}
 \int_{\Theta_\eps} u \fls v_\eps = \mathcal{K}_1\eps + \mathcal{K}_2 \eps^2+ o(\eps^2), 
 \end{equation}
 then we are interested in $\mathcal{K}_2$. From the previous expressions, also using that $H_\eps(z) = H(z) + o(1)$, 
 \begin{align*}
 \mathcal{K}_2 = & \int_{\partial \Omega} f \int_0^1 \left(\frac{t}{1-t}\right)^s\bigg[\bar c_s  U_1(z) f t  + \bar c_s f H t  -\bar c_s \partial_\nu (F_s d^{1-s})   \\
 &+\bar c_s (1+s) f(z) U_1(z)+\bar c_* H f (1-t)+ \bar c_{1+s} U_1 f (1-t)\bigg]\, dt\, d\sigma.
 \end{align*}
 Now notice that 
  \[
 \int_0^1 \left(\frac{t}{1-t}\right)^s  \, dt = s  \Gamma(s) \Gamma(1-s),
 \]
 \[
 \int_0^1 \left(\frac{t}{1-t}\right)^s t \, dt = \frac12 \Gamma(2+s) \Gamma(1-s),
 \]
  \[
 \int_0^1 \left(\frac{t}{1-t}\right)^s (1-t) \, dt = \frac12  \Gamma(1+s) \Gamma(2-s).
 \]
 In particular, using the values of $\bar c_s$ and $\bar c_{1+s}$ we have
 \begin{align*}
 \mathcal{K}_2 =&  \int_{\partial \Omega} H f^2 \int_0^1 \left(\frac{t}{1-t}\right)^s\left[\bar c_s t + \bar c_* (1-t) \right]\, dt\, d\sigma\,+ \\
 &  +(\Gamma(1+s))^2 \left(\int_{\partial\Omega} f \partial_\nu (F_s d^{1-s}) d\sigma-(1+s) \int_{\partial \Omega} U_1 f^2 \, d\sigma \right). 
 \end{align*}
 A direct computation yields that 
 \begin{equation}
 \label{eq.32}
 \mathcal{K}_2 = (\Gamma(1+s))^2 \left\{ \int_{\partial\Omega} f \partial_\nu (F_s d^{1-s}) d\sigma- \frac{1}{2}\int_{\partial \Omega} H f^2d\sigma   -(1+s) \int_{\partial \Omega} U_1 f^2 \, d\sigma\right\}.
 \end{equation}
 
 That is, recalling that $\Lambda = \Gamma(1+s)$, and from \eqref{eq.20}-\eqref{eq.diffmean}-\eqref{eq.31}-\eqref{eq.32}
 \begin{align*}
\mathcal{J}_{\Lambda, B}(v_\eps) & - \mathcal{J}_{\Lambda, B}(u) = \\
& 
= - (\Gamma(1+s))^2 \bigg\{ \int_{\partial\Omega} f \partial_\nu (F_s d^{1-s}) d\sigma- \frac{1}{2}\int_{\partial \Omega} H f^2d\sigma   \\
& \quad -(1+s) \int_{\partial \Omega} U_1 f^2 \, d\sigma \bigg\} \eps^2 -(\Gamma(1+s))^2\frac{\eps^2}{2}\int_{\partial\Omega} Hf^2 d\sigma + o(\eps^2).
 \end{align*}
 That is,
 \[
 	\frac{\mathcal{J}_{\Lambda, B}(v_\eps) - \mathcal{J}_{\Lambda, B}(u)}{\Gamma(1+s)^2} = \left\{ (1+s) \int_{\partial \Omega} U_1 f^2 \, d\sigma-\int_{\partial\Omega} f \partial_\nu (F_s d^{1-s}) d\sigma\right\}\eps^2 + o(\eps^2). 
 \]
 Now, since $u$ is a minimizer or a stable critical point, we get the desired result. 
 \\[0.1cm]
 {\bf Step 4.} Let us now show how to take general $f\in C^\infty_c(\partial\Omega)$ without the sign restriction. Just split $f = f_+ - f_-$ with $f_+ = \max\{f, 0\}$ and $f_- = -\min\{f, 0\}$, so that $f_+ \ge 0$ and $f_- \ge 0$. By the previous steps, we then have that 
 \[
-(1+s) \int_{\partial \Omega} \partial_\nu\left(\frac{u}{d^s}\right) f_\pm^2 \, d\sigma \le -\int_{\partial \Omega} f_\pm \partial_\nu\left(\frac{F_\pm}{d^{s-1}}\right)d\sigma,
\]
where $F_\pm$ is the solution to 
\[
\left\{
\begin{array}{rcll}
\fls F_\pm & = & 0 & \quad\textrm{in } (\Omega \cap B)\\
F_\pm & = & 0 & \quad\textrm{in } (\Omega\cap B)^c\\
\frac{F_\pm}{d^{s-1}} & = & s f_\pm& \quad \textrm{on }\partial \Omega.
\end{array}
\right.
\]
Let us also consider the solution $F_s$ to 
\[
\left\{
\begin{array}{rcll}
\fls F_s & = & 0 & \quad\textrm{in } (\Omega \cap B)\\
F_s & = & 0 & \quad\textrm{in } (\Omega\cap B)^c\\
\frac{F_s}{d^{s-1}} & = & s f& \quad \textrm{on }\partial \Omega.
\end{array}
\right.
\]
Notice that by linearity (and uniqueness of solution) we have that $F_s  = F_+ - F_-$. In particular, we have that 
\[
f \partial_\nu (F_sd^{1-s})= f_+\partial_\nu (F_+d^{1-s}) + f_- \partial_\nu (F_-d^{1-s}) - f_+ \partial_\nu (F_-d^{1-s}) - f_- \partial_\nu (F_+d^{1-s}).
\]

Notice, also, that when $f_+ > 0$, then $f_- = 0$ and therefore we must have $\partial_\nu(F_- d^{1-s}) \ge 0$. Thus, $f_+ \partial_\nu (F_-d^{1-s}) \ge 0$ and $f_- \partial_\nu (F_+d^{1-s}) \ge 0$, and 
\begin{equation}
\label{eq.ffpfm}
f \partial_\nu (F_sd^{1-s})\le f_+\partial_\nu (F_+d^{1-s}) + f_- \partial_\nu (F_-d^{1-s}).
\end{equation}
Thus
\begin{align*}
-(1+s) \int_{\partial \Omega} \partial_\nu\left(\frac{u}{d^s}\right) f^2 \, d\sigma & = -(1+s) \int_{\partial \Omega} \partial_\nu\left(\frac{u}{d^s}\right) (f_+^2+f_-^2) \, d\sigma\\
& \le -\int_{\partial \Omega} f_+ \partial_\nu\left(\frac{F_+}{d^{s-1}}\right)d\sigma-\int_{\partial \Omega} f_- \partial_\nu\left(\frac{F_-}{d^{s-1}}\right)d\sigma\\
& \le -\int_{\partial \Omega} f \partial_\nu\left(\frac{F_s}{d^{s-1}}\right)d\sigma,
\end{align*}
where in the last inequality we are using \eqref{eq.ffpfm}. The proof is now complete, noticing that $F_s = s F$. 
\end{proof}

Let us now prove our main result, the stability condition in Theorem~\ref{thm.main}. Before doing so, we first prove the following result, which is valid for non-homogeneous functions, too. 

\begin{thm}
\label{thm.cor.1}
Let $u\in C^s(\R^n)$ be a global  stable solution for \eqref{eq.energy}, in the sense \eqref{eq.criticalpoint}-\eqref{eq.stablesolution}. Let $\Omega \coloneqq \{u > 0\}$, and assume that $\partial\Omega$ is $C^{2,\alpha}$ outside the origin. 
Let us consider $\mathcal{K}_{\Omega, s}$ and $H_{\Omega, s}$ as defined in Definition~\ref{defi.kernel}. 

Then, we have
\begin{equation}
\label{eq.stabilityFin}
 \int_{\partial\Omega} \int_{\partial\Omega} \big(f(x) - f(y)\big)^2\mathcal{K}_{\Omega, s}(x, y)d\sigma(x)d\sigma(y) \ge \int_{\partial\Omega} H_{\Omega, s}  f^2 d\sigma 
\end{equation}
for all $f\in C^\infty_c(\partial\Omega\setminus \{0\})$.
\end{thm}
\begin{proof}
Let $\Omega\subset\R^n$ be a smooth domain (not necessarily bounded). Let $f\in C^\infty_c(\partial \Omega)$ be a smooth function defined on the  boundary of $\Omega$. Alternatively, let us assume $\partial\Omega$ is smooth in ${\rm supp}\, f$. We then define $T_\Omega(f):\partial \Omega\to \R$ as follows. 

Let $F$ be the unique solution to 
\begin{equation}
\label{eq.Finall}
\left\{
\begin{array}{rcll}
\fls F & = & 0 & \quad\textrm{in } \Omega \\
F & = & 0 & \quad\textrm{in } \R^n\setminus \Omega\\
\displaystyle\frac{F}{d^{s-1}} & = &  f& \quad \textrm{on }\partial \Omega\\
F(x) & \to & 0 & \quad\textrm{as } |x|\to \infty,
\end{array}
\right.
\end{equation}
which can be obtained by a Green kernel representation (as in \cite{Aba15}) or as the limit, when ${\rm diam}(B)\to \infty$, of solutions to \eqref{eq.FinB}. Then, we define for $x_\circ \in \partial\Omega$,
\begin{equation}
\label{eq.defL}
T_\Omega(f)(x_\circ)  = \partial_\nu \left(\frac{F}{d^{s-1}}\right)(x_\circ) = \lim_{\Omega\ni x\to x_\circ} \frac{F(x)-d^{s-1}(x)f(x_\circ)}{d^s(x)},
\end{equation}
where $\nu = \nu(x)\in \mathbb{S}^{n-1}$ denotes the unit inward normal vector to $\partial\Omega$ at $x$.

Let us denote by $G_{\Omega,s}(x, y)$ the Green function of the operator $\fls$ for the domain $\Omega$. That is, given a function $h:\Omega\to \R$, $u_h(x) = \int_\Omega h(y) G_{\Omega, s}(x, y) \, dy$ satisfies 
\[
\left\{
\begin{array}{rcll}
\fls u_h & = & h & \quad\textrm{in } \Omega \\
u_h & = & 0 & \quad\textrm{in } \R^n\setminus \Omega.
\end{array}
\right.
\]

Then, we define for each $x\in \Omega$, $y\in \partial\Omega$, 
\[
\bar G_{\Omega, s}(x, y) \coloneqq \lim_{\Omega\ni\bar y \to y} \frac{G_{\Omega, s}(x, \bar y)}{d^s(\bar y)}
\]
which is well defined by the Green function estimates (see \cite{CS98}). 
By the arguments in \cite{Aba15}, given $f\in C^\infty_c(\partial\Omega)$ then we have that
\[
F(x) = \int_{\partial\Omega}  f(y) \bar G_{\Omega, s}(x, y) \, d\sigma(y)
\]
satisfies \eqref{eq.Finall}. Now notice that, from \eqref{eq.defL},
\begin{align*}
T_\Omega(f)(x_\circ) & = \lim_{\Omega\ni x\to x_\circ} \frac{1}{d(x)}\left\{\int_{\partial\Omega} f(y) \frac{\bar G_{\Omega, s}(x, y)}{d^{s-1}(x)}\, d\sigma(y) -f(x_\circ)\right\}\\
& = \lim_{\Omega\ni x\to x_\circ}\int_{\partial\Omega} (f(y) - f(x_\circ)) \frac{\bar G_{\Omega, s}(x, y)}{d^{s}(x)} d\sigma(y)\,+\\
& \quad + f(x_\circ) \lim_{\Omega\ni x\to x_\circ}\frac{1}{d(x)}\left\{\int_{\partial\Omega} \frac{\bar G_{\Omega, s}(x, y)}{d^{s-1}(x)}\, d\sigma(y) -1\right\}.
\end{align*}

We now see that the second term corresponds to the operator $T_\Omega$ applied to the constant function 1, $f(x_\circ)T_\Omega(1)$. For the first term, we recover the kernel $\mathcal{K}_{\Omega, s}$ from Definition~\ref{defi.kernel}. Thus, 
\begin{equation}
\label{eq.Tomega}
T_\Omega(f)(x) = \int_{\partial\Omega} \big(f(y) - f(x)\big) \mathcal{K}_{\Omega, s}(x, y) d\sigma(y)  + f(x) T_\Omega(1). 
\end{equation}

Taking limits when ${\rm diam}( B) \to \infty$, condition \eqref{eq.condFinB} in Proposition~\ref{prop.main} can be expressed as 
\[
-C_s \int_{\partial \Omega} \partial_\nu\left(\frac{u}{d^s}\right) f^2  \le -\int_{\partial \Omega} fT_\Omega(f),\qquad C_s = \frac{1+s}{s\Lambda} \Gamma(1+s).
\]
Using \eqref{eq.Tomega} and from the symmetry of the kernel, $\mathcal{K}_{\Omega, s}$, and re-ordering terms we have
\begin{equation}
\label{eq.exp1}
\int_{\partial\Omega} \left(T_\Omega(1) - C_s U_1 \right) f^2d\sigma \le \frac12 \int_{\partial\Omega}\int_{\partial\Omega} \left(f(x) - f(y)\right)^2\mathcal{K}_{\Omega, s}(x, y)d\sigma(x)d\sigma(y), 
\end{equation}
where we are defining, for $x\in \partial\Omega$,
\[
U_1(x) \coloneqq \partial_\nu\left(\frac{u}{d^s}\right)(x).
\]
That is, we can expand $u$ at boundary points as
\[
u = \frac{\Lambda}{\Gamma(1+s)} d^s + U_1 d^{1+s} + \dots.
\]
Let now $\be \in \mathbb{S}^{n-1}$, and consider the function $\partial_{\be} u$. From the previous expansion, $u_\be \coloneqq \partial_\be u$ satisfies
\[
\left\{
\begin{array}{rcll}
\fls  u_\be& = & 0 & \quad\textrm{in } \Omega \\
u_\be & = & 0 & \quad\textrm{in } \R^n\setminus \Omega\\
\displaystyle\frac{u_\be}{d^{s-1}}(x) & = &  \be \cdot \nu(x) \frac{s\Lambda}{\Gamma(1+s)} & \quad \textrm{for }x\in \partial \Omega\\
u_\be(x) & \to & 0 & \quad\textrm{as } |x|\to \infty,
\end{array}
\right.
\]
where $\nu(x)$ denotes the unit inward normal vector at $x\in \partial\Omega$. Thus, we can compute $U_1(x_\circ)$ at $x_\circ\in \partial\Omega$ as 
\[
\frac{s\Lambda}{\Gamma(1+s)} T_\Omega(v_{x_\circ})(x_\circ) = (1+s) U_1(x_\circ). 
\]
where $v_{x_\circ}(x) \coloneqq \nu(x_\circ) \cdot \nu(x)$. 

Putting it back in \eqref{eq.exp1} we get 
\[
\int_{\partial\Omega} T_\Omega(1-v_x)(x)  f(x)^2d\sigma(x) \le \frac12 \int_{\partial\Omega}\int_{\partial\Omega} \left(f(x) - f(y)\right)^2\mathcal{K}_{\Omega, s}(x, y)d\sigma(x)d\sigma(y).
\]

Notice now that $1-v_{x_\circ} (x)= \frac12 |\nu(x_\circ) - \nu(x)|^2$ so, if we define 
\[
H_{\Omega, s}(x) \coloneqq T_{\Omega}(|\nu(x)- \nu(\cdot)|^2)(x),
\]
then the stability condition reads as 
\[
\int_{\partial\Omega} H_{\Omega, s}  f^2 d\sigma \le \int_{\partial\Omega}\int_{\partial\Omega} \left(f(x) - f(y)\right)^2\mathcal{K}_{\Omega, s}(x, y)d\sigma(x)d\sigma(y).
\]
for all $f\in C^\infty_c(\partial\Omega)$ and assuming $\partial\Omega$ is smooth on ${\rm supp} f$. 

On the other hand, if $g(y) = |\nu(x) - \nu(y)|^2$ then $g(x) = 0$ and so we can express
\[
H_{\Omega, s}(x) \coloneqq \int_{\partial \Omega} |\nu(x) - \nu(y)|^2 \mathcal{K}_{\Omega, s}(x, y) d\sigma(y),
\]
as we wanted to see. 
\end{proof}

In case of homogeneous solutions, we have the following:

\begin{lem}
\label{lem.estimatesK}
Let $u$ be a global $s$-homogeneous solution to the fractional one-phase problem, i.e.,
\[
\left\{
\begin{array}{rcll}
\fls u & = & 0 & \quad\textrm{in } \Omega\\
u & = & 0 & \quad\textrm{in } \Omega^c\\
\displaystyle\Gamma(1+s) \frac{u}{d^{s}} & = & \Lambda& \quad \textrm{on }\partial \Omega.
\end{array}
\right.
\]
Assume that $\Omega$ is a $C^{1,\alpha}$ cone outside the origin. 
Then, the kernel $\mathcal{K}_{\Omega, s}$ defined in \eqref{eq.kernelK_C} is homogeneous of degree $-n$, and satisfies
\[
\frac{1}{C} \frac{1}{|x-y|^n} \le \mathcal{K}_{\Omega, s}(x, y) \le C \frac{1}{|x-y|^n}\quad\text{for all}~~x, y \in \partial\Omega,
\]
for some constant $C$ depending only on $n$, $s$, and $\Omega$. 
\end{lem}

\begin{proof}
The domain $\Omega$ is a cone, so the Green function satisfies the scaling property $G_{\Omega, s}(rx, ry) = r^{2s-n}G_{\Omega, s}(x, y)$.
This implies that
\[
\mathcal{K}_{\Omega, s}(rx, ry) = r^{-n}\mathcal{K}_{\Omega, s}(x, y).
\]
Thus, it only remains to prove that, if $|x-y| = 1$, then 
\[
\frac{1}{C} \le \mathcal{K}_{\Omega, s}(x, y) \le C \quad\text{for all}~~x, y \in \partial\Omega\quad\text{s.t.}\quad |x-y| = 1.
\]

Notice first that, since $u$ is an $s$-harmonic function in $\Omega$, then by well-known estimates in $C^{1,\alpha}$ domains \cite{RS17} we have that $u(x)\asymp d^s(x)$ for $|x|\asymp 1$.
Then, since both $u$ and $d^s$ are homogeneous of degree $s$, we deduce that 
\[u\asymp d^s\quad \textrm{in}\quad \overline\Omega.\]

On the other hand, given $y_0\in \Omega$, we know that the Green function $G_{\Omega,s}(x,y_0)$ is $s$-harmonic in $x$ in the domain $\Omega\setminus\{y_0\}$.
Thus, if $|x-y_0|\asymp 1$, then by the boundary Harnack principle for the fractional Laplacian in Lipschitz domains \cite{Bogdan} (applied to $G_{\Omega,s}(x,y_0)$ and $u(x)$) we know that the function $G(x,y_0)$ must be comparable to $d^s(x)$, for every such $y_0\in \Omega$.
This means that
\[\mathcal M_{\Omega,s,x}(\bar y)\coloneqq\lim_{\Omega\ni \bar x\to x} \frac{G_{\Omega,s}(\bar x,\bar y)}{d^s(\bar x)}\in (0,\infty)\]
for every fixed $x\in\partial\Omega$ and $\bar y\in \Omega$ such that $|x-\bar y|\asymp 1$.
Now notice that, for every fixed $x\in \partial\Omega$, the function $\mathcal M_{\Omega,s,x}(\bar y)$ is $s$-harmonic in $\bar y$ in a neighborhood of $y\in \partial\Omega$, provided that $|x-y|=1$.
Hence, using again the boundary Harnack inequality for the fractional Laplacian, we deduce that actually 
\[\mathcal M_{\Omega,s,x}(\bar y) \asymp d^s(\bar y),\]
which clearly implies $\mathcal{K}_{\Omega, s}(x, y)\asymp~1$, as wanted.
\end{proof}

Thus, as a consequence, we have:

\begin{proof}[Proof of Theorem~\ref{thm.main}]
We apply Theorem~\ref{thm.cor.1} to $s$-homogeneous solutions $u$. In this case, $\Omega$ is a cone, and the homogeneity of $\mathcal{K}_{\Omega, s}$ comes from scaling of the Green function. The estimate \eqref{eq.Kcone} follows from Lemma~\ref{lem.estimatesK}. 

Finally, the $(-1)$-homogeneity of $H_{\Omega, s}$ follows from the $(-n)$-homogeneity of the kernel together with the $0$-homogeneity of $\nu$. 
\end{proof}

\section{Stable cones are trivial in 2D}

\label{sec.4}

Let us now give the proof of Corollary~\ref{cor.2D}, stating that cone-like solutions in $\R^2$ are not stable, in the sense \eqref{eq.criticalpoint}-\eqref{eq.stablesolution}. 

\begin{proof}[Proof of Corollary~\ref{cor.2D}]
We have that
\[
\left\{
\begin{array}{rcll}
\fls u & = & 0 & \quad\textrm{in } \Omega\\
u & = & 0 & \quad\textrm{in } \Omega^c\\
\frac{u}{d^{s}} & = & 1& \quad \textrm{on }\partial \Omega,
\end{array}
\right.
\]
$u\in C^s(\R^2)$ and $\Omega^c = \{u = 0\}$ is a cone.
By boundary regularity for $s$-harmonic functions, it is not difficult to see that the cone $\Omega^c$ cannot have zero density points.
In particular, the contact set is the union of circular sectors and they are smooth outside of the origin. 

We argue by contradiction, and we assume that $u$ is not a half-space solution, but it is stable. In this case, by Theorem~\ref{thm.main}, the stability condition \eqref{eq.stabilityFin_C} implies
\begin{equation}
\label{eq.contradictionIR0}
C \int_{\partial\Omega}\int_{\partial \Omega} \frac{\left(f(x)-f(y)\right)^2}{|x-y|^2}  d\sigma(x)d\sigma(y) \ge \int_{\partial\Omega} \frac{f(x)^2}{|x|} d\sigma(x)
\end{equation}
for some $C$ that depends on $\Omega$, and for all $f\in C^\infty_c(\partial\Omega\setminus\{0\})$. Let us show that, for an appropriate $f$, \eqref{eq.contradictionIR0} does not hold, thus reaching a contradiction. 

In particular, we choose $f = f_R(x) = \varphi(|x|)\zeta_R(|x|)$ radial with
 \[
\zeta_R(t) = 
\left\{
\begin{array}{ll}
1 & \text{for } 0< t < R\\
(2R-t)R^{-1} & \text{for } R\le t < 2R\\
0 & \text{for } t\ge 2R,\\
\end{array}
\right.
\]
and $\varphi(t)$ any smooth function such that $\varphi \equiv 1$ for $t > 1$ and $\varphi \equiv 0$ for $t \le \frac12$. Thus, $f_R$ is simply a Lipschitz function, that equals $1$ for $|x|\le R$, vanishes for $|x|\ge 2R$, and is linearly (radially) connected in-between. We have also multiplied by $\varphi$ to avoid dealing with the origin, that will not play a role in the computations below. 

Notice that, on the one hand, since $\partial\Omega$ is one-dimensional,
\begin{equation}
\label{eq.Ilinf}
\int_{\partial\Omega\setminus B_{1/2}} \frac{f_R(x)^2}{|x|} \ge \int_{1}^R \frac{dr}{r}= \ln(R)\to \infty,\quad\text{as}\quad R\to \infty.
\end{equation}

On the other hand,  using again that $\partial\Omega$ are rays emanating from the vertex and thus the problem can be reduced to a one dimensional question, we have that
\[
\int_0^{\infty}\int_0^{\infty} \frac{(\zeta_R(t)-\zeta_R(\tau))^2}{(t-\tau)^2}dtd\tau = \int_0^\infty\int_0^\infty \frac{(\s_1(t) - \s_1(\tau))^2}{(t-\tau)^2}\, dt\, d\tau\le C <\infty
\]
for some $C$ depending only on $\Omega$. We have also used here that $\s_R(R t') = \s_1(t')$ and that $\s_1$ is Lipschitz and compactly supported. Combined with \eqref{eq.Ilinf}, this implies that, for $R$ large enough, \eqref{eq.stabilityFin_C} with $f = f_R$ does not hold, thus reaching a contradiction. 
\end{proof}

\section{The stability condition in the extension domain}
\label{sec.5}

We finish this section by proving the stability condition in the extension domain, \eqref{eq.condFinB_E_I}, expressed in Proposition~\ref{prop.main_E} below.

Let us first define the operators ${\lambda}$, $\gamma_0$, and $\gamma_1$ acting on a function $F = F(x, y):\R^{n+1}_+\to \R$ (where $\R^{n+1}_+ = \R^n\times\R_+$) and returning a function on $\partial\Omega$ as follows. Here, $\Omega\subset\R^n$ is a fixed smooth domain (it will be used with $\Omega = \{u  >0 \}$), and $d(x) = {\rm dist}(x, \partial \Omega)$ (the distance in the thin space), where $x\in \R^n$. 

Let $\bar F(x) = F(x, 0)$. Then, we define
\begin{equation}
\label{eq.gamma0}
\gamma_0(F):\partial\Omega\to \R \quad \text{as} \quad \gamma_0(F) (x) = \lim_{t\downarrow 0} \left(\frac{\bar F}{d^{s-1}}\right) (x+t\nu(x)),
\end{equation}
and
\begin{equation}
\label{eq.gamma1}
\gamma_1(F) :\partial\Omega\to \R \quad \text{as} \quad \gamma_1(F) (x) = \lim_{t\downarrow 0} \left[\partial_\nu\left(\frac{\bar F}{d^{s-1}}\right) \right](x+t\nu(x)),
\end{equation}
where $\nu$ is the unit inward normal vector on $\partial\Omega$. On the other hand, we define
\begin{equation}
\label{eq.gamma-1}
{\lambda}(F) :\partial\Omega\to \R \quad \text{as} \quad {\lambda}(F) (x) = \lim_{y\downarrow 0} y^{1+s} L_a F(x, y).
\end{equation}

We sometimes refer to ${\lambda}(F)$, $\gamma_0(F)$, and $\gamma_1(F)$ simply as $L_a F y^{1+s}$, $F/d^{s-1}$ and $\partial_\nu(F/d^{s-1})$ on $\partial\Omega$, respectively.

The stability condition now can be alternatively stated as follows. Here, we use the notation introduced in Lemma~\ref{lem.form2} by denoting $\delta(x)$ the signed distance to $\partial\Omega$ (i.e. $\delta(x)  ={\rm dist}(x, \partial\Omega)$ for $x\in \Omega$, $\delta(x) = -{\rm dist}(x, \partial\Omega)$ for $x\in \Omega^c$), and
\[
\mathcal{U}_s  = (r+\delta)^s,\quad r = (\delta^2 + y^2)^\frac12. 
\]

\begin{prop}
\label{prop.main_E}
Let $u$ be a stable solution for \eqref{eq.energy} in $\R^n$, in the sense \eqref{eq.criticalpoint}-\eqref{eq.stablesolution}. 
Let $\Omega \coloneqq \{u > 0\}\subset \R^n$, and assume that $\partial \Omega$ is $C^{2,\alpha}$. Then, $u$ satisfies
\begin{equation}
\label{eq.condFinB_E}
\begin{split}
\int_{\partial \Omega} \gamma_0(F) \gamma_1(F) & -\frac{\Gamma(2+s)}{s\Lambda} \int_{\partial \Omega} \partial_\nu\left(\frac{u}{d^s}\right) \gamma_0(F)^2 \, \,-2^{s-1}\frac{1-s}{s}\int_{\partial\Omega} \gamma_0(F){\lambda}(F)\leq  \\
& \qquad \qquad\le -\frac{d_s}{\Gamma(1+s)\Gamma(s)} \int_{\Omega} F \partial_y^a F - \frac{d_s}{\Gamma(1+s)\Gamma(s)} \int_{\{y > 0\}} F L_a F,
\end{split}
\end{equation}
for all $F:\overline{\R^{n+1}_+}\to \R$ compactly supported such that $F \equiv 0$ in $\Omega^c\times\{0\}$ and $F/|\nabla \mathcal{U}_s|\in C^{1,\alpha}(\overline{\R^{n+1}_+})$; where $\nu$ is the unit inward normal vector on $\partial\Omega$, and we have used the notation introduced in \eqref{eq.gamma0}-\eqref{eq.gamma1}-\eqref{eq.gamma-1} with $\Omega = \{u > 0\}$, and \eqref{eq.La}-\eqref{eq.FracDer}-\eqref{eq.ds}.
\end{prop}

%for all $F:\overline{\R^{n+1}_+}\to \R$ compactly supported such that $F \equiv 0$ in $(\R^n\setminus \Omega)\times\{0\}$ and $\gamma_0(F)\in C^1_c(\partial\Omega)$; where $\nu$ is the unit inward normal vector on $\partial\Omega$, and we have used the notation introduced in \eqref{eq.gamma0}-\eqref{eq.gamma1} with $\Omega = \{u > 0\}$, and \eqref{eq.La}-\eqref{eq.FracDer}-\eqref{eq.ds}.
%***

The condition that $F/|\nabla \mathcal{U}_s|$ is $C^{1,\alpha}$ is natural in order to make sense of ${\lambda}$, $\gamma_0(F)$, and $\gamma_1(F)$; and moreover it is the one satisfied by functions $F$ behaving like \emph{large}-solutions to the Dirichlet problem for the fractional Laplacian (as we will see in Proposition~\ref{prop.stab_E}).

%We recall that we will be using the definitions of $\gamma_0$ and $\gamma_1$ introduced in \eqref{eq.gamma0}-\eqref{eq.gamma1} acting on a function $F = F(x, y):\R^{n+1}_+\to \R$ (recall $\R^{n+1}_+ = \R^n\times\R+$) 

\begin{rem}
In the context of Proposition~\ref{prop.main}, if we extend $F$ to be $L_a$-harmonic in $\{y>0\}$, then  we have that $\gamma_0(F) = f$ and the right-hand side of \eqref{eq.condFinB} is actually $-\int_{\partial\Omega}\gamma_0(F) \gamma_1(F)$. 
\end{rem}

Finally, before proving Proposition~\ref{prop.main_E}, let us state the following well-known integration by parts involving \emph{large} and \emph{standard} solutions.

\begin{lem}
\label{lem.IBP}
Let $\Omega \subset \R^n$ be a $C^2$ domain. Let $u$, $v$ be such that $u = v\equiv 0$ in~$\Omega^c$. 
Assume that $v$ is a \emph{large} solution (namely, $v/d^{s-1}\in C^0(\overline{\Omega})$ ) and let $u$ be a \emph{standard} solution (namely, $u\in C^0({\R^n})$). Then, 
\[
\int_{\partial\Omega} \frac{v}{d^{s-1}}\cdot \frac{u}{d^s} = \frac{1}{\Gamma(1+s)\Gamma(s)}\int_\Omega\big\{v \fls u - u \fls v\big\},
\]
where $d = {\rm dist}(x, \partial\Omega)$.
\end{lem}

\begin{proof}
When $(-\Delta)^s u\in C^\infty_c(\Omega)$, this integration by parts formula corresponds to \cite[Proposition 1.2.2]{Aba15} combined with \cite[Lemma B.1]{CGV20}.
(In $C^\infty$ domains it was first proved in \cite[Corollary 4.5]{Gru18}; see also \cite{Gru19}.)
By approximation, we can consider arbitrarily $\fls u$ integrable.
\end{proof}

We can now prove the stability condition in the extended variable.

\begin{proof}[Proof of Proposition~\ref{prop.main_E}]
We divide the proof into two steps.
\\[0.1cm]
{\bf Step 1.} Let us denote $\tilde F(x) = F(x, 0)$ and let us split
\[
\tilde F = F_0 + F_1,
\]
where $F_0 = F_1 = 0$ in $\Omega^c$, and 
\[
\left\{
\begin{array}{rcll}
\fls F_0 & = & 0 & \quad\textrm{in } \Omega\\
\gamma_0(F_0) & = & \gamma_0(\tilde F) & \quad\textrm{on } \partial\Omega,
\end{array}
\right.
\qquad 
\left\{
\begin{array}{rcll}
\fls F_1 & = & \fls \tilde F & \quad\textrm{in } \Omega\\
\gamma_0(F_1) & = & 0 & \quad\textrm{on } \partial\Omega,
\end{array}
\right.
\]
where we notice that from the condition on $F$, $\gamma_0(\tilde F)$ is well defined. In particular, $F_0$ is a large-solution, whereas $F_1$ is a standard-solution. 

From Proposition~\ref{prop.main} we have that
\[
-(1+s)\frac{\Gamma(1+s)}{s\Lambda} \int_{\partial \Omega} \partial_\nu\left(\frac{u}{d^s}\right) (\gamma_0(F_0))^2 \, d\sigma \le -\int_{\partial \Omega} \gamma_0(F_0) \gamma_1(F_0)d\sigma,
\]
where by approximation we are using that it is enough to assume $\gamma_0(F_0) \in C^1_c(\partial \Omega\cap B)$.
(Notice that we are also taking a global solution by letting ${\rm diam(B)}\to \infty$ in Proposition~\ref{prop.main}.) Take now $F_1$ such that $\gamma_0(F_1) \equiv 0$ (so it is not a \emph{large}-type solution), and suppose $F_1 \equiv 0$ in $\Omega^c$. 

Notice that, if we denote $C_{s,\Lambda} \coloneqq (1+s)\frac{\Gamma(1+s)}{s\Lambda}$, then from the previous inequality we obtain
\begin{equation}
\label{eq.follows1}
\int_{\partial\Omega}\gamma_0(\tilde F) \gamma_1(\tilde F) - C_{s,\Lambda} \int_{\partial\Omega}\partial_\nu\left(\frac{u}{d^s}\right) (\gamma_0(\tilde F))^2\le \int_{\partial\Omega}\gamma_0(F_0) \gamma_1(F_1). 
\end{equation}

We also used here that $\gamma_0(F_1) = 0$. In particular, we have $\gamma_1(F_1) = \frac{F_1}{d^s}$, and we can use Lemma~\ref{lem.form2} to get that
\begin{equation}
\label{eq.follows2}
\int_{\partial\Omega}\gamma_0(F_0) \gamma_1(F_1) = \frac{1}{\Gamma(1+s)\Gamma(s)}\int_{\Omega} F_0 \fls F_1\le \frac{1}{\Gamma(1+s)\Gamma(s)}\int_{\Omega} \tilde F \fls \tilde F,
\end{equation}
where in the last inequality we have used that $F_0$ is $s$-harmonic in $\Omega$ and $\int_\Omega F_1\fls F_1 \ge 0$ (which holds for all functions).

% Thus, we have
%\[
%\int_{\partial\Omega}\gamma_0(\tilde F) \gamma_1(\tilde F) - C_{s,\Lambda} \int_{\partial\Omega}\partial_\nu\left(\frac{u}{d^s}\right) \gamma_0(\tilde F)^2\le  \frac{1}{\Gamma(1+s)\Gamma(s)}\int_{\Omega} \tilde F \fls \tilde F
%\]

Finally, we want to deal with the last term, $\int_\Omega \tilde F\fls \tilde F$. Let us consider $F_2$ to be the $a$-harmonic extension of $\tilde F$ to $\{y > 0\}$. Namely, 
\[
\left\{
\begin{array}{rcll}
L_a F_2 & = & 0 & ~~\textrm{in } \{y > 0\}\\
F_2 & = & F & ~~\textrm{on } \{y = 0\},
\end{array}\quad \text{so}\quad
\right. 
\int_\Omega \tilde F \fls \tilde F= -d_s \lim_{\eps\downarrow 0}\int_{\Omega\cap \{\delta > \eps\}}F_2 \partial_y^a F_2.
\]
(Recall \eqref{eq.La}-\eqref{eq.FracDer}-\eqref{eq.ds}.) We have denoted here $\partial_y^a v = \lim_{y \downarrow 0} y^a \partial_y v(x, y)$, and from now on we use the notation from Lemma~\ref{lem.constants}. 

Let us  do some manipulations. We use the following Green's identity:
\[
\int_{D} (gL_a f -fL_a g) = \int_{\partial D} (-f\partial_{\vec{n}}  g y^a+g\partial_{\vec{n}} f y^a),
\]
for all pairs of functions $f$ and $g$ such that each of the previous terms is well-defined, and where ${\vec{n}}$ denotes the outward normal to corresponding domain. Then we have, if we denote $\Omega_\eps = \{\delta > \eps\}$ and $A^+ := A\cap  \{y > 0\}$ for any set $A \subset \R^{n+1}$,
\begin{align*}
\int_{\Omega_\eps} F_2 \partial_y^a F_2 = -\int_{\partial(\{y \ge 0\}\setminus \{r \le \eps\})} F \partial_{\vec{n}} F_2 y^a +\int_{(\partial\{r \ge \eps\})^+} F \partial_{\vec{n}} F_2 y^a.
\end{align*}
For the first term, and using the Green identity above, we have
\begin{align*}
\int_{\partial(\{y \ge 0\}\setminus \{r \le \eps\})} F \partial_{\vec{n}}F_2 y^a & = \int_{\partial(\{y \ge 0\}\setminus\{r \le \eps\})} \partial_{\vec{n}} F F_2 y^a - \int_{\{r\ge \eps\}^+} L_a F F_2\\
&   = -\int_{\Omega_\eps} \partial_y^a  F F+\int_{(\partial\{r \ge \eps\})^+} \partial_{\vec{n}} F F_2 y^a - \int_{\{r\ge \eps\}^+} L_a F F_2.
\end{align*}
If we denote 
\[
I_\eps := \int_{(\partial\{r \ge \eps\})^+} (F-F_2) \partial_{\vec{n}}F_2 y^a-\int_{(\partial\{r \ge \eps\})^+} \partial_{\vec{n}} (F-F_2) F_2 y^a
\]
we then have
\[
\int_{\Omega_\eps} F_2 \partial_y^a F_2 -\int_{\{r\ge \eps\}^+} FL_a F = I_\eps+\int_{\Omega_\eps}\partial_y^a F F -\int_{\{r\ge\eps\}^+} (F-F_2)L_a(F-F_2).
\]
Letting $\eps \downarrow 0$ and using that $\int_{\{y \ge 0\}} (F-F_2)L_a(F-F_2)\le 0$ (where the boundary term vanishes by scaling and because $F = F_2$ on $\{y = 0\}$) we obtain 
\begin{equation}
\label{eq.follows3}
\int_\Omega \tilde F \fls \tilde F \le -d_s\left(\int_\Omega F \partial_y^a F + \int_{\{y > 0\}} F L_a F+\lim_{\eps\downarrow 0} I_\eps \right),
\end{equation}
where it remains to be computed the explicit value of $\lim_{\eps \downarrow 0} I_\eps$ in terms of $F$. 
\\[0.1cm]
{\bf Step 2.}
 In order to do that, we use expansions of $F$ and $F_2$ in the spirit of those in Lemma~\ref{lem.form2} (from where we take the notation, as well). In this case, the role of the first order expansion $\frac{\mathcal{U}_s}{2^s}$ is played by $\frac{\mathcal{U}_s}{2^s r}$; while the $a$-harmonic function $\mathcal{V}_s$ now is $\tilde{\mathcal{V}}_s$ defined as
\[
\tilde{\mathcal{V}}_s := \frac{\mathcal{U}_s}{r} - \frac{H}{2} \frac{\mathcal{U}_s}{r} (\delta - r),\qquad L_a \tilde{\mathcal{V}}_s = - \frac{H\mathcal{U}_s y^a}{r^3} \left( s\delta r + (1-s) r^2 - \delta^2\right) = O(r^{-s}).
\] 
If we assume that $0\in \partial\Omega$, we can expand $F$ around a free boundary point as 
\[
F = h \frac{\mathcal{U}_s}{2^{s}r}+\dots = h_0 \tilde{\mathcal{V}}_s + \frac{\mathcal{U}_s}{2^s r}\left(h_1 (\delta - r) + h_2 r + A'\cdot x_\tau\right)+\dots,
\]
for some function $h\in C^{1,\alpha}$, and where $x_\tau$ denotes the directions tangent to $\partial\Omega$ (or perpendicular to the unit outward normal to $\partial \Omega$ on the thin space, $\nu$). 

In this way, the harmonic extension of $F|_{\{y = 0\}}$ towards $\{y \ge 0\}$ is  
\[
F_2 = h_0 \tilde{\mathcal{V}}_s + \frac{\mathcal{U}_s}{2^s r}\left(h_2 r + A'\cdot x_\tau\right)+\dots\quad\text{so that}\quad F - F_2 =  h_1 \frac{\mathcal{U}_s}{2^s r}(\delta -r)+\dots.
\]

Thus, in the definition of $I_\eps$ we can change variables and decompose the integral on $\{r = 0\}^+$ as an integral for $x\in \partial \Omega$ times an integral on $x+\eps(\nu\cos\theta + \hat{y}\sin\theta)$ for $\theta\in (0, \pi)$, and where $\hat y = \be_{n+1} = (0,\dots,0,1)\in \R^{n+1}$. Doing so, and plugging the previous functions on $I_\eps$, we obtain 
\[
\lim_{\eps\downarrow 0} I_\eps = 2^{-2s}\int_{\partial \Omega} h_0(x) h_1(x) \int_0^\pi (1+\cos(\theta))^{2s} (\cos\theta - 1)(\sin\theta)^{1-2s}\, d\theta\, d\sigma(x),
\]
where $h_0$ and $h_1$ are now functions corresponding to the respective coefficient at each boundary point $x\in \partial\Omega$. 
We obtain this result by observing that $\delta = \eps\cos\theta$ on $\{r\ge\eps\}^+$, $r = \eps$, $\partial_{\vec{n}}\delta = -\cos\theta$, and $\partial_{\vec{n}}r = -1$.  We now compute the innermost integral (using Mathematica 11.2 to do this computation)
\[
 \int_0^\pi (1+\cos(\theta))^{2s} (\cos\theta - 1)(\sin\theta)^{1-2s}\, d\theta = -\frac{2\pi s(1-s)}{\sin(\pi s)}
\]
to get
\begin{equation}
\label{eq.follows4}
\lim_{\eps\downarrow 0} I_\eps = -2^{-2s+1}\frac{\pi s(1-s)}{\sin(\pi s)} \int_{\partial\Omega} h_0(x) h_1(x).
\end{equation}
To finish, observe that $h_0(x) = \gamma_0(F)$, and that at first order 
\[
L_a F = h_1 L_a \left(\frac{\mathcal{U}_s}{2^sr} (\delta - r)\right)  = h_1 2 \frac{\mathcal{U}_s y^a}{2^sr^3}(rs - \delta)+\dots
\]
so that $h_1 = \frac{2^{1-s}r^3}{\mathcal{U}_s y^a(rs-\delta)}L_a F$ on $\partial\Omega$, understood as a limit. In particular, recalling the definition \eqref{eq.gamma-1}, 
$h_1 = s^{-1}2^{s-1} {\lambda}(F)$, and so the result follows joining \eqref{eq.follows1}-\eqref{eq.follows2}-\eqref{eq.follows3}-\eqref{eq.follows4}.
%\begin{align*}
%\int_\Omega (F_2 \partial_y^a F_2 - F \partial_y^a F) - \int_{\{y > 0\}} F L_a F & = \int_{\Omega} F\partial_y^a(F_2-F) - \int_{\{y > 0\}} FL_a(F-F_2)\\
%& =  \int_{\Omega} (F_2-F) \partial_y^a F- \int_{\{y > 0\}} (F-F_2) L_aF\\
%& \quad +\lim_{\eps \downarrow 0} \int_{\partial B_\eps^+} (F-F_2)\partial_{\vec{n}}(F-F_2)y^a
%\end{align*}
%where in the last equality we have used the following Green identity for the operator $L_a$,
%\[
%\int_{D} (gL_a f -fL_a g) = \int_{\partial D} (-f\partial_{\vec{n}} g y^a+g\partial_{\vec{n}}f y^a),
%\]
%for all pairs of functions $f$ and $g$ such that each of the previous terms is well-defined, and where $n$ denotes the outward normal to corresponding domain**.  Since $F_2 = F$ on $\Omega$ and $L_a F_2 = 0$ in $\{y > 0\}$ we get
%\[
%\int_\Omega (F_2 \partial_y^a F_2 - F \partial_y^a F) - \int_{\{y > 0\}} F L_a F  = -\int_{\{y > 0\}} (F-F_2)L_a(F-F_2) \ge 0. 
%\]
%In all, 
%\[
%\int_\Omega \tilde F \fls \tilde F \le -d_s\left(\int_\Omega F \partial_y^a F + \int_{\{y > 0\}} F L_a F\right)
%\]
%where $F$ is such that $G = F$ on $(\R^n\setminus \Omega)\times\{0\}$. This concludes the proof. 
%
%***
%In all, I think we have
%\[
%\int_\Omega \tilde F \fls \tilde F \le -d_s\left(\int_\Omega F \partial_y^a F + \int_{\{y > 0\}} F L_a F++\lim_{\eps \downarrow 0} \int_{ B_\eps^+} |\nabla(F-F_2)|^2y^a\right)
%\]
\end{proof}

\section{Axially symmetric stable cones}
\label{sec.6}
Let us use the stability condition to show that, at least in low dimensions, axially symmetric homogeneous (thus conical) solutions are either unstable or one-dimensional. Let us also fix 
\[
\Lambda = \Gamma(1+s)
\]
from now on, so that the fractional normal derivative at free boundary points is fixed to be 1. 

Namely, let us suppose that we have a (global) solution $u$ that is axially symmetric and $C^s$. That is, 
\begin{equation}
\label{eq.axsim}
u = u(x_1, x_2, \dots, x_n) = u( \s, \tau)\quad\text{where}\quad \s \coloneqq \sqrt{x_1^2+\dots+x_{n-1}^2}~\text{and}~\tau = x_n,
\end{equation}
 and it satisfies \eqref{eq.OPT} outside of the origin. We denote $\Omega = \{u > 0\}$.
 
 Let $\bar u:\R^{n}\times\R_+$ be the $a$-harmonic extension of $u$. Namely (recall \eqref{eq.La}.), 
\begin{equation}
\label{eq.aext}
 \bar u(x, 0) = u(x)\quad\text{for}~~x\in \R^n,\qquad\text{and}\qquad L_a \bar u = 0\quad\text{in}\quad \{y > 0\}. 
\end{equation}

Let us denote by 
\[u_\s \coloneqq \hat{\s}\cdot\nabla u\]
the derivative along the direction $\s$, so that $\hat{\s} \coloneqq \frac{\s}{|\s|}$, similarly we denote $\bar u_\s$. Then, the following stability condition holds for $u_\s$.

\begin{prop}
\label{prop.stab_E}
Let $u\in C^s(\R^n)$ be a global minimizer to \eqref{eq.energy}, in the sense \eqref{eq.energy2}-\eqref{eq.loc_min} for all $B\subset\R^n$; or a stable critical point to \eqref{eq.energy}, in the sense \eqref{eq.criticalpoint}-\eqref{eq.stablesolution}. Assume $u$ is that $\Omega \coloneqq \{u > 0\}$ is a $C^{2,\alpha}$ domain outside the origin.

Let us also assume that $u$ is axially symmetric (see \eqref{eq.axsim}). If we denote by $\bar u$ its $a$-harmonic extension, \eqref{eq.aext}, then we have that 
\begin{equation}
\label{eq.stab_E}
\int_{\{y > 0\}}  \bar u_\s^2 |\nabla \eta|^2y^a 
\ge (n-2) \int_{\{y > 0\}}\bar u_\s^2 \eta^2 \s^{-2} y^a 
\end{equation}
for all $\eta\in C^\infty_c(\overline{\R^{n+1}})$ such that $\partial_y \eta = 0$ on $\{y = 0\}$. 
\end{prop}

\begin{proof}
By multiplying our solution by a constant, we assume without loss of generality that $\Lambda = \Gamma(1+s)$. Let us start by noting that, since $u$ solves the (fractional) one-phase problem and by Lemma~\ref{lem.form2}, we have that the expansion of $u$ around free boundary points is
\[
u(x) = d^s(x) + A(x) d^{1+s} + \dots\qquad\text{for $x\in \{u > 0\}$,}
\]
where we recall that $d(x) = {\rm dist}(x, \{u = 0\})$, and we have defined $A(x) = \partial_\nu\left(u/d^{s}\right)$ (from the notation in Lemma~\ref{lem.form2}, the tangential direction of $\nabla \eta$ along the free boundary vanishes, since $\eta$ is constant there). We can similarly compute an expansion of $u_\s$ as 
\begin{equation}
\label{eq.notsimply}
u_\s(x) = sd^{s-1}(x) (\hat{\s}\cdot \nu) + (1+s) A(x) d^s (\hat{\s}\cdot\nu)+\dots\qquad\text{for $x\in \{u > 0\}$}
\end{equation}
where we recall $\hat{\s} = \s / |\s|$.
Notice that, from the previous two expressions, we have
\begin{equation}
\label{eq.relHu}
\frac{1+s}{s}\partial_\nu\left(\frac{u}{d^s}\right) \frac{u_\s}{d^{s-1}} = \partial_\nu\left(\frac{u_\s}{d^{s-1}}\right), 
\end{equation}
where we are using that $u_\s$ is a large solution, and so the values of $u_\s d^{1-s}$ and $\partial_\nu(u_\s d^{1-s})$ are well-defined on $\partial\Omega = \partial\{u = 0\}$, and they are equal to
\begin{equation}
\label{eq.75is}
\frac{u_\s}{d^{s-1}} = s(\hat{\s}\cdot\nu),\quad \partial_\nu\left(\frac{u_\s}{d^{s-1}} \right) = (1+s) A(x) (\hat{\s}\cdot\nu) \quad\text{on $\partial\Omega$}.
\end{equation}

Let us now consider the stability condition from Proposition~\ref{prop.main_E} in the case $\Lambda = \Gamma(1+s)$. Namely, for any $F$ such that ${\lambda}(F)$, $\gamma_0(F)$, and $\gamma_1(F)$ (recall \eqref{eq.gamma0}-\eqref{eq.gamma1}-\eqref{eq.gamma-1}) are well-defined and $F = 0$ on $(\R^n\setminus \Omega)\times\{0\}$ then 
\begin{equation}
\label{eq.becomes}
\begin{split}
\int_{\partial \Omega} & \gamma_0(F) \gamma_1(F) -\frac{1+s}{s}\int_{\partial \Omega}  \partial_\nu\left(\frac{u}{d^s}\right) \gamma_0(F)^2 -2^{s-1}\frac{1-s}{s}\int_{\partial \Omega} {\lambda}(F) \gamma_0(F) \,\leq  \\
& \qquad\qquad\le - \frac{d_s }{\Gamma(1+s)\Gamma(s)}\int_{\Omega} F \partial_y^a F -\frac{d_s}{\Gamma(1+s)\Gamma(s)}\int_{\{y > 0\}} F L_a F. 
\end{split}
\end{equation}

Take now, as test function $F$, $F = {\bar u}_\s \eta$ for some smooth, compactly supported $\eta$ such that $\eta|_{\{y = 0\}}$ is compactly supported outside of $\{{\s} = 0\}$ (so that, $\partial\Omega$ is smooth on ${\rm supp}~\eta$), $\eta\in C^\infty_c(\overline{\R^{n+1}}\setminus \{\s = y = 0\})$. Recall, also, that $\bar u$ denotes the $a$-harmonic extension of $u$ towards $\{y > 0\}$. Note that such choice of $F$ satisfies the condition that $F/|\nabla \mathcal{U}_s|$ is $C^{1,\alpha}$.

On the one hand, by means of \eqref{eq.relHu}, we have 
\begin{equation}
\label{eq.usethis}
\frac{1+s}{s}\int_{\partial \Omega} \partial_\nu\left(\frac{u}{d^s}\right) \gamma_0(F)^2  = \int_{\partial \Omega} \partial_\nu\left(\frac{u_\s}{d^{s-1}}\right) \frac{u_\s}{d^{s-1}}\eta^2 .
\end{equation}
On the other hand, let us compute ${\lambda}(F)$.  Differentiating the expression $L_a \bar u = 0$ in $\{y > 0\}$ in the $\hat{\s}$ direction, we obtain that 
\begin{equation}
\label{eq.desiredresult}
L_a {\bar u}_\s = \frac{n-2}{\s^2} {\bar u}_\s y^a.
\end{equation}

Thus, $L_a(F) = {\bar u}_\s L_a \eta + \eta L_a {\bar u}_\s+ 2y^a \nabla {\bar u}_\s\cdot \nabla \eta$, and from the definition of ${\lambda}(F)$, \eqref{eq.gamma-1}, and thanks to \eqref{eq.desiredresult} only the last terms survives, 
\[
{\lambda}(F) =\nabla \eta(x)\cdot \lim_{y \downarrow 0} y^{2-s}\nabla {\bar u}_\s(x, y).
\]

In order to compute it, let us consider the expansion of ${\bar u}_\s$ around $\partial\Omega$. Notice that ${\bar u}$ is the $a$-harmonic extension of our original $u$ towards $\{y > 0\}$, so the expansion around $\partial\Omega$ is not simply \eqref{eq.notsimply} and rather we have to consider the variable $y$ as well.

From the proof of Lemma~\ref{lem.form2}, and using the notation there, we have that at first order around free boundary points,
\[
{\bar u}(x, y) = 2^{-s} (\delta  + r)^{s} + \dots
\]
where we recall that $\delta$ denotes the signed distance to $\partial\Omega$ (in the first $n$ variables) and $r$ is the distance in $\R^{n+1}_+$ to $\partial \Omega$, that is, $r = (\delta^2+y^2)^\frac12$. In particular, differentiating the previous expression in the direction $\s$, we obtain an expansion of the derivative around $\partial\Omega$,
\[
{\bar u}_\s(x, y) = s2^{-s} (\delta+r)^{s-1} \hat{\s}\cdot \nabla (\delta +r)+\dots = s2^{-s} \frac{(\delta+r)^{s}}{r} (\hat{\s}\cdot \nu) +\dots,
\]
where we used that $\nabla r = \frac{1}{r}(\delta\nabla \delta, y)$ and $\nabla \delta = \nu$ at first order. Notice that we differentiate the expansion to obtain an expansion for the derivative: indeed, since we are assuming that $\Omega$ is a $C^{2,\alpha}$ domain, the function $\frac{\bar u}{2^{-s}(\delta+r)^s}\in C^{1,\alpha}$. That is, 
$
\bar u = 2^{-s}(\delta+r)^s \phi
$
for some $\phi\in C^{1,\alpha}$. We can now differentiate $\bar u$ to get $\nabla \bar u = s2^{-s}(\delta+r)^{s-1}(\nabla \delta+\nabla r)\phi +2^{-s}(\delta+r)^s\nabla\phi$. Now, the second term is lower order with respect to the first one, so we get the desired expansion. 

Similarly, 
\[
\nabla\eta(x) \cdot \nabla {\bar u}_\s(x, y) = s2^{-s}\nabla \eta \cdot \nu (sr  - \delta)\frac{(\delta+r)^s}{r^3}(\hat{\s}\cdot \nu)+\dots,
\]
so that 
\[
{\lambda}(F) = s^2 2^{-s+1} \eta_\nu (\hat{\s}\cdot \nu)
\]

Thus, \eqref{eq.becomes} becomes, using \eqref{eq.75is}-\eqref{eq.usethis} as well
\begin{equation}
\label{eq.putback}
s^3 \int_{\partial\Omega}\eta\partial_\nu \eta (\hat{\s}\cdot\nu)^2  \le - \frac{d_s}{ \Gamma(1+s)\Gamma(s)}\int_{\Omega} F \partial_y^a F -\frac{d_s}{ \Gamma(1+s)\Gamma(s)}\int_{\{y > 0\}} F L_a F.
\end{equation}

Let us now deal with the right-hand side of the previous expression. For simplicity, we are assuming that
\[
\partial_y^a \eta \equiv 0\quad\text{on}\quad \{y = 0\}
 \] for all $a\in (-1, 1)$ (we will later choose $\eta$ as such). Notice that the first term is then   
 \begin{equation}
 \label{eq.Fvanishes}
 \int_{\Omega} F \partial_y^a F = \int_{\Omega} \eta^2 (u_\s) \partial_y^a({\bar u}_\s) = 0,
 \end{equation}
 since $\partial_y^a({\bar u}_\s) = (\partial_y^a {\bar u})_\s$, and $-d_s \partial_y^a {\bar u} = (-\Delta)^s u$, so that $u_\s (\partial_y^a {\bar u})_\s   = u_\s C((-\Delta)^s u)_\s\equiv 0$ on $\Omega$. For the second term, we have 
 \[
 \int_{\{y > 0\}} F L_a F = \int_{\{y > 0\}} (\eta L_a \eta {\bar u}_\s^2 + 2 {\bar u}_\s \eta \nabla {\bar u}_\s \cdot \nabla \eta y^a + {\bar u}_\s L_a {\bar u}_\s \eta^2).
 \]
 
We notice that the three terms above are integrable, and we would like to integrate by parts the first term. However, such integration, if done directly, would yield non-integrable terms, and we have to be a bit more delicate with this step. Let us, then, integrate by parts the first term above. 

We denote by $E_\eps \coloneqq \{X\in \R^{n+1}_+ : {\rm dist}(X, \partial\Omega) = \eps\}$ and $E_{>\eps} \coloneqq \{X\in \R^{n+1}_+ : {\rm dist}(X, \partial\Omega) > \eps\}$. We then want to compute
\[
\int_{\{y > 0\}} \eta L_a \eta {\bar u}_\s^2 = \lim_{\eps\downarrow 0}\int_{E_{> \eps}} \eta L_a \eta {\bar u}_\s^2.
\]
Notice that, in $E_{> \eps}$, $u_\s$ is smooth and we can integrate by parts the previous expression, to obtain (recall $\partial_y^a \eta = 0$ on $\{y = 0\}$) 
\[
\int_{E_{> \eps}} \eta L_a \eta {\bar u}_\s^2 = -\int_{E_{> \eps}} \nabla \eta\cdot\nabla (\eta {\bar u}_\s^2) y^a + \int_{E_\eps}\vec{n}\cdot\nabla \eta \eta {\bar u}_\s^2 y^a d\mathcal{H}_n,
\]
where the second integral is performed on $E_\eps$ (which has Hausdorff dimension $n$). The vector $\vec{n}$ denotes the outward unit normal vector to $E_{> \eps}$. 

%In particular, notice that (for $s = \frac12$) the measure of $E_\eps$ is asymptotically like $\eps$, and ${\bar u}_\s^2$ behaves like $\eps^{-1}$ near $\partial\Omega$, so the second term above  is a finite non-zero contribution. 

Putting it all together and letting $\eps \downarrow 0$, we have that 
\[
\int_{\{y > 0\}} FL_a F = \int_{\{y > 0\}}\left( {\bar u}_\s L_a {\bar u}_\s \eta^2 - |\nabla \eta|^2 {\bar u}_\s^2y^a\right) + \lim_{\eps \downarrow 0 } \int_{E_\eps}\vec{n}\cdot\nabla \eta \eta {\bar u}_\s^2 y^a d\mathcal{H}_n.
\]

Let us compute this last term.  With this notation, the vector $\vec{n}$ corresponds to $-\nabla r$, and since we were assuming that $\partial_y \eta = 0$, $\vec{n}\cdot \nabla \eta = -\frac{\delta}{r}\nu\cdot\nabla \eta$ at first order. That is, 
\[
\lim_{\eps\downarrow 0} \int_{E_\eps}\vec{n}\cdot\nabla \eta \eta {\bar u}_\s^2 y^ad\mathcal{H}_n = -s^2  2^{-2s}\lim_{\eps\downarrow 0}\int_{E_\eps} \eta  \partial_\nu\eta \delta \frac{(\delta+r)^{2s}}{r^{3}} (\hat{\s}\cdot \nu)^2 y^a.
\]

Notice that $ r \equiv \eps$ on $E_\eps$. We decompose the integral on $E_\eps$ as an integral for $x\in \partial\Omega$ times an integral on $x+\eps (\nu\cos\theta + \hat{y}\sin\theta)$ for $\theta\in (0,\pi)$, where $\hat y = \boldsymbol{e}_{n+1} = (0,\dots,0,1)\in \R^{n+1}$. Such decomposition has Jacobian $\eps$ at leading order, so (also using that $\eta \partial_\nu\eta (\hat \s \cdot\nu)^2$ is smooth)
\begin{align*}
&\lim_{\eps\downarrow 0}  \int_{E_\eps}\vec{n}\cdot\nabla \eta \eta {\bar u}_\s^2 y^a d\mathcal{H}_n \,= \\
&\qquad\qquad = -s^2 2^{-2s} \lim_{\eps\downarrow 0}  \eps^{-3}\int_{\partial \Omega} \eta (x) \partial_\nu \eta(x) (\hat \s\cdot \nu)^2(x) \eps \int_0^\pi \mathcal{F}(x, \theta, \eps)\,d\theta\, d\sigma(x),
\end{align*}
where, if we denote by $f_0$ the function $f_0(x, y) = \delta (\delta + r)^{2s} y^a$, then 
\[
\mathcal{F}(x, \theta, \eps) = f_0 (x+\eps (\nu\cos\theta + \hat{y}\sin\theta)).
\]
In particular, at leading order we have that $\delta(x+\eps (\nu\cos\theta + \hat{y}\sin\theta)) = \eps\cos(\theta)$ and $r = \eps$ so 
\[
\mathcal{F}(x, \theta, \eps) = \eps^2 \cos(\theta) (1+\cos(\theta))^{2s} \sin(\theta)^{1-2s}	.
\]

Now we compute 
\[
\int_0^\pi \cos(\theta) ( 1+\cos(\theta))^{2s} \sin(\theta)^{1-2s} d\theta = \frac{2\pi s^2}{\sin(\pi s)}.
\]
(We used Mathematica 11.2 again to do this computation.) Putting everything together, we obtain that 
\[
\lim_{\eps\downarrow 0}  \int_{E_\eps}\vec{n}\cdot\nabla \eta \eta {\bar u}_\s^2 y^a d\mathcal{H}_n = -2^{-2s}\frac{2\pi s^4}{\sin(\pi s)}\int_{\partial\Omega}\eta_\nu\eta (\hat{\s}\cdot \nu)^2
\]
and therefore, 
\[
\int_{\{y > 0\}} FL_a F = \int_{\{y > 0\}}\left( {\bar u}_\s L_a {\bar u}_\s \eta^2 - |\nabla \eta|^2 {\bar u}_\s^2y^a\right) -2^{-2s}\frac{2\pi s^4}{\sin(\pi s)}\int_{\partial\Omega}\eta_\nu\eta (\hat{\s}\cdot \nu)^2.
\]
Putting it back in \eqref{eq.putback} and recalling that \eqref{eq.Fvanishes}, 
\[
\int_{\{y > 0\}} |\nabla \eta|^2 {\bar u}_\s^2y^a 
\ge \int_{\{y > 0\}}{\bar u}_\s L_a {\bar u}_\s \eta^2 
\]
for all $\eta \in C^\infty_c(\overline{\R^{n+1}}\setminus \{\s = y = 0\})$ such that the previous expression is well-defined on both sides, and with $\partial_y^a \eta = 0$ on $\{y = 0\}$. We have used that (recalling \eqref{eq.ds})
 \[
 s^3 = \frac{d_s}{\Gamma(1+s)\Gamma(s)} 2^{-2s}\frac{2\pi s^4}{\sin(\pi s)}. 
 \]
Using now \eqref{eq.desiredresult} we get the desired result.

 We end up by noticing that, from the regularity of $\eta$, if $\partial_y \eta = 0$ on $\{y = 0\}$ then $\partial_y \eta = 0$ on $y = 0$ for all $a\in (-1, 1)$. Also, the fact that $\eta$ is compactly supported outside of $\{\s = y = 0\}$ can also be removed by standard arguments using the axial symmetry of $u$ (see \cite[Proof of Proposition 1.3]{San18} and also \cite[Proof of Theorem 1.7]{FR19}).
\end{proof}

Let us now use an appropriate test function in the expression \eqref{eq.stab_E} to deduce properties of axially-symmetric global stable solutions. We will apply such result to cones, so from now on we will assume that $u$ is a $C^s$ global solution (local minimizer or stable solution)  which is $s$-homogeneous. 
In particular, $\bar u$ is also $s$-homogeneous, $u_\s$ is $(s-1)$-homogeneous and ${\bar u}_\s$ is $(s-1)$-homogeneous as well. 

\begin{proof}[Proof of Theorem~\ref{thm.low-dim}]
Let $\alpha>0$ to be fixed. For any $\eps > 0$ and $R > 1$, let us define $\eta_{\eps, R}$ as 
\[
  \eta_{\eps, R} = \left\{ \begin{array}{ll}
  \s^{-\alpha}\rho_R & \textrm{ if } \s > \eps\\
  \eps^{-\alpha}\rho_R  & \textrm{ if } \s \le \eps
  \end{array}\right.,\qquad \rho_{R}(X) = \left\{ \begin{array}{ll}
  1 & \textrm{ in } B_R^+\\
  0 & \textrm{ in } \R^{n+1}_+ \setminus B^+_{2R}
  \end{array}\right.,
\]
where $\rho_R \in C^\infty_c(\overline{\R^{n+1}_+})$, $\rho_R\ge 0$, is a smooth, radial, non-increasing function such that $|\nabla \rho_R|\le C R^{-1}$ for some fixed universal constant $C$. We have denoted here, as an abuse of notation, 
\[
B_\rho^+ \coloneqq \{(x, y)\in \R^n\times \R: |x|^2+y^2 \le \rho\quad\text{and}\quad y \ge 0\}.
\]

Then, 
\[
  |\nabla \eta_{\eps, R}|^2 \le \left\{ \begin{array}{ll}
  \alpha^2 \s^{-2\alpha-2} \rho_R^2 & \textrm{ in } B_R^+\cap \{\s > \eps\}\\
    \alpha^2(1+\delta) \s^{-2\alpha-2} \rho_R^2 + C_\delta \s^{-2\alpha}|\nabla \rho_R|^2 & \textrm{ in } B^+_{2R}\setminus B^+_R\cap \{\s > \eps\}\\
  \eps^{-2\alpha}|\nabla \rho_R|^2  & \textrm{ if } \s \le \eps.
  \end{array}\right.
\]
for any $\delta > 0$, and where $C_\delta$ is a constant depending on $\delta$. 

Let us now use $\eta = \eta_{\eps, R}$ as a test function in \eqref{eq.stab_E} (notice that we can do so by approximation, since $\eta_{\eps, R}$ is Lipschitz). On the right-hand side of \eqref{eq.stab_E} we get 
\begin{equation}
\label{eq.combine0}
\int_{\{y > 0\}} {\bar u}_\s^2\eta^2\s^{-2} y^a\, dx  = \int_{\{\s > \eps\}\cap B_{2R}^+} {\bar u}_\s^2 \s^{-2\alpha-2}\rho^2_R y^a\, dx + \eps^{-2\alpha} \int_{\{\s\le \eps\}\cap B^+_{2R}} {\bar u}_\s^2 \rho_R^2 y^a. 
\end{equation}
On the other hand, on the left hand-side we have
\begin{align*}
\int_{\{y > 0\}} {\bar u}_\s^2|\nabla \eta|^2 & y^a\, dx \le  \alpha^2(1+\delta)\int_{\{\s > \eps\}\cap B_{2R}^+} {\bar u}_\s^2 \s^{-2\alpha-2} \rho_R^2 y^a\, dx\,+ \\
& + C_\delta R^{-2}\int_{B_{2R}^+\setminus B_R^+\cap \{\s > \eps\}} {\bar u}_\s^2 \s^{-2\alpha} y^a
 + C \eps^{-2\alpha}R^{-2}\int_{\{\s \le \eps\}\cap B^+_{2R}}{\bar u}^2_\s y^a\, dx.
\end{align*}
where we have used that $|\nabla \rho_R|\le C R^{-1}$. Combining this with \eqref{eq.combine0} and \eqref{eq.stab_E} we deduce that (also notice that $\nabla \rho_R$ vanishes on $B_R^+$)
\begin{equation}
\label{eq.boundtwoterms}
\begin{split}
(n-&2-\alpha^2(1+\delta))\int_{\{\s > \eps\}\cap B_{2R}^+} {\bar u}_\s^2 \s^{-2\alpha-2}\rho^2_R y^a\, dx\,\leq \\
& \le C_\delta R^{-2}\int_{B_{2R}^+\setminus B_R^+\cap \{\s > \eps\}} {\bar u}_\s^2 \s^{-2\alpha} y^a  + CR^{-2} \eps^{-2\alpha}\int_{B^+_{2R}\setminus B_R^+\cap \{\s \le \eps\}}{\bar u}^2_\s y^a\, dx.
\end{split}
\end{equation}

In particular, for the previous inequality to be useful we will assume, from now on, that
\[
n > 2+\alpha^2
\]
(and we can choose $\delta$ appropriately).

Let us bound the two terms on the right-hand side of \eqref{eq.boundtwoterms}. We start with the second term, by scaling the integral and using that $\bar u_\s$ is $(s-1)$-homogeneous (recall that $a = 1-2s$):
\[
\int_{B^+_{2R}\setminus B_R^+\cap \{\s \le \eps\}}{\bar u}^2_\s y^a\, dx = R^n \int_{B^+_{2}\setminus B_1^+\cap \{\s \le \eps/R\}}{\bar u}^2_\s y^a\, dx.
\]
Notice now that $\bar u_\s^2$ is bounded in the region $B^+_{2}\setminus B_1^+\cap \{\s \le \eps/R\}$ if $\eps/R$ is small enough. Indeed, for $\eps/R$ small enough, $B_2^+\setminus B_1^+\cap \partial \Omega = \varnothing$ (we are fixing the function $u$) and if $\{y = 0\}$ we have that either $u_\s\equiv 0$ in this set, or it is $s$-harmonic; then, for $\{y > 0\}$ we can use classical estimates for $L_a$-extensions (see, for example, \cite{JN17}). 

Thus, the last integral can just be bound by $ |B^+_{2}\setminus B_1^+\cap \{\s \le \eps/R\}| = C \eps^{n-1} R^{-n+1}$ and we get 
\[
 CR^{-2} \eps^{-2\alpha}\int_{B^+_{2R}\setminus B_R^+\cap \{\s \le \eps\}}{\bar u}^2_\s y^a\, dx \le C\eps^{n-1-2\alpha} R^{-1}.
\]
On the other hand, 
\[
\int_{B_{2R}^+\setminus B_R^+\cap \{\s > \eps\}} {\bar u}_\s^2 \s^{-2\alpha} y^a  = R^{n-2\alpha}\int_{B_{2}^+\setminus B_1^+\cap \{\s > \eps/R\}} {\bar u}_\s^2 \s^{-2\alpha} y^a.
\]

We now notice that the last integral is bounded by a constant independently of $\eps$ (and therefore, we can let $\eps\downarrow 0$). Indeed, we want to show that
\[
\int_{B_2^+\setminus B_1^+} \bar u_\s^2\s^{-2\alpha}y^a\le C
\] 
for some $C$ that might depend on $u$. 

We separate the integral for $\{\s > c\}$ and $\{\s \le c\}$, for some $c>0$ small enough such that $B_2^+\setminus B_1^+\cap \{\s \le c\} \cap \partial\Omega = \varnothing$. As before, in this case, $\bar u_\s$ is bounded in $B_2^+\setminus B_1^+\cap \{\s \le c\} $ so we only need to show that 
\[
\int_{B_2^+\setminus B_1^+} \s^{-2\alpha} y^a \le C
\]
for some $C$. Notice now that, by using $dx dy \mapsto \s^{n-2} d\s dt dy$ we have
\[
\int_{B_2^+\setminus B_1^+} \s^{-2\alpha} y^a \le C\int_0^1 y^a\, dy \int_0^1 \s^{-2\alpha+n-2}\, d\s,
\]
which is bounded if $a > -1$ and $-2\alpha+n-2 > -1$. The first part always holds, and the second part is true since by assumption $n > 2+\alpha^2 \ge 1+2\alpha$.

On the other hand, for $\{\s > c\}$ we just need to show that 
\[
\int_{B_2^+\setminus B_1^+} \bar u_\s^2y^a \le C.
\]
This is just a localised $W^{1,2}(y^a dX)$ norm (the extended $W^{s,2}$ norm), which is bounded for solutions.

Putting everything together, we have obtained that
\[
(n-2-\alpha^2(1+\delta))\int_{\{\s > \eps\}\cap B_{2R}^+} {\bar u}_\s^2 \s^{-2\alpha-2}\rho^2_R y^a\, dx \le C\eps^{n-1-2\alpha} R^{-1} + CR^{n-2\alpha-2}.
\]

Since $n > 2+\alpha^2$, we have $n-1-2\alpha> 0$ and letting $\eps\downarrow 0$ we get 
\[
\int_{B_{R}^+} {\bar u}_\s^2 \s^{-2\alpha-2} y^a\, dx \le CR^{n-2\alpha-2}.
\]

Now, if $n -2\alpha -2 < 0$, by letting $R\to \infty$ we obtain that $\bar u_\s\equiv 0$ in $\R^{n+1}_+$, so that $u_\s \equiv 0$. That is, $u$ depends only on $t$ and it is one-dimensional. We just need to check for which $n$ the previous conditions can be satisfied for some $\alpha > 0$. 

We have that
\[
2\alpha> n-2 >\alpha^2 \quad\Longrightarrow\quad \left(\frac{n-2}{2}\right)^2 < n-2,
\]
which holds for $2 < n < 6$, with $n = 6$ being the critical case. The case $n = 2$ has been already shown in Corollary~\ref{cor.2D}, so $u$ is one-dimensional for $n \le 5$. 
\end{proof}

\section{Numerical stability condition for axially symmetric cones}
\label{sec.7}

Let us consider $\mathcal{C}_{s, n}$ the axially symmetric cone
\[
\mathcal{C}_{s, n} \coloneqq \{x = (x', x_n)\in \R^{n-1}\times\R : |x'|> \beta_{n,s} |x_n| \}
\]
where $\beta_{n,s} > 0$ is the unique constant such that there exists an $s$-homogeneous solution $u_s$ to
\[
\left\{
\begin{array}{rcll}
\fls u_s & = & 0 & \quad\textrm{in } \mathcal{C}_{s, n}\\
u_s & = & 0 & \quad\textrm{in } \R^n\setminus \mathcal{C}_{s, n} \vspace{1mm} \\
\displaystyle\frac{u_s}{d^{s}} & = & 1& \quad \textrm{on }\partial \Co_{s, n}.
\end{array}
\right.
\]

In this section, we will study what is the expression of the stability condition from Theorem~\ref{thm.main} when applied to $u_s$, for radial functions $f = f(r)$, $r = |x|$. We already know that cones $\mathcal{C}_{s, n}$ are unstable for $n \le 5$ by Theorem~\ref{thm.low-dim}, and we expect them to be stable for $n \ge 7$. We believe that this approach might be useful to understand the case $n = 6$, in which we expect axially-symmetric cones to be unstable (given that the previous proposition seemed to hold until $n = 6-\delta$ for any $\delta > 0$ independently of $s\in (0, 1)$, and the result holds for $s = 1$.

We do so by finding an inequality that can be numerically checked, whose validity would imply the instability of the conical solution $u_s$. We refer to \cite{CJK04} for the analogous result for $s=1$. 

Let us denote, from now on, $\mathcal{C} \coloneqq \mathcal{C}_{n, s}$. The stability condition \eqref{eq.stabilityFin} when applied to the cone $\mathcal{C}$ is 
\begin{equation}
\label{eq.stab_fg}
\int_{\partial\Co} H_{\Co}  f^2 d\sigma \le \int_{\partial\Co\times\partial\Co} \left(f(x) - f(y)\right)^2\mathcal{K}_{\Co}(x, y)d\sigma(x)d\sigma(y)
\end{equation}
with
\[
H_{\mathcal{C}}(x) \coloneqq \int_{\partial \Co} |\nu(x) - \nu(y)|^2 \mathcal{K}_{\Co}(x, y) d\sigma(y),
\]
where $\mathcal{K}_\Co(x, y)$ is our boundary kernel, obtained from the Green function $G_{\Co}(x, y)$ as 
\[
\mathcal{K}_\Co (x, y) =  \lim_{\substack{\Co\ni \bar x\to x\\\Co\in \bar y \to y }} \frac{G_{\Co}(\bar x ,\bar y)}{d^s(\bar x)d^s(\bar y)}. 
\]

Notice, moreover, that from the symmetry of our problem and the fact that $\mathcal{K}_\Co$ is $(-n)$-homogeneous, $H_\Co(\lambda x) = \lambda^{-1}H_\Co(x)$, so
\[
H_\Co(x) = \frac{H_1}{|x|},\qquad H_1 = \int_{\partial\Co}|\nu(x_1) - \nu(y)|^2\mathcal{K}_\Co(x_1, y)d\sigma(y)
\]
where $x_1\in \partial\Co\cap\partial B_1$ is arbitrary.

Let $f = f(|x|) = f(r)$ be given by $f(r) = r^{\frac{2-n}{2}} g(r)$ for some $g$. Then, the left-hand side of \eqref{eq.stab_fg} can be rewritten as
\[
\int_\Co H_\Co f^2 d\sigma = H_1 |\partial\Co \cap \mathbb{S}^{n-1}|\int_0^\infty g^2(r) \frac{dr}{r},
\]
where $|\partial\Co \cap \mathbb{S}^{n-1}|$ is explicit depending only on the angle $\beta_{n, s}$ in the definition of $\Co = \Co_{n, s}$ (and $n$). 

For the fractional semi-norm part, let us define for any $x\in \Co$ with $|x| = r$,
\[
\tilde{\mathcal{K}}_{\Co}:\R_+\times\R_+ \to \R,\qquad \tilde{\mathcal{K}}_{\Co}(r, t) \coloneqq t^{2-n} \int_{|y| = t} \mathcal{K}_\Co(x, y) d\sigma(y)
\]
(notice that in this definition, the value is independent of $x$, as long as $|x| = r$, by symmetry).
Then, 
\begin{align*}
\int_{\partial\Co}\int_{\partial\Co} & \big(f(x) - f(y)\big)^2\mathcal{K}_{\Co}(x, y)d\sigma(x)d\sigma(y) =\\
\quad &= 2\int_{\partial\Co} f(x) \int_{\partial\Co} \left(f(x) - f(y)\right)\mathcal{K}_{\Co}(x, y)d\sigma(x)d\sigma(y)\\
\quad & = 2\int_0^\infty \int_{|x| = r} f(r) \int_0^\infty t^{n-2}(f(r) - f(t))\tilde{\mathcal{K}}_\Co(r, t) dt\, dr\\
\quad & = 2|\partial\Co \cap \mathbb{S}^{n-1}|\int_0^\infty f(r)r^{n-2} \int_0^\infty t^{n-2}(f(r) - f(t))\tilde{\mathcal{K}}_\Co(r, t) dt\, dr\\
\quad & = 2|\partial\Co \cap \mathbb{S}^{n-1}|\int_0^\infty g(r) r^{\frac{n}{2}}\int_0^\infty t^{n-2}(g(r)r^{\frac{2-n}{2}} - g(t)t^{\frac{2-n}{2}})\tilde{\mathcal{K}}_\Co(r, t) dt\, \frac{dr}{r}.
\end{align*}

If we denote 
\[
(\tilde \Lambda g )(r) \coloneqq 2r^{\frac{n}{2}}\int_0^\infty t^{n-2}(g(r)r^{\frac{2-n}{2}} - g(t)t^{\frac{2-n}{2}})\tilde{\mathcal{K}}_\Co(r, t) dt
\]
then our condition \eqref{eq.stab_fg} is 
\[
H_1 \int_0^\infty g^2(r) \frac{dr}{r} \le \int_0^\infty g(r) \tilde \Lambda g (r) \frac{dr}{r}.
\]
We now want to apply the Mellin transform ($\tilde h(\xi)  \coloneqq \int_0^\infty h(r) r^{-i\xi}\frac{dr}{r}$) and use Plancherel's theorem. To do so, notice that $\tilde \Lambda$ is invariant under dilations, or $0$-homogeneous ($[\Lambda g(\lambda \cdot) ](r) = [\Lambda g] (\lambda r)$), and so it is represented by a Fourier-Mellin multiplier operator on $L^2(\R^+, dr/r)$, with symbol that we denote $m(\xi)$ (cf. \cite{CJK04}). Using Plancherel's theorem to the previous inequality we have 
\[
H_1\int_{-\infty}^\infty|\tilde g(\xi)|^2d\xi \le \int_{-\infty}^\infty|m(\xi)||\tilde g(\xi)|
\]
for all $\tilde g$ such that $g\in C^\infty_0(0, \infty)$. This class is dense in $L^2(\R_+, dr/r)$ and thus we have
\[
H_1\le \inf_{\xi}|m(\xi)|\le m(0). 
\] 
The value of $m(0)$ can be computed as $\tilde \Lambda (1)$ (the operator $\tilde \Lambda$ applied to the constant function equal to 1), which by homogeneity is constant (again, cf. \cite{CJK04}), and so 
\[
m(0) = [\tilde \Lambda(1)](1)  =2\int_0^\infty t^{n-2}\big(1-t^{\frac{2-n}{2}}\big)\tilde{\mathcal{K}}_\Co(1, t)\, dt. 
\]
Finally, notice that $\tilde{\mathcal{K}}_\Co(1, t) = t^{-n} \tilde{\mathcal{K}}_\Co(1, 1/t)$, so
\begin{align*}
m(0) & = 2\int_0^1 t^{n-2}\big(1-t^{\frac{2-n}{2}}\big)\tilde{\mathcal{K}}_\Co(1, t)\, dt+2\int_0^1 t^{2-n}\big(1-t^{\frac{n-2}{2}}\big)t^{n}\tilde{\mathcal{K}}_\Co(1, t)\, \frac{dt}{t^2}\\
& = 2\int_0^1\big(1-t^{\frac{n-2}{2}}\big)^2 \tilde{\mathcal{K}}_\Co(1, t)\, dt. 
\end{align*}

In all, we need to check for which $n$ does the following inequality fail, to deduce that for those $n$ the corresponding cones is unstable. 
\[
\int_{\partial\Co}|\nu(x_1) - \nu(y)|^2\mathcal{K}_\Co(x_1, y)d\sigma(y)\le 2\int_0^1\big(1-t^{\frac{n-2}{2}}\big)^2 \tilde{\mathcal{K}}_\Co(1, t)\, dt,
\]
which can be numerically computed once the aperture $\beta_{n, s}$ and the corresponding Green function $G_{\mathcal{C}}$ are computed. As said before, we already know the cases $n \le 5$, so we are only interested in the case $n = 6$.


\begin{thebibliography}{99}

\bibitem[Aba15]{Aba15}
	\newblock N. Abatangelo,
	\newblock {\em Large $s$-harmonic functions and boundary blow-up solutions for the fractional Laplacian},
	\newblock Discrete Contin. Dyn. Syst. A 35 (2015), 5555-5607.

\bibitem[AR20]{AR20}
	\newblock N. Abatangelo, X. Ros-Oton, 
	\newblock {\em Obstacle problems for integro-differential operators: higher regularity of free boundaries},
	\newblock Adv. Math. 360 (2020), 106931.


\bibitem[All12]{All12}
	\newblock M. Allen, 
	\newblock {\em Separation of a lower dimensional free boundary in a two-phase problem},
	\newblock Math. Res. Lett. 19 (2012), 1055-1074.
	
\bibitem[AC81]{AC81}
	\newblock H. Alt, L. Caffarelli, 
	\newblock {\em Existence and regularity for a minimum problem with free boundary},
	\newblock J. Reine Angew. Math 325 (1981), 105-144.
	


\bibitem[ACF82]{ACF82}
	\newblock H. Alt, L. Caffarelli,  A. Friedman,
	\newblock {\em Asymmetric jet flows},
	\newblock Comm. Pure Appl. Math. 35 (1982), 29-68.

\bibitem[ACF82b]{ACF82b}
	\newblock H. Alt, L. Caffarelli,  A. Friedman,
	\newblock {\em Jet flows with gravity},
	\newblock J. Reine Angew. Math. 331 (1982), 58-103.
	
\bibitem[ACF83]{ACF83}
	\newblock H. Alt, L. Caffarelli,  A. Friedman,
	\newblock {\em Axially symmetric jet flows},
	\newblock Arch. Rational Mech. Anal. 81 (1983), 97-149.
	
\bibitem[Bog97]{Bogdan} K. Bogdan, \emph{The boundary Harnack principle for the fractional Laplacian}, Studia Math. 123 (1997), 43-80.
	
\bibitem[BL82]{BL82} J. D. Buckmaster, G. S. Ludford, \emph{Theory of Laminar Flames}, Cambridge Univ. Press, Cambridge, 1982.

\bibitem[CC04]{CC04} X. Cabr\'e, A. Capella, \emph{On the stability of radial solutions of semilinear elliptic equations in all of $\R^n$}, C. R. Acad. Sci. Paris, Ser. I 338 (2004), 769-774.


\bibitem[CCS20]{CCS20} X. Cabr\'e, E. Cinti, J. Serra, \emph{Stable $s$-minimal cones in $\R^3$ are flat for $s\sim 1$}, J. Reine Angew. Math. 764 (2020), 157-180.

\bibitem[CFRS20]{CFRS20} X. Cabr\'e, A. Figalli, X. Ros-Oton, J. Serra, \emph{Stable solutions to semilinear elliptic equations are smooth up to dimension 9}, Acta Math. 224 (2020), 187-252.

\bibitem[CR13]{CR13} X. Cabr\'e, X. Ros-Oton, \emph{Regularity of stable solutions up to dimension 7 in domains of double revolution}, Comm. Partial Differential Equations 38 (2013), 135-154.

%\bibitem[CT09]{CT09} X. Cabr\'e, J. Terra, \emph{Saddle-shaped solutions of bistable diffusion equations in all of $\R^{2m}$}, J.~Eur. Math. Soc. 11 (2009), 819-843.



\bibitem[Caf87]{Caf87}
	\newblock L. Caffarelli, 
	\newblock {\em A Harnack inequality approach to the regularity of free boundaries. I. Lipschitz free boundaries are $C^{1,\alpha}$},
	\newblock Rev. Mat. Iberoam. 3 (1987), 139-162.
		
\bibitem[Caf88]{Caf88}
	\newblock L. Caffarelli, 
	\newblock {\em A Harnack inequality approach to the regularity of free boundaries. III. Existence theory, compactness, and dependence on $X$},
	\newblock  Ann. Scuola Norm. Sup. Pisa Cl. Sci. 15 (1988), 583-602.
	
\bibitem[Caf89]{Caf89}
	\newblock L. Caffarelli, 
	\newblock {\em A Harnack inequality approach to the regularity of free boundaries. II. Flat free boundaries are Lipschitz},
	\newblock  Comm. Pure Appl. Math. 42 (1989), 55-78.
	
\bibitem[CJK04]{CJK04}
	\newblock L. Caffarelli, D. Jerison, C. Kenig,
	\newblock {\em Global energy minimizers for free boundary problems
and full regularity in three dimensions},
	\newblock Contemp. Math. 350 (2004), 83-97.

\bibitem[CRS10]{CRS10}
	\newblock L. Caffarelli, J. Roquejoffre, Y. Sire,
	\newblock {\em Variational problems with free boundaries for the fractional Laplacian},
	\newblock J. Eur. Math. Soc. 12 (2010), 1151-1179.
	
\bibitem[CS05]{CS05}
	\newblock L. Caffarelli, S. Salsa, 
	\newblock {\em 	A Geometric Approach to Free Boundary Problems},
	\newblock Graduate Studies in Mathematics, 68. American Mathematical Society, Providence, RI, 2005. x+270.

\bibitem[CS07]{CS07}
	\newblock L. Caffarelli, L. Silvestre,
	\newblock {\em An extension problem related to the fractional Laplacian},
	\newblock Comm. Partial Differential Equations 32 (2007),
1245-1260.


\bibitem[CV95]{CV95} L. Caffarelli, J. L. V\'azquez, \emph{A free-boundary problem for the heat equation arising in flame propagation}, Trans. Amer. Math. Soc. 347 (1995), 411-441.

\bibitem[CGV21]{CGV20}
	\newblock H. Chan, D. G\'omez-Castro, J. L. V\'azquez,
	\newblock {\em Blow-up phenomena in nonlocal eigenvalue problems: when theories of $L^1$ and $L^2$ meet},
	\newblock J. Funct. Anal. 280 (2021), 108845. 
	
	
	
\bibitem[CDDS11]{CDDS11}
	\newblock A. Capella, J. D\'avila, L. Dupaigne, Y. Sire, 
	\newblock {\em Regularity of Radial Extremal Solutions for Some Non-Local Semilinear Equations,
}
	\newblock Comm. Partial Differential Equations 36 (2011), 1353-1384. 
	
	
\bibitem[CS98]{CS98}
	\newblock Z. Chen, R. Song,
	\newblock {\em Estimates on Green functions and Poisson kernels for symmetric stable processes},
	\newblock Math. Ann. 312 (1998),
465-501.



\bibitem[DDW18]{DDW18} J. D\'avila, M. Del Pino, J. Wei, \emph{Nonlocal $s$-minimal surfaces and Lawson cones}, J. Differential Geom. 109 (2018), 111-175.



\bibitem[DJ09]{DJ09} D. De Silva, D. Jerison, \emph{A singular energy minimizing free boundary}, J. Reine Angew. Math. 635 (2009), 1-22.


\bibitem[DR12]{DR12}
	\newblock D. De Silva, J. Roquejoffre,
	\newblock {\em Regularity in a one-phase free boundary
problem for the fractional Laplacian},
	\newblock Ann. Inst. H. Poincar\'e Anal. Non Lin\'eaire, 29 (2012), 335-367.


\bibitem[DS12]{DS12}
	\newblock D. De Silva, O. Savin,
	\newblock {\em $C^{2,\alpha}$ regularity of flat free boundaries for the thin one-phase problem},
	\newblock J. Differential Equations 253 (2012), no. 8, 2420-2459.

\bibitem[DS15]{DS15b}
	\newblock D. De Silva, O. Savin,
	\newblock {\em $C^\infty$ regularity of certain thin free boundaries},
	\newblock Indiana Univ. Math. J. 64 (2015), 1575-1608.
	
\bibitem[DS15b]{DS15}
	\newblock D. De Silva, O. Savin,
	\newblock {\em Regularity of Lipschitz free boundaries for the thin one-phase problem},
	\newblock J. Eur. Math. Soc. 17 (2015), 1293-1326.
	
\bibitem[DSS14]{DSS14}
	\newblock D. De Silva, O. Savin, Y. Sire, 
	\newblock {\em A one-phase problem for the fractional Laplacian: regularity of flat free boundaries},
	\newblock Bull. Inst. Math. Acad. Sin. (N.S.) 9 (2014), 111-145.
	
	
\bibitem[DRV17]{DRV17}
	\newblock S. Dipierro, X. Ros-Oton, E. Valdinoci,
	\newblock {\em Nonlocal problems with Neumann boundary conditions},
	\newblock Rev. Mat. Iberoam. 33 (2017), 377-416.
	
\bibitem[EE19]{EE19}
	\newblock N. Edelen, M. Engelstein,
	\newblock {\em Quantitative stratification for some free-boundary problems},
	\newblock Trans. Amer. Math. Soc. 371 (2019), 2043-2072.

\bibitem[EKPSS21]{EKPSS20}
	\newblock M. Engelstein, A. Kauranen, M. Prats, G. Sakellaris, Y. Sire,
	\newblock {\em Minimizers for the thin one-phase free boundary problem},
	\newblock Comm. Pure Appl. Math. 74 (2021), 1971-2022.
	
	
\bibitem[ESV20]{ESV20}
	\newblock M. Engelstein, L. Spolaor, B. Velichkov,
	\newblock {\em Uniqueness of the blowup at isolated singularities for the Alt-Caffarelli functional},
	\newblock Duke Math. J. 169 (2020), 1541-1601.

%\bibitem[FS20]{FS20}
%	\newblock J. Felipe-Navarro, T. Sanz-Perela,
%	\newblock {\em Uniqueness and stability of the saddle-shaped solution to the fractional Allen-Cahn equation},
%	\newblock Rev. Mat. Iberoam., in press (2020). 
%	
%\bibitem[FS20b]{FS20b}
%	\newblock J. Felipe-Navarro, T. Sanz-Perela,
%	\newblock {\em Semilinear integro-differential equations, I: odd solutions with respect to the Simons cone},
%	\newblock J. Funct. Anal.  278 (2020), 108309, 48 pp.
	
%\bibitem[FS19]{FS19}
%	\newblock J. Felipe-Navarro, T. Sanz-Perela,
%	\newblock {\em Semilinear integro-differential equations, II: one-dimensional and saddle-shaped solutions to the Allen-Cahn equation},
%	\newblock preprint arXiv (2019). 
	
\bibitem[FR19]{FR19}
	\newblock X. Fern\'andez-Real, X. Ros-Oton, 
	\newblock {\em On global solutions to semilinear elliptic equations related to the one-phase free boundary problem},
	\newblock Discrete Contin. Dyn. Syst. A 39 (2019), 6945-6959.
	
\bibitem[FFMMM15]{FFMMM15}
	\newblock A. Figalli, N. Fusco, F. Maggi, V. Millot, M. Morini, 
	\newblock {\em Isoperimetry and stability properties of balls with respect to nonlocal energies},
	\newblock Comm. Math. Phys. 336 (2015), 441-507.
	
\bibitem[Gru15]{Gru15}
	\newblock G. Grubb,
	\newblock {\em Fractional Laplacians on domains, a development of H\"ormander's theory of $\mu$-transmission pseudodifferential operators},
	\newblock Adv. Math. 268 (2015), 478-528.
	
\bibitem[Gru18]{Gru18}
	\newblock G. Grubb,
	\newblock {\em Green's formula and a Dirichlet-to-Neumann operator for fractiona-order pseudodifferential operators},
	\newblock Comm. Partial Differential Equations 43 (2018), 750-789.

\bibitem[Gru19]{Gru19}
	\newblock G. Grubb,
	\newblock {\em Exact Green's formula for the fractional Laplacian and perturbations},
	\newblock Math. Scand. 126 (2020), 568-592.

\bibitem[JS15]{JS15}
	\newblock D. Jerison, O. Savin,
	\newblock {\em Some remarks on stability of cones for the one-phase free boundary problem},
	\newblock Geom. Funct. Anal. 25 (2015), 1240-1257.

\bibitem[JN17]{JN17}
	\newblock Y. Jhaveri, R. Neumayer,
	\newblock {\em Higher regularity of the free boundary in the obstacle problem for the fractional Laplacian},
	\newblock Adv. Math. 311 (2017), 748-795.
	
	
	
\bibitem[LWW17]{LWW16} Y. Liu, K. Wang, J. Wei, \emph{Global minimizers of the Allen-Cahn equation in dimension $n = 8$}, J. Math. Pures Appl. 108 (2017), 818-840.

\bibitem[LWW21]{LWW19} Y. Liu, K. Wang, J. Wei, \emph{On smooth solutions to one-phase free boundary problem in $\mathbb\R^{n}$}, Int. Math. Res. Not. 20 (2021), 15682-15732.
	
\bibitem[PY07]{PY07} A. Petrosyan, N. K. Yip, \emph{Nonuniqueness in a free boundary problem from combustion}, J.~Geom. Anal. 18 (2007), 1098-1126.

\bibitem[RS14]{RS14}
	\newblock X. Ros-Oton, J. Serra,
	\newblock {\em The Dirichlet problem for the fractional Laplacian: regularity up to the boundary},
	\newblock J. Math. Pures Appl. 101 (2014), 275-302.
	
\bibitem[RS17]{RS17}
	\newblock X. Ros-Oton, J. Serra,
	\newblock {\em Boundary regularity estimates for nonlocal elliptic equations in $C^1$ and $C^{1,\alpha}$ domains},
	\newblock Ann. Mat. Pura Appl. 196 (2017), 1637-1668.
	
\bibitem[San18]{San18}
	\newblock T. Sanz-Perela,
	\newblock {\em Regularity of radial stable solutions to semilinear elliptic equations for the fractional Laplacian},
	\newblock Commun. Pure Appl. Anal. 17 (2018), 2547-2575.
	
\bibitem[ST10]{ST10}
	\newblock P. Stinga, J. L. Torrea,
	\newblock {\em Extension problem and Harnack's inequality for some fractional operators},
	\newblock Comm. Partial Differential Equations 35 (2010), 2092-2122.
	
\bibitem[TTV18]{TTV18}
	\newblock S. Terracini, G. Tortone, S. Vita,
	\newblock {\em On s-harmonic functions on cones},
	\newblock 	Anal. PDE 11 (2018), 1653-1691.


\bibitem[We03]{We03} G. S. Weiss, \emph{A singular limit arising in combustion theory: fine properties of the free boundary}, Calc. Var. PDE 17 (2003), 311-340.

\end{thebibliography}
\end{document}